\documentclass{article}
\usepackage[lang = british]{ems-jfg} %% change to `american' if you use American English

\usepackage[margin=0pt]{subfig}
\usepackage{float}
\usepackage{mathtools}
\usepackage{amsfonts}
\usepackage{empheq}
\usepackage{mathrsfs}

\theoremstyle{plain}

\newtheorem{theorem}{Theorem}
\newtheorem{lemma}{Lemma}
\newtheorem{definition}{Definition}
\newtheorem{corollary}{Corollary}

\counterwithin{subsection}{section}
\counterwithin{equation}{section}
\counterwithin{theorem}{section}
\counterwithin{lemma}{section}
\counterwithin{definition}{section}
\counterwithin{figure}{section}
\counterwithin{corollary}{section}

\numberwithin{equation}{section}
\numberwithin{table}{section}
\numberwithin{figure}{section}
\DeclareMathOperator{\re}{Re}
\DeclareMathOperator{\diam}{diam}
\newcommand{\vaddspace}[1]{\rule{0pt}{#1pt}}
\newcommand{\newparagraph}{\vaddspace{10}}

\makeatletter
\def\oversortoftilde#1{\mathop{\vbox{\m@th\ialign{##\crcr\noalign{\kern3\p@}%
      \sortoftildefill\crcr\noalign{\kern3\p@\nointerlineskip}%
      $\hfil\displaystyle{#1}\hfil$\crcr}}}\limits}

\def\sortoftildefill{$\m@th \setbox\z@\hbox{$\braceld$}%
  \braceld\leaders\vrule \@height\ht\z@ \@depth\z@\hfill\braceru$}
\makeatother

\allowdisplaybreaks

\begin{document}

%------
% Insert the title of your paper and (if necessary)
% a short title for the running head.
%------
\title{The Complex Dimensions of Every Sierpi\'nski Carpet Modification of Dust Type}
\titlemark{DUST CARPETS}

%------
% Insert full names of the authors.
% Add further authors as follows:
%  \emsauthor{2}{}{}
%  \emsauthor{3}{}{}
% etc.
% Abbreviate first names for the running head.
%------
\emsauthor{1}{Jade Leathrum}{}
%\emsauthor{2}{Elijah Guptill}{}
%\emsauthor{3}{Daniel Sebo}{}
%\emsauthor{4}{Sean Watson}{}
%\emsauthor{5}{Erin Pearse}{}
%\emsauthor{6}{Michel Lapidus}{}

%------
% Use \authormark if the list of authors is too
% long for the running head: \authormark{A.~Doe et al.}
%------

%------
% Add one \emsaffil and one \email for each author.
%------
\emsaffil{1}{Department of Mathematics, University of California Riverside, 900 University Avenue, Riverside, CA 92521-0135, USA \email{jade.leathrum@ucr.edu}}

%------
% Add MSC 2020 codes according to www.ams.org/msc/msc2020.html.
% Secondary codes (in square brackets) are optional.
%------
\classification[28A75, 11B37]{28A80}

%------
% Add a list of keywords.
%------

%------
% Insert your abstract.
%------
\begin{abstract}
We investigate modified Sierpi\'nski Carpet fractals, constructed by dividing a square into a square $n \times n$ grid, removing a subset of the squares at each step, and then repeating that process for each square remaining in that grid. If enough squares are removed and in the proper places, we produce ``dust type'' carpets, which have a path-connected complement and are themselves not path-connected. We study these fractals using the fractal zeta functions, first introduced by Michel Lapidus, Goran Radunovi\'c, and Darko \v{Z}ubrini\'c in their book \emph{Fractal Zeta Functions and Fractal Drums}, from which we devised an analytical and combinatorial algorithm to compute the complex dimensions of every Sierpi\'nski Carpet modification of dust type.
\end{abstract}

\maketitle

%------
% INSERT THE BODY OF THE PAPER HERE (except
% acknowledgments, funding info and bibliography)
%------

\section{Introduction}

A classic example of a geometric fractal is known as the ``Sierpi\'nski Carpet''. It is constructed by starting with a unit square in $\mathbb{R}^2$, dividing it into a uniform $3\times3$ grid, and removing the middle square. Each square remaining from the previous step is then divided into its own $3\times3$ grid and the middles of each of those grids are also removed. Continuing this process indefinitely yields Figure \ref{fig:normalsierpinskicarpet}.
\begin{figure}[H]
\includegraphics[scale=0.15]{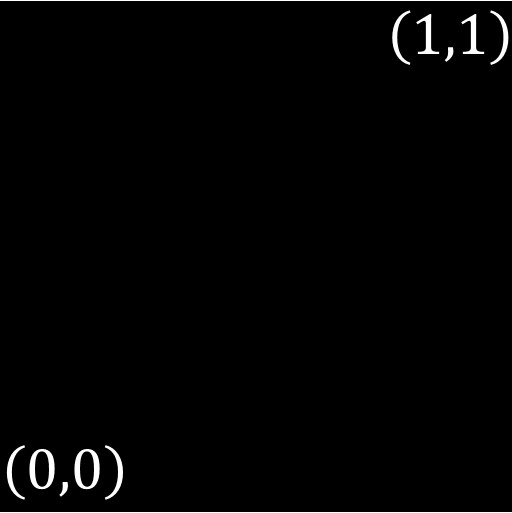}
\includegraphics[scale=0.15]{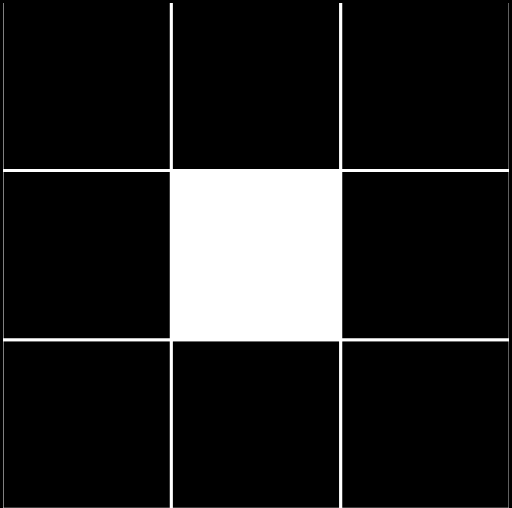}
\includegraphics[scale=0.15]{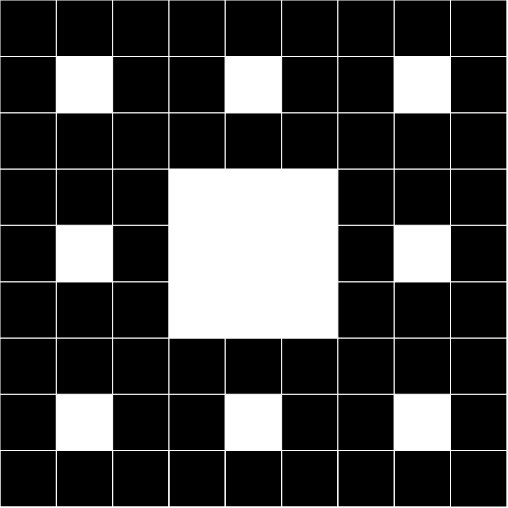}
\includegraphics[scale=0.15]{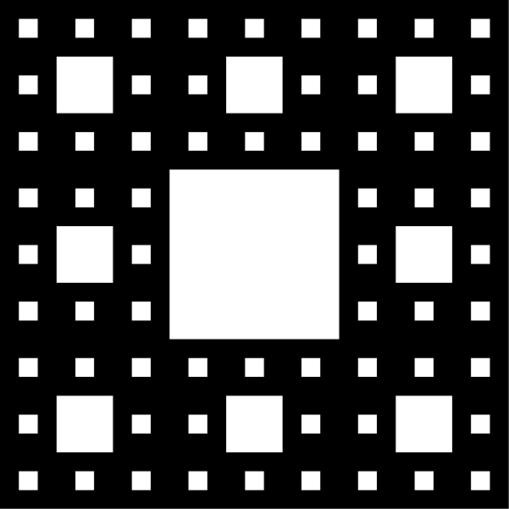}
\includegraphics[scale=0.15]{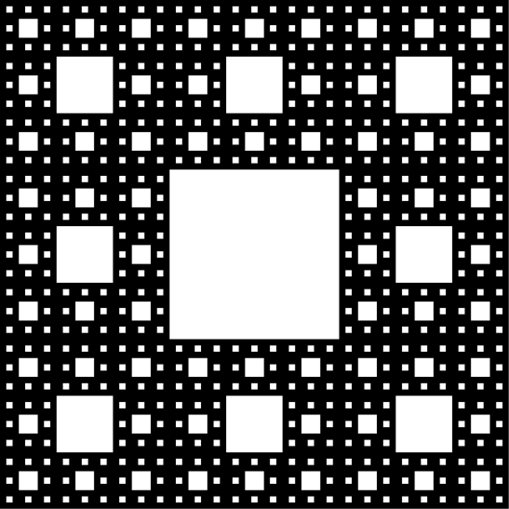}
\includegraphics[scale=0.15]{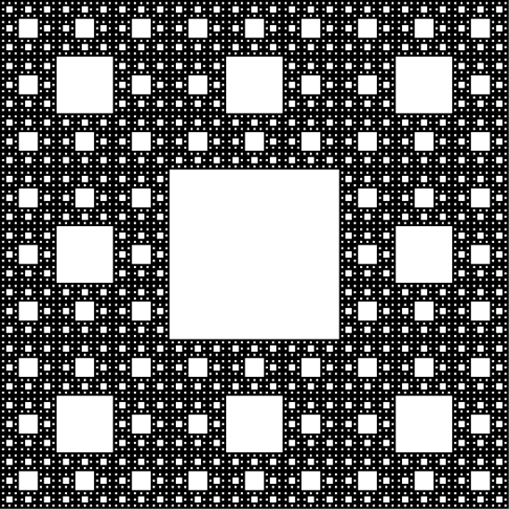}
\caption{The first 5 steps of the construction of the classic Sierpi\'nski Carpet.}
\label{fig:normalsierpinskicarpet}
\end{figure}

The classic Sierpi\'nski Carpet has numerous interesting geometric, topological, and measure-theoretic properties (\cite{sierpinskicarpetproperties}). However, we do not have to limit ourselves to only removing the middle square, and we do not have to limit ourselves to a $3\times3$ grid, as demonstrated in Figure \ref{fig:sierpinskicarpetmodificationintro}.
\begin{figure}[H]
\includegraphics[trim={0 0 0 1.5cm},clip=true,scale=0.15]{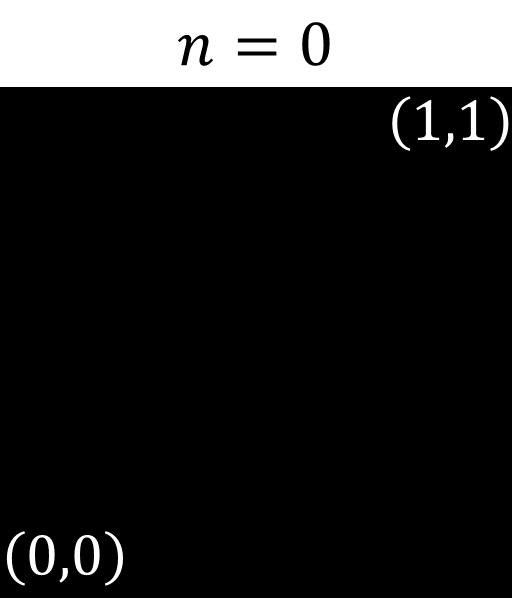}
\includegraphics[trim={0 0 0 1.5cm},clip=true,scale=0.15]{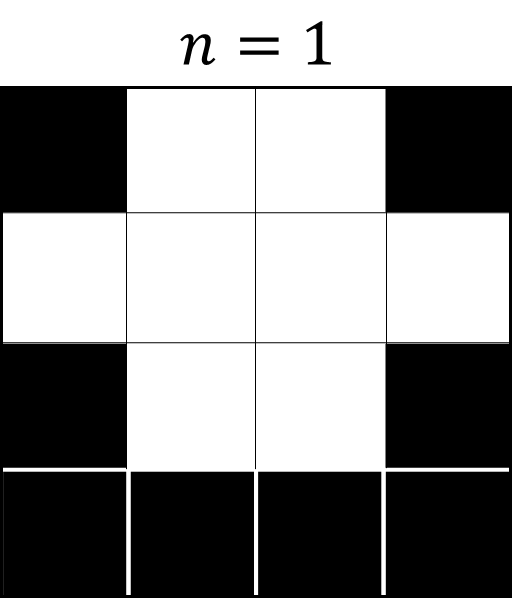}
\includegraphics[trim={0 0 0 1.5cm},clip=true,scale=0.15]{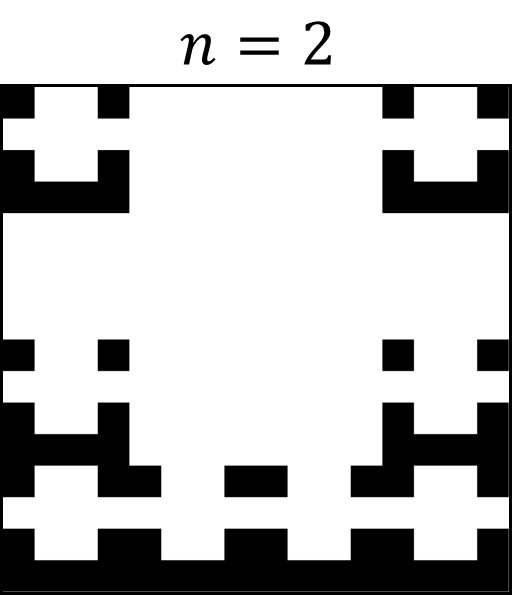}
\includegraphics[trim={0 0 0 1.5cm},clip=true,scale=0.15]{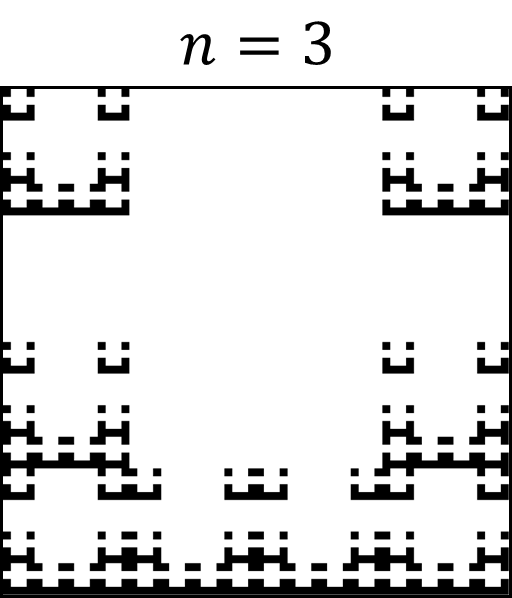}
\includegraphics[trim={0 0 0 1.5cm},clip=true,scale=0.15]{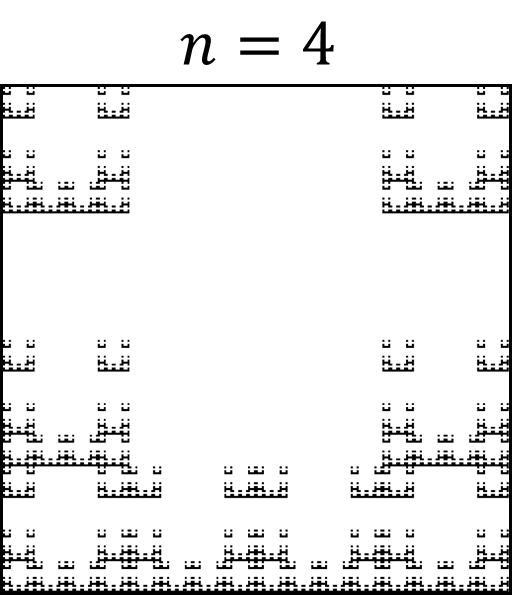}
\includegraphics[trim={0 0 0 1.5cm},clip=true,scale=0.15]{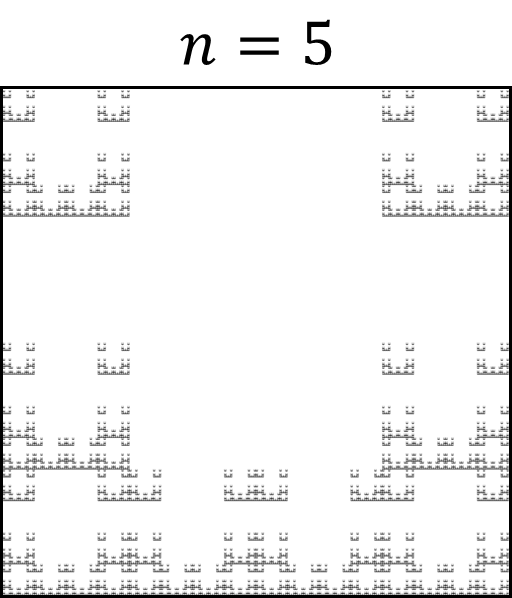}
\caption{The first 5 steps of the construction of a modified Sierpi\'nski Carpet utilizing a $4\times4$ grid.}
\label{fig:sierpinskicarpetmodificationintro}
\end{figure}

If enough squares are removed and in specific locations on the grid, we produce Sierpi\'nski Carpet modifications which look visually similar to the Cantor Dust. We call such carpets ``dust type''. Interestingly, dust type Sierpi\'nski Carpet modifications are extremely common, making up approximately 60-80\% of all the modifications that can be constructed on a fixed grid size larger than $3\times3$ (unique up to rotation and reflection). \\

In this paper, our goal is to study dust type Sierpi\'nski Carpet modifications by computing a particular ``fractal zeta function'' introduced in \cite{lapidusbook}. These complex-valued zeta functions use analysis methods to reveal deep geometric properties of the original fractal. The one we use in this paper is the ``tubular zeta function'', which has the Minkowski dimension as well as a set of ``complex dimensions'' as its poles. The meaning of these complex dimensions will be explained briefly in Section \ref{chap:introtubularzeta}, and is discussed in more depth in \cite{lapidusbook}, \cite{lapidusbook2}, and \cite{lapidusbook3}. However, direct computations of the tubular zeta function are often difficult and result in non-elementary integrals. A central theme of this paper is the ability to reduce the computations from analysis to combinatorics. In the process, we improve on some results from the literature and give confirmations of other results. \\

This paper is organized as follows:
\begin{enumerate}
    \item \textbf{Introduction} The meaning of complex-valued dimensions (Section \ref{chap:introtubularzeta}) and definitions of terms (Section \ref{chap:dustdefinitions}). Includes a worked example of the Ternary Cantor Set (Section \ref{chap:ternarycantorsetexample}) and a brief preview of the results (Section \ref{chap:previewresults}).
    \item \textbf{Helper Theorems} Scaling theorems (Section \ref{chap:scalingtheorem}), holomorphicity theorems (Section \ref{chap:entireextensiontheorem}), and existence of ``thresholds'' which are central to the results that follow (Section \ref{chap:thresholdtheorem}).
    \item \textbf{Completely Dusty Case} The complex dimensions for the simplest dust type Sierpi\'nski Carpet modifications. Includes a worked example of the Cantor Dust (Section \ref{chap:cantordustexample}).
    \item \textbf{General Case} Explaining the geometry and combinatorics of the general dust type Sierpi\'nski Carpet modifications (Sections \ref{chap:intersectiontypes} and \ref{chap:countingintersections}), culminating in Theorem \ref{thm:gammaconnected2} (Section \ref{chap:theorem42}). Includes a worked example of a modified Cantor Grill (Section \ref{chap:cantorstripesathomeexample}).
    \item \textbf{Future Study} Discussion of polyhedral zeta functions (Section \ref{chap:polyhedralzetafunctions}), non-dust type Sierpi\'nski Carpet modifications (Section \ref{chap:nondustcarpets}), Bedford--McMullen Carpets (Section \ref{chap:bedfordmcmullencarpets}), and an extension of some results from $\mathbb{R}^2$ to $\mathbb{R}^N$ for general $N \in \mathbb{N}$ (Section \ref{chap:mengersponges}).
\end{enumerate}

%For each major section, have similar "introduce the topic" like done with the presentation!

%Talk about Sierpinski Carpets - what are they, what are they for?

\subsection{The Tubular Zeta Function} \label{chap:introtubularzeta}
We will study the geometry of nonempty compact subsets $A \subseteq \mathbb{R}^N$ for $N \in \mathbb{N}$ via the tubular zeta function, defined in \cite{lapidusbook}:
\begin{equation}
\tilde{\zeta}_A(s,\delta) = \int_0^{\delta} t^{s-N-1} |A_t| \,\mathrm{d}t, \label{eq:tubularzetafunctiondef}
\end{equation}
where $A_t = \{y \in \mathbb{R}^N : d(x,y) < t\}$, $d$ is the Euclidean metric in $\mathbb{R}^N$, and $|\cdot|$ is Lebesgue measure in $\mathbb{R}^N$. The function $|A_t|$ is commonly called the ``tube function''. Using the tube function, we can define the following properties of $\frac{|A_t|}{t^{N-d}}$ for some $d > 0$:
\begin{align}
\mathscr{M}_d^* &= \limsup_{t\rightarrow0^+} \frac{|A_t|}{t^{N-d}} \\
\mathscr{M}_{*,d} &= \liminf_{t\rightarrow0^+} \frac{|A_t|}{t^{N-d}}.
\end{align}
We can now define the terms in Table \ref{table:minkowskithings}. \\
\begin{table}[!h]
\begin{center}
\begin{tabular}{|c|c|c|}
\hline
\textbf{Name} & \textbf{Mathematical Definition} \\
\hline
\hline
Upper Minkowski dimension $\overline{D}$ & $\overline{D} = \inf\{d > 0 : \mathscr{M}_d^* = 0\}$ \\
\hline
Lower Minkowski dimension $\underline{D}$ & $\underline{D} = \inf\{d > 0 : \mathscr{M}_{*,d} = 0\}$ \\
\hline
Upper Minkowski content $\mathscr{M}^*$ & $\mathscr{M}^* = \mathscr{M}_{\overline{D}}^*$ \\
\hline
Lower Minkowski content $\mathscr{M}_*$ & $\mathscr{M}_* = \mathscr{M}_{*,\underline{D}}$ \\
\hline
\end{tabular}
\caption{Definitions for Minkowski dimensions and contents.}
\label{table:minkowskithings}
\end{center}
\end{table} \\

All the examples studied in this paper will be homogeneous self-similar fractals which satisfy the open set condition, so the upper and lower Minkowski dimensions will be equal (and hence we will refer to $D = \overline{D} = \underline{D}$ as simply the ``Minkowski dimension''), the Minkowski dimension $D$ will equal the box-counting and Hausdorff dimensions, and $0 < \mathscr{M}_* < \mathscr{M}^* < \infty$. In the theory of fractal strings, it is known that in the case where $A \subseteq \mathbb{R}^1$ is the boundary of a geometric realization of a lattice self-similar string, $\frac{|A_t|}{t^{1-D}}$ oscillates between the upper and lower Minkowski contents in a multiplicatively-periodic manner in $\ln(t)$ by \cite[Theorem 8.23]{lapidusbook3}. As stated in \cite[Problem 6.2.35]{lapidusbook}, it is conjectured that the same oscillation behavior will be present in $\frac{|A_t|}{t^{N-D}}$ for lattice self-similar sets $A \subseteq \mathbb{R}^N$ for $N \in \mathbb{N}$ which satisfy the open set condition. Examples for which the conjecture is known to hold include the classic Sierpi\'nski Carpet (\cite[Example 2.3.36]{lapidusbook}) and Sierpi\'nski Gasket (\cite[Equation 1.1.28]{lapidusbook}), and further details are given in \cite{lapidusbook}, \cite{lapidusbook2}, and \cite{lapidusbook3}. The asymptotics of the tube function lead to considering the following modified Taylor Series:
\begin{equation}
|A_t| = \sum_{d \in \mathscr{P}(A)} c_d t^{N-d},
\end{equation}
where $\mathscr{P}(A) \subseteq \mathbb{C}$ is a multiset, and where the largest real part of an element of $\mathscr{P}(A)$ will be exactly $D$. According to the theory of complex dimensions, oscillatory behavior of $\frac{|A_t|}{t^{N-D}}$ is encoded in the non-real complex elements of $\mathscr{P}(A)$. Since our examples are bounded subsets of $\mathbb{R}^N$, $\mathscr{P}(A)$ will be a discrete countable multiset where each distinct element has finite multiplicity (\cite[p. 8]{lapidusbook}). \\

The aforementioned $\tilde{\zeta}_A(s,\delta)$ from Equation (\ref{eq:tubularzetafunctiondef}) falls under a class of ``fractal zeta functions'', with the tubular zeta function being a truncated Mellin transform of $|A_t|$. The Mellin transform of $t^d$ is $\frac{1}{s+d}$ under suitable conditions on $t$ and $s$, so since the tubular zeta function is a shifted and truncated Mellin transform of $|A_t|$, we can compute the multiset $\mathscr{P}(A)$ by locating the poles of the tubular zeta function. Thus, it is valuable to find the poles of the tubular zeta function, as it translates to information about the tube function, which then gives helpful information about the geometry of the set $A$. Furthermore, by \cite[Theorem 2.2.1]{lapidusbook}, we know that the tubular zeta function is holomorphic on the half-plane $\{s \in \mathbb{C} : \re(s) > D\}$, and we will be taking meromorphic continuations in order to locate all the other poles. This meromorphic continuation, when it exists, will necessarily be unique by the principle of analytic continuation. \\
%In the examples we study, we will have by \cite[Theorem 4.7.2]{lapidusbook} that $\tilde{\zeta}_A$ has a meromorphic continuation to all of $\mathbb{C}$, which will necessarily be unique by the principle of analytic continuation. \\

For computation, as a starting example, consider the unit square
\begin{equation}
A = [0,1] \times [0,1] \subseteq \mathbb{R}^2.
\end{equation}
When we apply the tubular zeta function to the unit square, we are led to the following computation:
\begin{align}
\tilde{\zeta}_A(s,\delta) &= \int_0^{\delta} t^{s-3} (1+4t+\pi t^2) \,\mathrm{d}t \nonumber \\
&= \frac{\delta^{s-2}}{s-2} + \frac{4\delta^{s-1}}{s-1} + \frac{\pi\delta^s}{s}.
\end{align}
We can use the principle of analytic continuation to produce a meromorphic continuation of the resulting function to all of $\mathbb{C}$, thus allowing us to effectively ignore the half-plane of convergence for the integral. Using this continuation, we conclude that the tubular zeta function on the unit square has poles at $s = 0,1,2$. We will denote this as follows:
\begin{equation}
\mathscr{P}(A) = \{0,1,2\}.
\end{equation}
This corresponds to the intuitive notion that the unit square is geometrically constructed out of components of Minkowski dimension 0, 1, and 2 (points, lines, and 2-dimensional faces, respectively). We now want to find the poles of the tubular zeta function when we let the set $A$ be a fractal.
\newparagraph

\subsection{Modified Sierpi\'nski Carpets} \label{chap:dustdefinitions}
%Define what a "Sierpinski Carpet Modification" is, define what "Dust Type" means
In this paper, we will use $\mathbb{N}$ to denote the set of strictly positive integers. Any occurrences of ``$0^0$'' that may appear in formulas will be treated as $\displaystyle \lim_{x\rightarrow 0^+} x^x = 1$, for the sake of simplicity in the definitions. \\

\begin{definition}[Sierpi\'nski Carpet Modification] \label{def:sierpinskicarpetmodification}
Let $p \in \mathbb{N}$ be such that $p \geq 2$. We construct a \textit{Sierpi\'nski Carpet modification} by the following algorithm:
\begin{enumerate}
    \item Begin with the unit square $A_0 = [0,1] \times [0,1] \subseteq \mathbb{R}^2$.
    \item Dissect the square into a uniform $p \times p$ grid, producing nonoverlapping squares of side length $\frac{1}{p}$.
    \item (Produce Level 1 of the fractal construction) Choose a collection of $m \in \mathbb{N}$ grid squares to keep, removing all the others. We will call the result $A_1$. We require $m \in [1,p^2-1] \cap \mathbb{N}$ in order to ensure $A_1 \neq \emptyset$ and $A_1 \neq A_0$.
    \item (Produce Level $n$ of the fractal construction from Level $n-1$) For each square that remains in $A_{n-1}$ (Level $n-1$ of the fractal construction), divide it into its own $p \times p$ grid (making nonoverlapping squares of side length $\frac{1}{p^n}$) and remove the corresponding collection of squares from each grid. We will call the result $A_n$. \label{recursionstep}
    \item Repeat Step \ref{recursionstep} infinitely-many times to produce $\displaystyle A = \bigcap_{n=1}^{\infty} A_n$.
\end{enumerate}
\end{definition}

\begin{definition}[Dust Type] \label{def:dusttype}
Let $A$ be a Sierpi\'nski Carpet modification constructed from a $p \times p$ grid where at least 2 squares are kept at each step. $A$ is of ``dust type'' if $A^C$ is path-connected but $A$ itself is not.
\end{definition}

See Figure \ref{fig:sierpinskicarpetmodification} for a specific example.

\begin{figure}[H]
\includegraphics[scale=0.3]{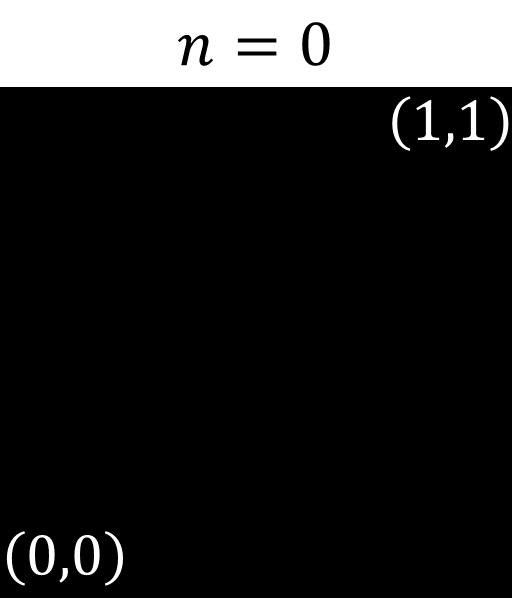}
\includegraphics[scale=0.3]{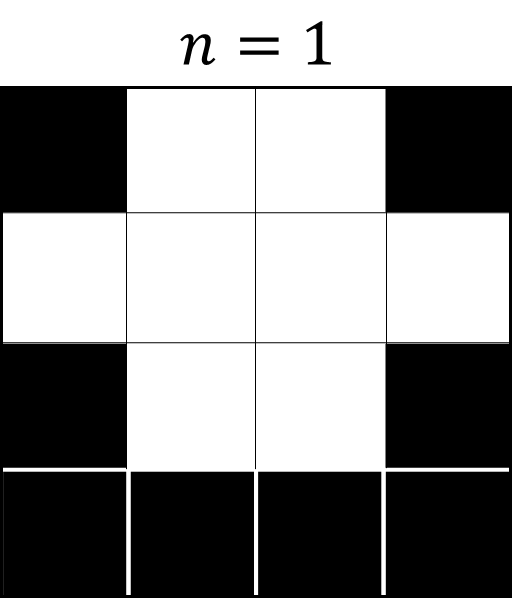}
\includegraphics[scale=0.3]{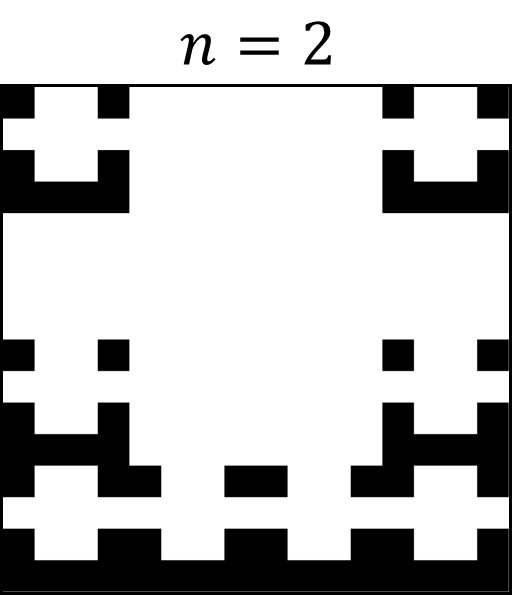}\\
\includegraphics[scale=0.3]{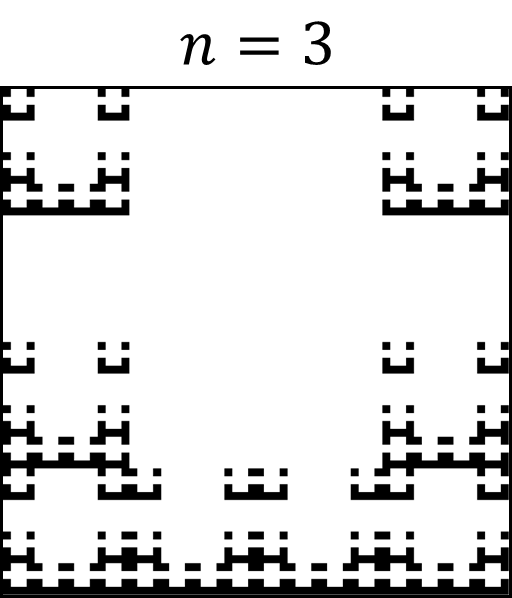}
\includegraphics[scale=0.3]{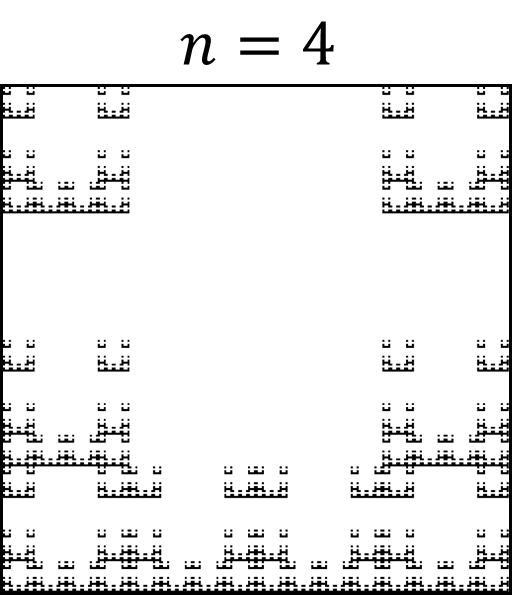}
\includegraphics[scale=0.3]{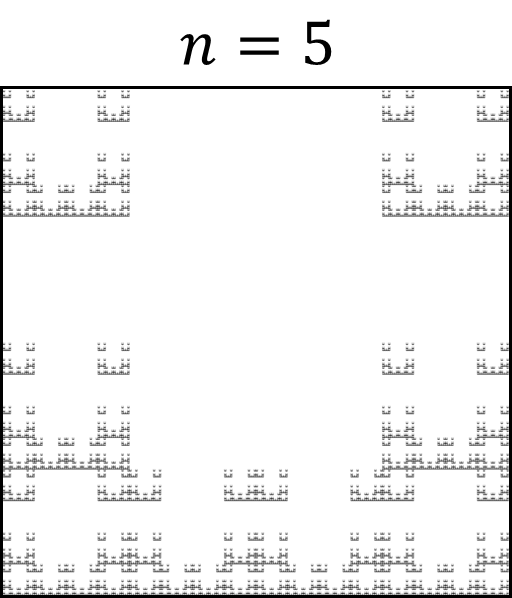}\\
\caption{The first 5 steps of the construction of an example Sierpi\'nski Carpet modification of dust type in the $4\times4$ grid.}
\label{fig:sierpinskicarpetmodification}
\end{figure}

If a Sierpi\'nski Carpet modification in the $p \times p$ grid has $m = 1$ square kept in the construction process, then as discussed in \cite[Chapter 3.1]{mastersthesis}, the resulting set will simply be a translation of $\{(0,0)\}$ and thus will not have fractal-like properties. Thus, we will only consider cases where $m \in [2, p^2-1]$. Furthermore, Figure \ref{fig:all2x2} displays all the Sierpi\'nski Carpet modifications that can be generated from the $2\times2$ grid with $m = 2,3$ (unique up to rotation and reflection). \\
\begin{figure}[H]
\includegraphics[scale=0.2]{Images/NEquals0}
\includegraphics[scale=0.2]{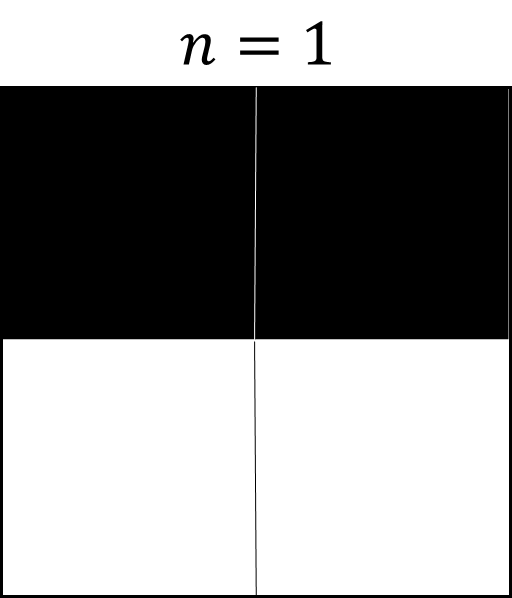}
\includegraphics[scale=0.2]{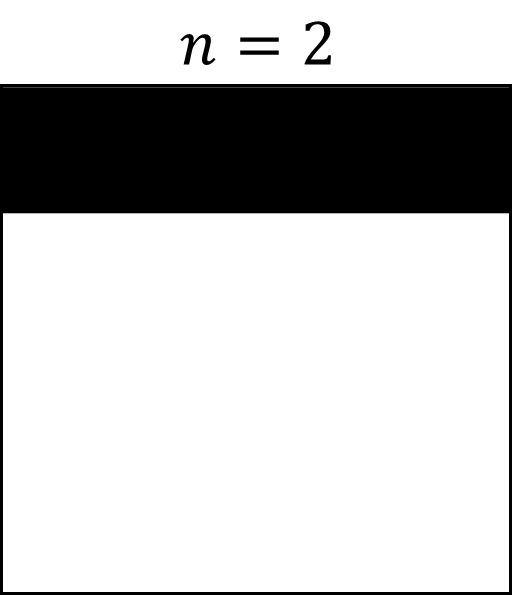}
\includegraphics[scale=0.2]{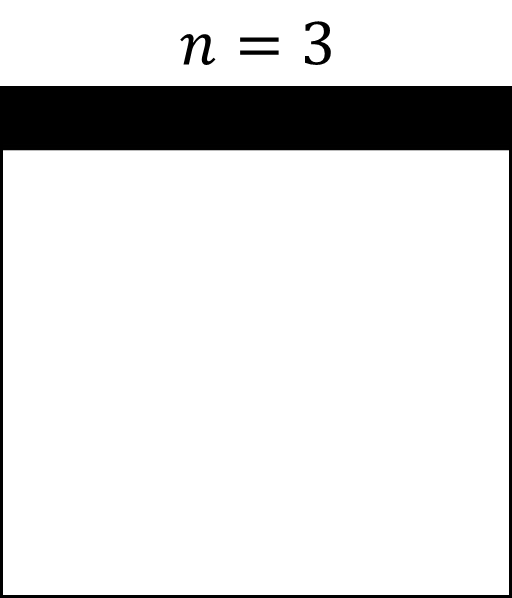}
\includegraphics[scale=0.2]{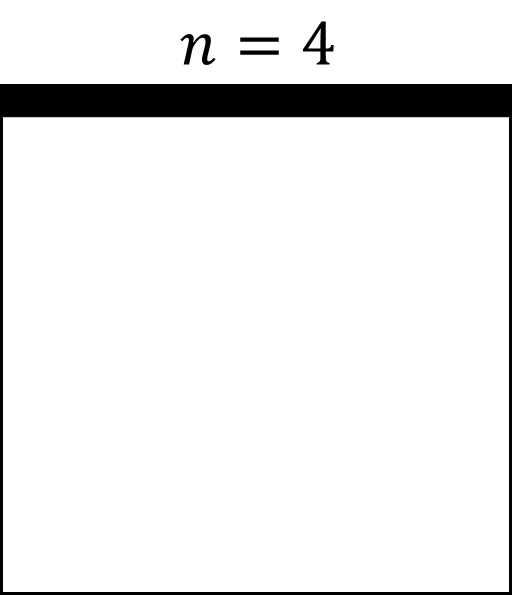}\\
\includegraphics[scale=0.2]{Images/NEquals0}
\includegraphics[scale=0.2]{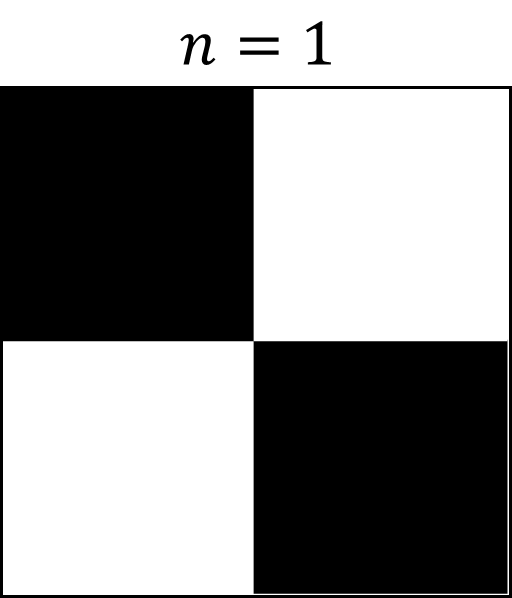}
\includegraphics[scale=0.2]{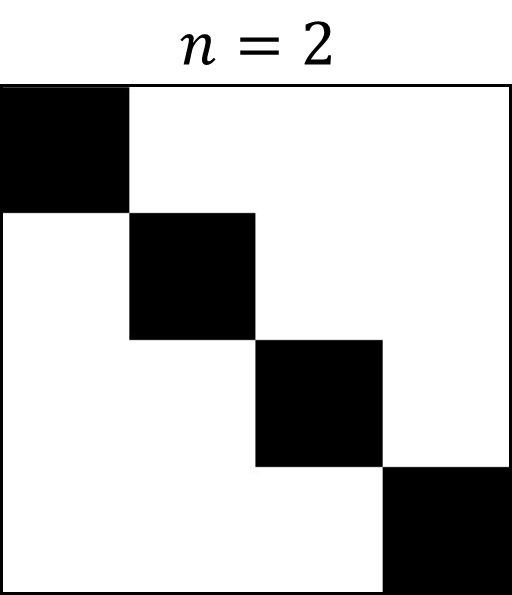}
\includegraphics[scale=0.2]{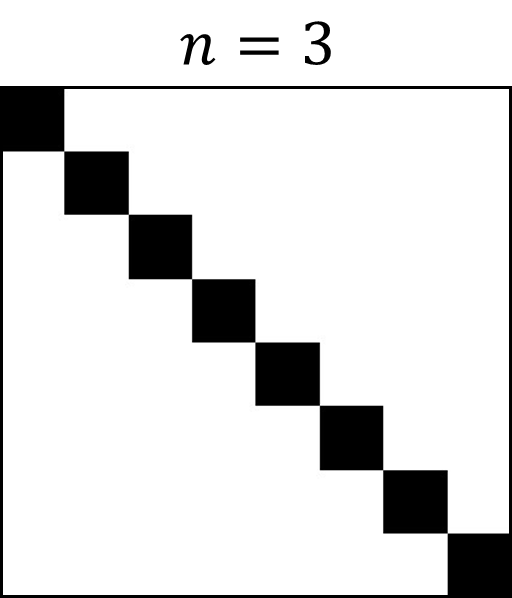}
\includegraphics[scale=0.2]{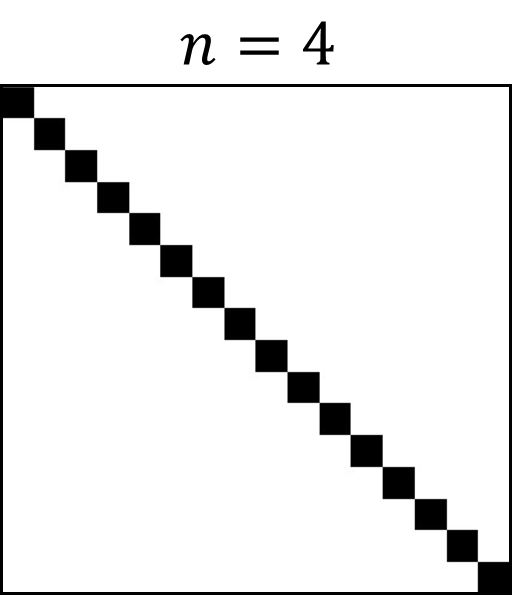}\\
\includegraphics[scale=0.2]{Images/NEquals0}
\includegraphics[scale=0.2]{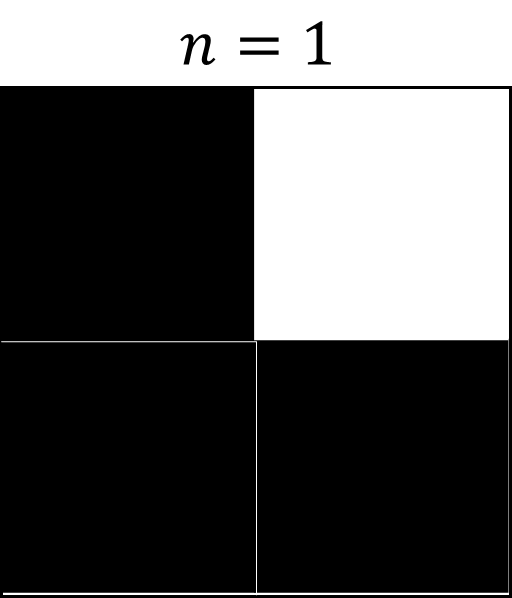}
\includegraphics[scale=0.2]{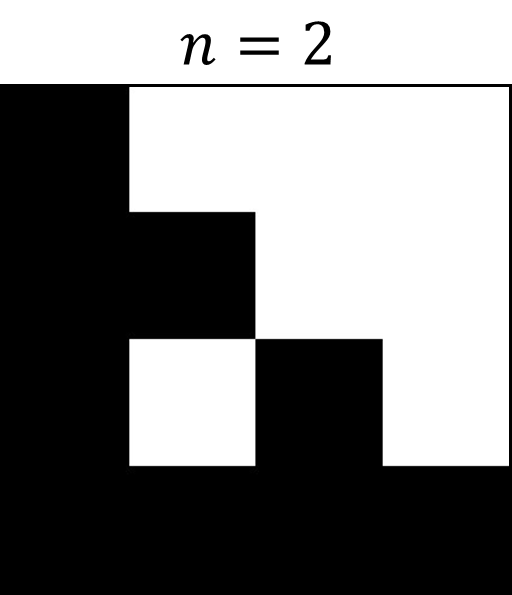}
\includegraphics[scale=0.2]{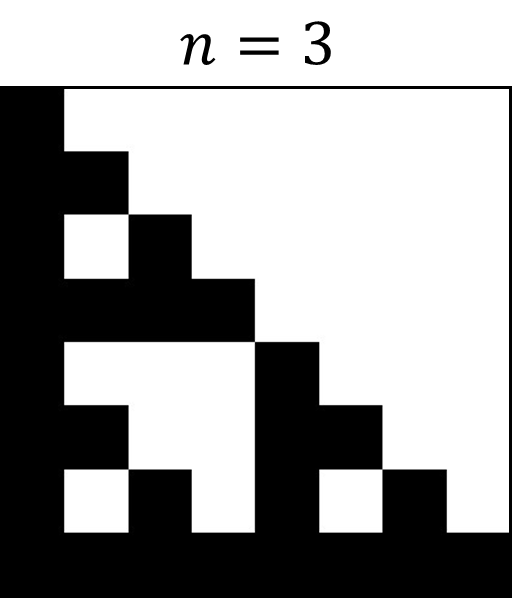}
\includegraphics[scale=0.2]{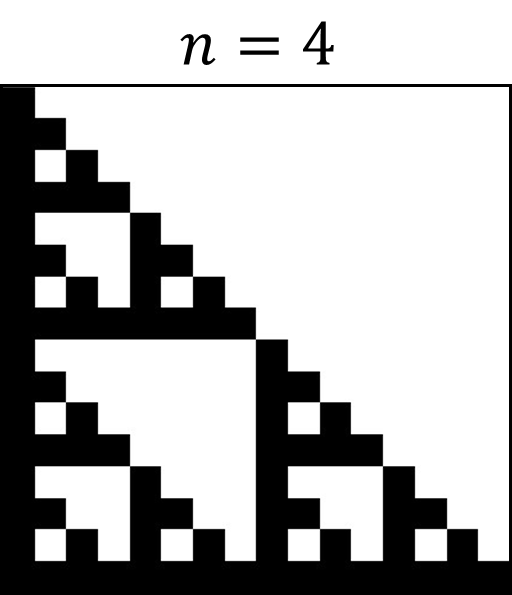}
\caption{The first 4 steps of the construction of all possible Sierpi\'nski Carpet modifications in the $2\times2$ grid for $m = 2,3$.}
\label{fig:all2x2}
\end{figure}
Figure \ref{fig:all2x2} shows that none of the Sierpi\'nski Carpets modifications on the $2\times2$ grid will be of dust type. Thus, we will only consider the Sierpi\'nski Carpet modifications on the $p \times p$ grid where $p \geq 3$ and $m \in [2, p^2-1]$.
\newparagraph

\subsection{Motivating Example: Ternary Cantor Set} \label{chap:ternarycantorsetexample}
%Go through Cantor Set computation to explain why we DON'T want to do that
We will use the following computation to demonstrate why we wish to avoid manual computation of the area function. Consider a Sierpi\'nski Carpet modification which produces the Ternary Cantor Set, given in Figure \ref{fig:cantorset}.
\begin{figure}[H]
\includegraphics[scale=0.3]{Images/NEquals0}
\includegraphics[scale=0.3]{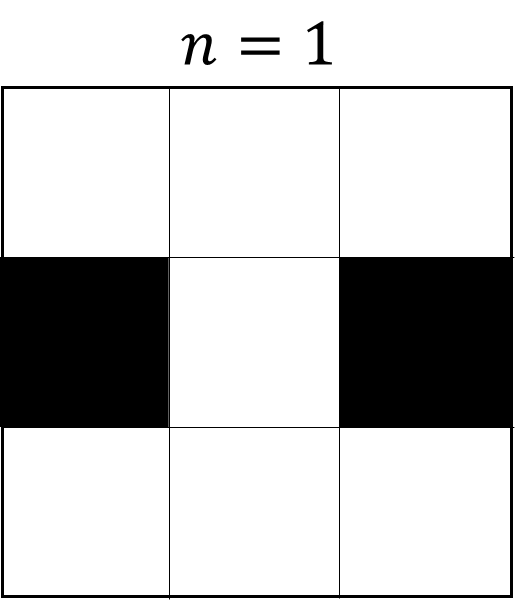}
\includegraphics[scale=0.3]{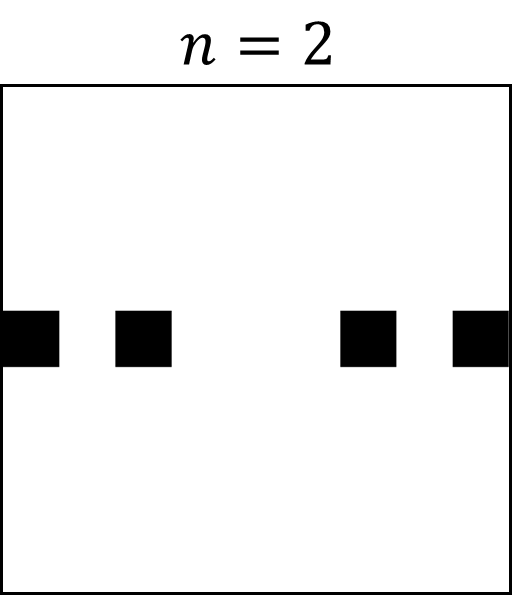}\\
\includegraphics[scale=0.3]{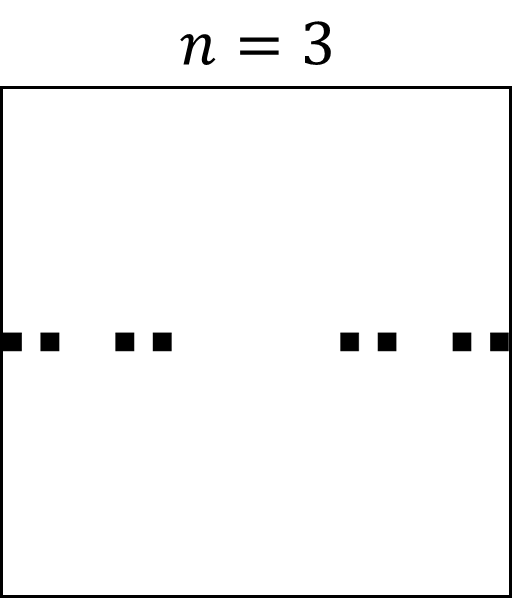}
\includegraphics[scale=0.3]{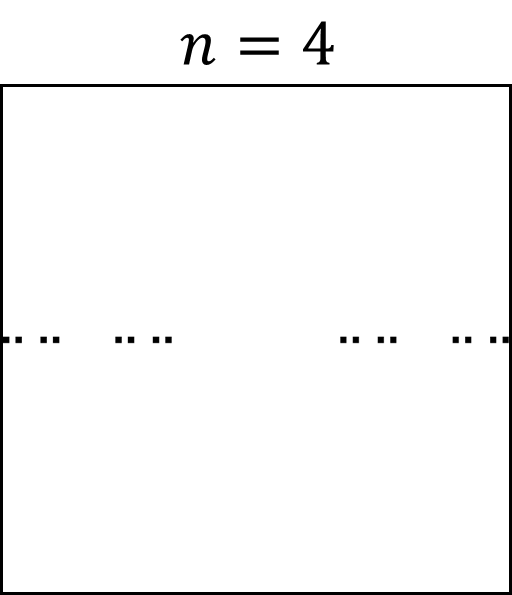}
\includegraphics[scale=0.3]{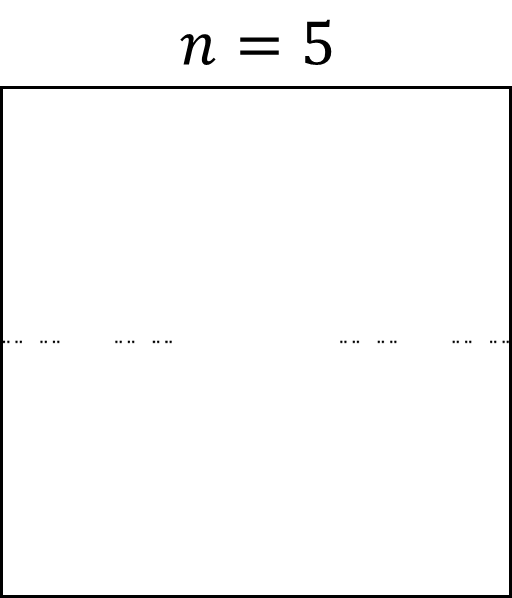}\\
\caption{The first 5 steps of the construction of a Sierpi\'nski Carpet modification which produces the Ternary Cantor Set.}
\label{fig:cantorset}
\end{figure}
We will compute the area of $A_t$ by considering a solid pill-shape area and then subtracting the excess area. The excess areas are shown in Figure \ref{fig:cusparea}.
\begin{figure}[H]
\includegraphics[scale=0.5]{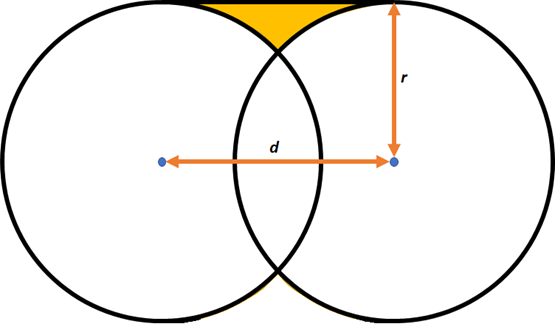}
\caption{Area of cusp between two circles, notated $C(r,d)$.}
\label{fig:cusparea}
\end{figure}
Using standard calculus techniques, we see that when $d \in (0,2r)$, we have the following area:

\begin{align}
C(r,d) &= 2\int_0^{\frac{d}{2}} \left(r - \sqrt{r^2 - \left(x-\frac{d}{2}\right)^2}\right) \,\mathrm{d}x \nonumber\\
&= dr - \frac{1}{4} d \sqrt{4r^2 - d^2} - r^2 \arctan\left(\frac{d}{\sqrt{4r^2 - d^2}}\right) \label{eq:cusparea}.
\end{align}
As we will see later, we can assume for simplicity that $\delta > \frac{1}{6}$. When $t \in \left(\frac{1}{6}, \delta\right]$, we obtain the following area formula:
\begin{equation}
|A_t| = 2t + \pi t^2 - \sum_{k=1}^{\infty} 2^k C\left(\frac{1}{3^k}, t\right).
\end{equation}
When $t \in \left(\frac{1}{2\cdot3^{n+1}}, \frac{1}{2\cdot3^n}\right]$ for some $n \in \mathbb{N}$, we obtain:
\begin{equation}
|A_t| = 2^n (2t + \pi t^2) - 2^n \sum_{k=1}^{\infty} 2^k C\left(\frac{1}{3^{k+n}}, t\right).
\end{equation}
Using properties of Lebesgue integrals, we can dissect the integral over $[0,\delta]$ into a sum of integrals across a countable number of disjoint subintervals:
\begin{align}
\tilde{\zeta}_A(s,\delta) &= \int_0^{\delta} t^{s-3} |A_t| \,\mathrm{d}t \nonumber\\
&= \int_{\frac{1}{6}}^{\delta} t^{s-3} \left(2t + \pi t^2 - \sum_{k=1}^{\infty} 2^k C\left(\frac{1}{3^k}, t\right)\right) \,\mathrm{d}t \label{eq:cantorsetbigt}\\
&+ \sum_{n=1}^{\infty} \int_{\frac{1}{2\cdot3^{n+1}}}^{\frac{1}{2\cdot3^n}} t^{s-3} \left(2^n (2t + \pi t^2) - 2^n \sum_{k=1}^{\infty} 2^k C\left(\frac{1}{3^{k+n}}, t\right)\right) \,\mathrm{d}t. \label{eq:cantorsetsumoft}
\end{align}
Because of the complexity of the $C(r,d)$ function, the integrals in Equations (\ref{eq:cantorsetbigt}) and (\ref{eq:cantorsetsumoft}) cannot be evaluated with traditional methods. However, \cite[Example 2.3.31]{lapidusbook} and \cite[Sections 1.1 and 1.2]{lapidusbook3} discusses the Ternary Cantor Set $\mathfrak{C}$ embedded in $\mathbb{R}^1$ and computes precisely that
\begin{equation}
\mathscr{P}(\mathfrak{C}) = \left\{\log_3(2) + \frac{2\pi i j}{\ln(3)} : j \in \mathbb{Z}\right\}.
\end{equation}
Furthermore, \cite[Theorem 4.7.3]{lapidusbook} states that the poles are invariant under choice of $\mathbb{R}^N$ to embed the set into, as long as $N \geq \dim(A) = \log_3(2)$. Therefore, since the construction of the set $A$ produces exactly the Ternary Cantor Set embedded in $\mathbb{R}^2$, we have the following:
\begin{equation}
\mathscr{P}(A) = \left\{\log_3(2) + \frac{2\pi i j}{\ln(3)} : j \in \mathbb{Z}\right\}.
\end{equation}
\newparagraph

%Result is Lemma 2.2.31 (page 136) of FZF: if increase N, don't change the poles
    %So if you figure out the Cantor Set dimensions on R^1, you figure it out for R^2 as well
    %Alternatively, Cantor Stripes is an extrusion (aka Fractal Grill) of Cantor Set. Again, figure out Cantor Set on R^1, then use Lemma 2.2.31 to get the poles for Cantor Stripes
%Also: Theorem 4.7.3 (page 394) has an exact formula for when A is a bounded set, gives an exact integral formula to convert from the tubular zeta for A to the tubular zeta for Ax{0}

\subsection{Preview of Main Results} \label{chap:previewresults}
We will be studying dust type Sierpi\'nski Carpet modifications in $\mathbb{R}^2$. The methods that we discuss in this paper are designed to circumvent the manual calculation of the area function. However, we will only be able to find a \textit{maximal set} of poles, or a multiset of \textit{possible} complex dimensions, leaving open the possibility for zero-pole cancellation. \\

In the case of $A \subseteq \mathbb{R}^2$ being a dust type Sierpi\'nski Carpet modification constructed from a $p\times p$ grid, we obtain the following:
\begin{equation}
\mathscr{P}(A) \subseteq \left\{\log_p(m) + \frac{2\pi i j}{\ln(p)} : j \in \mathbb{Z}\right\} \cup \bigcup_{r \in R} \left\{\log_p(r) + \frac{2\pi i j}{\ln(p)} : j \in \mathbb{Z}\right\},\label{eq:allpossiblepoles}
\end{equation} 
for some $R \subseteq \{1,2,\dots,p\}$, where $m$ is the number of grid squares kept in the construction of the fractal. All of the poles we obtain will be at worst simple, except some dust type carpets will produce (at worst) poles of order 2 at $s = \log_p(1) + \frac{2\pi i j}{\ln(p)}$ for all $j \in \mathbb{Z}$. We will be able to determine the subset $R$ in Equation (\ref{eq:allpossiblepoles}) entirely from visually counting features present in $A_1, A_2$, and $A_3$, as well as determining the presence of particular features in $A_t$ for a positive-measure set of $t \in [0,\delta]$. Theorems \ref{thm:gammaconnected1} and \ref{thm:gammaconnected2} state these results more precisely with a detailed proof.
\newparagraph

\section{Helper Theorems} \label{chap:helpertheorems}
We begin with establishing key theorems, which together will allow us to reduce the computation of the tubular zeta function from analysis to combinatorics. Theorems \ref{thm:scaling} and \ref{thm:entireextension} and Corollary \ref{thm:anydelta} are well-known in the literature and are not limited to Sierpi\'nski Carpet modifications. Theorems \ref{thm:thresholdexists} and \ref{thm:infinitethresholdsexist} come directly from the fact that the sets of interest are dust type Sierpi\'nski Carpet modifications, helping to explain why the definition of ``dust type'' is sensible.

\subsection{Scaling Theorem} \label{chap:scalingtheorem}
%Can scale uniformly by lambda, just introduce lambda^s
Theorem \ref{thm:scaling} is obtained in \cite[Proposition 2.2.22]{lapidusbook}. For completeness, we provide a statement and proof of this result.

\begin{theorem}[Scaling Theorem] \label{thm:scaling}
If a measurable set $A \subseteq \mathbb{R}^N$ is uniformly scaled by a factor $\lambda > 0$, then for any $\delta > 0$, $\tilde{\zeta}_{\lambda A}(s, \lambda\delta) = \lambda^s \tilde{\zeta}_A(s, \delta)$.
\end{theorem}

\begin{proof}
Let $A \subseteq \mathbb{R}^N$ be Lebesgue-measurable, and let $\lambda > 0$. Then by definition of the Lebesgue measure in $\mathbb{R}^N$, we have that $\displaystyle |(\lambda A)_t| = \lambda^N \left|A_{\frac{t}{\lambda}}\right|$. Plugging this into the tubular zeta function formula yields:
\begin{equation}
\tilde{\zeta}_{\lambda A}(s, \lambda\delta) = \int_0^{\lambda\delta} t^{s-N-1} |(\lambda A)_t| \,\mathrm{d}t
= \int_0^{\lambda\delta} t^{s-N-1} \lambda^N |A_{\frac{t}{\lambda}}| \,\mathrm{d}t.
\end{equation}
We will now substitute $\tau = \frac{t}{\lambda}$:
\begin{equation}
= \lambda \int_0^{\delta} \tau^{s-N-1} \lambda^{s-1} |A_{\tau}| \,d\tau
= \lambda^s \int_0^{\delta} \tau^{s-N-1} |A_{\tau}| \,d\tau
= \lambda^s \tilde{\zeta}_A(s,\delta).
\end{equation}
\end{proof}
\newparagraph

\subsection{Entire Extension} \label{chap:entireextensiontheorem}
Theorem \ref{thm:entireextension} and the proceeding Corollary \ref{thm:anydelta} are stated in \cite[Proposition 2.2.13]{lapidusbook}. For completeness, we provide proofs of the results.

\begin{lemma} \label{thm:entireextensionlemma}
Let $V \subseteq \mathbb{C}$ be an open set, and let $(E, \mathscr{B}(E), \mu)$ be a measurable space with positive measure $\mu$. Let $f : V \times E \rightarrow \mathbb{C}$ be such that:
\begin{enumerate}
    \item $f(\cdot,t)$ is holomorphic for $\mu$-almost every $t \in E$.
    \item $f(s,\cdot)$ is $\mu$-measurable for all $s \in V$.
    \item For every compact $K \subseteq V$, there exists $g \in L^1(\mu)$ such that $|f(s,t)| \leq g(t)$ for all $s \in V$ and for $\mu$-almost every $t \in K$.
\end{enumerate}
Then $\displaystyle F(s) = \int_E f(s,t) \,\mathrm{d}\mu(t)$ is holomorphic on $V$, and $\displaystyle F^{(k)}(s) = \int_E \frac{\partial^k}{\partial s^k} f(s,t) \,\mathrm{d}\mu(t)$.
\end{lemma}

Lemma \ref{thm:entireextensionlemma} is a well-known corollary of the Lebesgue Dominated Convergence Theorem, Morera's Theorem, and the Cauchy integral formula, and it is stated in \cite[Theorem 2.1.47]{lapidusbook}. As such, we will omit its proof.
\newparagraph
%Shows up in FZF, but Lapidus doesn't know exactly where. It's a well-known corollary of Dominated Convergence Theorem and Morera's Theorem
    %Theorem 2.1.47 is the Lemma! (page 82)
%Put proof here!

%\subsection{Entire-Extension Theorem}
%If the bounds are non-0, then extends to entire function
\begin{theorem}[Entire Extension Theorem] \label{thm:entireextension}
Let $A \subseteq \mathbb{R}^N$ be an arbitrary nonempty bounded set for $N \in \mathbb{N}$. If $0 < \alpha < \beta < \infty$, then $\displaystyle F(s) = \int_{\alpha}^{\beta} t^{s-N-1} |A_t| \,\mathrm{d}t$ can be extended to an entire function.
\end{theorem}

\begin{proof}
Let $s_0 \in \mathbb{C}$. By definition of $A_t$, we have that $|A_t|$ is monotonic, continuous, and bounded on $[\alpha,\beta]$. So for any $s \in B(s_0,R) = \{z \in \mathbb{C} : |z-s_0| < R\}$, we have:
\begin{equation}
|F(s+N+1)| = \left|\int_{\alpha}^{\beta} t^s |A_t| \,\mathrm{d}t\right| \leq \int_{\alpha}^{\beta} |t^s| |A_t| \,\mathrm{d}t \leq |A_{\beta}| \int_{\alpha}^{\beta} t^{\re(s)} \,\mathrm{d}t.
\end{equation}
Thus, we deduce the following for the integrand $f(s,t) = t^s |A_t|$ where $s \in B(s_0,R)$:
\begin{equation}\label{eq:dominatedconvergence}
|t^s |A_t|| \leq |A_{\beta}| \cdot \max\{t^{\re(s_0) \pm R}\}
\end{equation}
The right-hand side of Equation \ref{eq:dominatedconvergence} is integrable on $[\alpha,\beta]$ since $\alpha > 0$ and $\beta$ is finite, and the bound is independent of $s \in B(s_0,R)$. Furthermore, $f(\cdot,t)$ is holomorphic for every $t \in [\alpha,\beta]$ since it is an exponential function in $s$. Hence, by Lemma \ref{thm:entireextensionlemma}, $F(s)$ is holomorphic at $s_0$. Therefore, since the choice of $s_0 \in \mathbb{C}$ was arbitrary, $F(s)$ is entire.
\end{proof}
\newparagraph
%Old proof
{
\iffalse
so by the Dominated Convergence Theorem, we have the following:
\begin{align*}
&\left|\lim_{h\rightarrow0} \frac{F(s+h+N+1) - F(s+N+1)}{h} - \int_{\alpha}^{\beta} \ln(t) t^s |A_t| \,\mathrm{d}t\right| \\
%&= \lim_{h\rightarrow0} \left|\frac{1}{h} \int_{\alpha}^{\beta} t^{s+h} |A_t| \,\mathrm{d}t - \frac{1}{h} \int_{\alpha}^{\beta} t^s |A_t| \,\mathrm{d}t - \int_{\alpha}^{\beta} t^s \ln(t) |A_t| \,\mathrm{d}t\right| \\
&= \lim_{h\rightarrow0} \left|\int_{\alpha}^{\beta} \left(\frac{t^h - 1}{h} - \ln(t)\right) t^s |A_t| \,\mathrm{d}t\right| \\
&\leq \lim_{h\rightarrow0} \int_{\alpha}^{\beta} \left|\frac{t^h - 1}{h} - \ln(t)\right| t^{\re(s)} |A_t| \,\mathrm{d}t \\
&= \int_{\alpha}^{\beta} \lim_{h\rightarrow0} \left|\frac{t^h - 1}{h} - \ln(t)\right| t^{\re(s)} |A_t| \,\mathrm{d}t.
\end{align*}
By limit properties, $\displaystyle \lim_{h\rightarrow0} \left|\frac{t^h - 1}{h} - \ln(t)\right| = 0$, and therefore $F(s)$ is differentiable at $s_0 \in \mathbb{C}$. Therefore, since the choice of $s_0$ was arbitrary, $F(s)$ is entire by Goursat's Theorem.\fi
}

\begin{corollary} \label{thm:anydelta}
Given arbitrary $\delta_1 > \delta_2 > 0$, $\tilde{\zeta}_A(s,\delta_1) - \tilde{\zeta}_A(s,\delta_2)$ extends to an entire function.
\end{corollary}

\begin{proof}
Consider two choices $\delta_1,\delta_2$ with $\delta_1 > \delta_2 > 0$. Then we obtain the following:
\begin{equation}
\tilde{\zeta}_A(s,\delta_1) - \tilde{\zeta}_A(s,\delta_2)
= \int_0^{\delta_1} t^{s-N-1} |A_t| \,\mathrm{d}t - \int_0^{\delta_2} t^{s-N-1} |A_t| \,\mathrm{d}t
= \int_{\delta_1}^{\delta_2} t^{s-N-1} |A_t| \,\mathrm{d}t. \label{eq:proofofcorollary}
\end{equation}
By Theorem \ref{thm:entireextension}, Equation (\ref{eq:proofofcorollary}) extends to an entire function. Therefore, $\tilde{\zeta}_A(s,\delta_1) - \tilde{\zeta}_A(s,\delta_2)$ extends to an entire function.
\end{proof}

By Corollary \ref{thm:anydelta}, we no longer need to notate $\delta$ in the tubular zeta function formula, and as such we will simply write the following for any $A \subseteq \mathbb{R}^N$ with $N \in \mathbb{N}$:
\begin{equation}
\tilde{\zeta}_A(s) = \int_0^{\delta} t^{s-N-1} |A_t| \,\mathrm{d}t.
\end{equation}
\newparagraph

\subsection{Threshold Theorems} \label{chap:thresholdtheorem}
Theorem \ref{thm:thresholdexists} can apply to all nonempty compact subsets of $\mathbb{R}^2$ whose complements are path-connected and which are not path-connected. Theorem \ref{thm:infinitethresholdsexist} applies to Sierpi\'nski Carpet modifications of dust type, specifically leveraging their self-similarity.

%Since disconnected, there exists a point where Inflated Neighborhood is no longer path-connected, that repeats infinitely-many times because self-similar

\begin{theorem}[Threshold Theorem] \label{thm:thresholdexists}
Let $A \subseteq \mathbb{R}^2$ be a nonempty compact set with a path-connected complement and which is itself not path-connected. Then there exists a maximal $T > 0$ where the set $A_T = \{x \in \mathbb{R}^2 : d(x,A) < T\}$ is not path-connected.
\end{theorem}
\begin{proof}
Consider the set $\mathscr{T} = \{t > 0 : A_t \text{ is not path-connected}\}$. Since $A$ is not path connected, we can consider two path-connected components $E_1$ and $E_2$ of $A$ for which there is no continuous path from a point in $E_1$ to a point in $E_2$ which goes through $A$. Since $A$ is compact, so are $E_1$ and $E_2$ by properties of $\mathbb{R}^2$ with the standard Euclidean topology. Since there is no continuous path from $E_1$ to $E_2$, they must be disjoint, and therefore setting
\begin{equation}
T = \frac{1}{2} \cdot \inf_{x\in E_1, y\in E_2} \{d(x,y)\} > 0,
\end{equation}
will guarantee that $A_T$ is not path-connected. Therefore, $\mathscr{T} \neq \emptyset$ since $T \in \mathscr{T}$. \\

Furthermore, by construction of $A$, we have that $\displaystyle \diam(A) = \sup_{x\in A, y \in A}\{d(x,y)\}$ is finite since $A \subseteq [0,1]\times[0,1]$. Thus, by choosing $\delta > \sqrt{2}$, we deduce that $A \subseteq B(x, \delta)$ for any $x \in A$, which shows that $A_{\delta}$ is simply-connected (and consequently it is path-connected). Thus, we have that $\mathscr{T}$ is bounded above by $\sqrt{2}$. Therefore, by completeness of $\mathbb{R}$, there exists a value $T = \sup(\mathscr{T})$ which is the maximal number where $A_T$ is not path-connected.
\end{proof}
\newparagraph

From here onwards, we will assume that $\delta > 0$ is large enough that $A_{\delta}$ is simply-connected (the proof of Theorem \ref{thm:thresholdexists} justifies that this is always possible). By Corollary \ref{thm:anydelta}, this assumption will have no impact on the poles of the tubular zeta function.
\newparagraph

\begin{lemma} \label{thm:inflationlemma}
Let $A \subseteq \mathbb{R}^N$ for $N \in \mathbb{N}$ and let $\lambda > 0$. Then for any $t \geq 0$:
\begin{equation}
(\lambda A)_t = \lambda \left(A_{\frac{1}{\lambda} t}\right).
\end{equation}
\end{lemma}
\begin{proof}
Consider the set $(\lambda A)_t = \{x \in \mathbb{R}^N : d(x,\lambda A) < t\}$. Making the substitution $x = \lambda y$ yields:
\begin{align*}
(\lambda A)_t &= \{\lambda y \in \mathbb{R}^N : d(\lambda y, \lambda A) < t\}\\
&= \{\lambda y \in \mathbb{R}^N : \lambda d(y, A) < t\}\\
&= \left\{\lambda y \in \mathbb{R}^N : d(y, A) < \frac{1}{\lambda} t\right\}\\
&= \lambda \left\{y \in \mathbb{R}^N : d(y, A) < \frac{1}{\lambda} t\right\}\\
&= \lambda \left(A_{\frac{1}{\lambda} t}\right).
\end{align*}
Therefore, $\displaystyle (\lambda A)_t = \lambda \left(A_{\frac{1}{\lambda} t}\right)$.
\end{proof}
\newparagraph

\begin{theorem}[Infinite Thresholds Theorem] \label{thm:infinitethresholdsexist}
Let $A \subseteq \mathbb{R}^2$ be a Sierpi\'nski Carpet modification of dust type constructed by keeping $m \geq 2$ squares from a $p \times p$ grid. Let $T > 0$ be the maximal number such that $A_T$ is not path-connected, as guaranteed by Theorem \ref{thm:thresholdexists}. Then for each $n \in \mathbb{N} \cup \{0\}$, $\frac{T}{p^{n+1}}$ is the maximal number such that every path-connected component of $A_{\frac{T}{p^n}}$ is not path-connected.
\end{theorem}

\begin{proof}
Let $n \in \mathbb{N}$. Since $A$ is constructed by keeping $m$ squares from a $p \times p$ grid, we have a self-similar fractal which can be written as follows:
\begin{equation}
A = \bigcup_{i=1}^m \left(\frac{1}{p} A + (x_i,y_i)\right),
\end{equation}
where $(x_i,y_i)$ are the corners of the subsquares and where we translate $\frac{1}{p} A$ in order to line up with the grid. Since this is a recursive definition, we have the following:
\begin{equation}
A = \bigcup_{j=1}^{m^{n+1}} \left(\frac{1}{p^{n+1}} A + (x_j,y_j)\right),
\end{equation}
where $(x_j,y_j)$ are the corners of the subsquares and where we translate $\frac{1}{p^{n+1}} A$ in order to line up with the grid. Since the Euclidean metric is translation-invariant in $\mathbb{R}^2$, we have the following for each $t \in [0,\delta]$:
\begin{equation} \label{eq:selfsimilardef}
A_t = \bigcup_{j=1}^{m^{n+1}} \left(\left(\frac{1}{p^{n+1}} A\right)_t + (x_j,y_j)\right).
\end{equation}
Let $T > 0$ be the maximal number such that $A_T$ is not path-connected, as guaranteed by Theorem \ref{thm:thresholdexists}. Then by properties of $\mathbb{R}^2$, we have 
\begin{equation}
A_{\frac{T}{p^n}} = (E_1)_{\frac{T}{p^n}} \cup (E_2)_{\frac{T}{p^n}} \cup\dots\cup (E_k)_{\frac{T}{p^n}},
\end{equation}
for some $k \leq m^n$, where each $(E_i)_{\frac{T}{p^n}} \cap (E_j)_{\frac{T}{p^n}} = \emptyset$ for $i \neq j$ and each $(E_j)_{\frac{T}{p^n}}$ is path-connected. Thus, for each $j = 1,\dots,k$, we have
\begin{equation}
(E_j)_{\frac{T}{p^n}} = \bigcup_{s \in S_j} \left(\left(\frac{1}{p^{n+1}} A\right)_{\frac{T}{p^n}} + (x_s,y_s)\right),
\end{equation}
for some $S_j \subseteq \{1,2,\dots,m^{n+1}\}$. By Lemma \ref{thm:inflationlemma}, we obtain the following:
\begin{equation}
(E_j)_{\frac{T}{p^{n+1}}} = \bigcup_{s \in S_j} \left(\left(\frac{1}{p^{n+1}} A\right)_{\frac{T}{p^{n+1}}} + (x_s,y_s)\right)
= \bigcup_{s \in S_j} \left(\frac{1}{p^{n+1}} A_T + (x_s,y_s)\right).
\end{equation}
Since $t = T$ is the maximal number such that $A_t$ is not path-connected, we therefore have by definition of $E_j$ that $t = \frac{T}{p^{n+1}}$ is the maximal number such that $(E_j)_t$ is not path-connected. This applies for each $j = 1,\dots,k$, and so $\frac{T}{p^{n+1}}$ is the maximal number such that the path-connected components of $A_{\frac{T}{p^n}}$ are no longer path-connected.
\end{proof}

\section{Simplest Case: Completely Dusty} \label{chap:alphaconnected}
The upcoming Theorem \ref{thm:alphaconnected} is intended to handle the simplest of dust type Sierpi\'nski Carpet modifications and highlight the techniques that will be used later.

\begin{definition}[Completely Dusty] \label{def:alphaconnected}
Let $A$ be a dust type Sierpi\'nski Carpet modification constructed from a $p \times p$ grid, and let $\alpha > 0$ be the maximal number such that $A_{\alpha}$ is not path-connected (which exists by Theorem \ref{thm:thresholdexists}). $A$ is \textit{completely dusty} if $A_{\alpha}$ is the union of $m$ disjoint copies of $\frac{1}{p} \left(A_{p\alpha}\right)$. %The inflated neighborhood breaks apart completely into m self-similar pieces
\end{definition}

The Ternary Cantor Set discussed earlier in Figure \ref{fig:cantorset} is an example of a completely dusty carpet. It is clear that $\alpha = \frac{1}{6}$ has the property that $A_{\frac{1}{6}}$ consists of 2 disjoint copies of $A_{\frac{1}{2}}$ which are scaled down by a factor of $\frac{1}{3}$, as shown in Figure \ref{fig:cantorsetalphaconnected}.
\begin{figure}[H]
\includegraphics[scale=0.3]{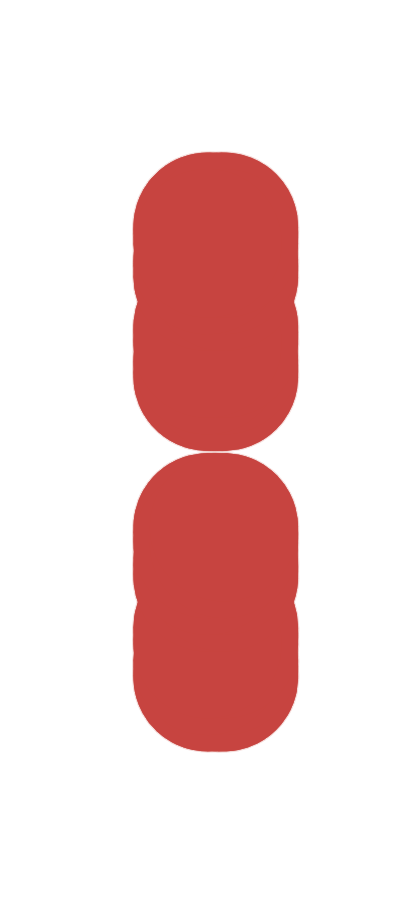}
\includegraphics[scale=0.3]{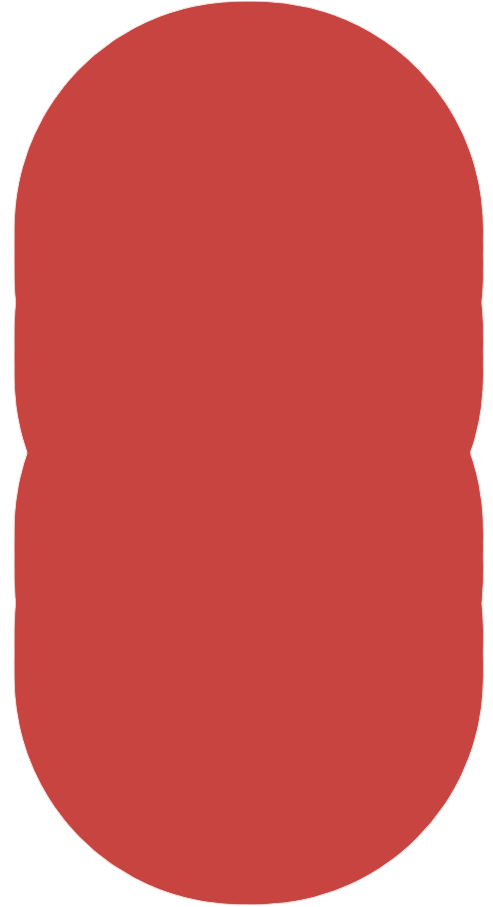}
\caption{The Ternary Cantor Set with inflated neighborhood shown at $t = \frac{1}{6}, \frac{1}{2}$.}
\label{fig:cantorsetalphaconnected}
\end{figure}

\begin{theorem}[Completely Dusty Theorem] \label{thm:alphaconnected}
Let $A$ be a completely dusty Sierpi\'nski Carpet modification constructed by keeping $m \geq 2$ squares from the $p \times p$ grid. Then $\tilde{\zeta}_A$ has a meromorphic continuation to all of $\mathbb{C}$, and
\begin{equation}
\mathscr{P}(A) \subseteq \left\{\log_p(m) + \frac{2\pi i j}{\ln(p)}: j \in \mathbb{Z}\right\}.
\end{equation}
\end{theorem}
\begin{proof}
Let $\alpha > 0$ be the maximal number such that $A_{\alpha}$ is not path-connected, as given by Theorem \ref{thm:thresholdexists}. Consider the function $\tilde{A}(t) : [\alpha, p\alpha] \rightarrow \mathbb{R}$ such that $\tilde{A}(t) = |A_t|$ whenever $t \in [\alpha,p\alpha]$. Then by Definition \ref{def:alphaconnected} and Theorem \ref{thm:infinitethresholdsexist}, we have that whenever $t \in \left[\frac{\alpha}{p^k}, \frac{\alpha}{p^{k-1}}\right]$ for some $k \in \mathbb{N}$, we obtain the following:
\begin{equation}
|A_t| = m \left|\left(\frac{1}{p} A\right)_{p t}\right| =\dots= m^k \left|\left(\frac{1}{p^k} A\right)_{p^k t}\right|.
\end{equation}
We can assume without loss of generality that $\alpha < \delta \leq p\alpha$. Plugging this into the tubular zeta function formula and using Theorem \ref{thm:scaling} yields:
\begin{align}
\tilde{\zeta}_A(s)
&= \int_{\alpha}^{\delta} t^{s-3} |A_t| \,\mathrm{d}t + \sum_{k=1}^{\infty} \int_{\frac{\alpha}{p^k}}^{\frac{\alpha}{p^{k-1}}} t^{s-3} |A_t| \,\mathrm{d}t \\
&= \int_{\alpha}^{\delta} t^{s-3} \tilde{A}(t) \,\mathrm{d}t + \sum_{k=1}^{\infty} \int_{\frac{\alpha}{p^k}}^{\frac{\alpha}{p^{k-1}}} t^{s-3} m^k \left|\left(\frac{1}{p^k} A\right)_{p^k t}\right| \,\mathrm{d}t\\
&= \int_{\alpha}^{\delta} t^{s-3} \tilde{A}(t) \,\mathrm{d}t + \sum_{k=1}^{\infty} \int_0^{\frac{\alpha}{p^{k-1}}} t^{s-3} m^k \left|\left(\frac{1}{p^k} A\right)_{p^k t}\right| \,\mathrm{d}t \nonumber\\&- \int_0^{\frac{\alpha}{p^k}} t^{s-3} m^k \left|\left(\frac{1}{p^k} A\right)_{p^k t}\right| \,\mathrm{d}t \\
&= \int_{\alpha}^{\delta} t^{s-3} \tilde{A}(t) \,\mathrm{d}t + \sum_{k=1}^{\infty} \frac{1}{p^{k s}} \int_0^{p\alpha} t^{s-3} m^k \left|A_t\right| \,\mathrm{d}t \nonumber\\&- \frac{1}{p^{k s}} \int_0^{\alpha} t^{s-3} m^k \left|A_t\right| \,\mathrm{d}t\\
&= \int_{\alpha}^{\delta} t^{s-3} \tilde{A}(t) \,\mathrm{d}t + \sum_{k=1}^{\infty} \frac{m^k}{p^{k s}} \int_{\alpha}^{p\alpha} t^{s-3} \tilde{A}(t) \,\mathrm{d}t.
\end{align}
Since $\tilde{A}(t)$ is monotone, continuous, and bounded due to it being a domain restriction of $|A_t|$, and $\alpha > 0$, we have by Theorem \ref{thm:entireextension} that $\displaystyle \int_{\alpha}^{\delta} t^{s-3} \tilde{A}(t) \,\mathrm{d}t$ and $\displaystyle \int_{\alpha}^{p\alpha} t^{s-3} \tilde{A}(t) \,\mathrm{d}t$ extend to entire functions. This means that the only source of possible poles is $\displaystyle\sum_{k=1}^{\infty} \frac{m^k}{p^{k s}}$. The aforementioned series can be written as a classic geometric series as follows:
\begin{equation}
\sum_{k=1}^{\infty} \left(\frac{m}{p^{s}}\right)^k = \frac{m}{p^s-m}.
\end{equation}
Notably, the sum on the left-hand side is defined only for $\left\{s \in \mathbb{C} : m < p^{\re(s)}\right\}$. However, the right-hand side is defined everywhere except at $\left\{\log_p(m) + \frac{2\pi i j}{\ln(p)} : j \in \mathbb{Z}\right\}$. Thus, by the principle of meromorphic extension, we can analytically extend $\tilde{\zeta}_A$ to $\mathbb{C} \backslash \left\{\log_p(m) + \frac{2\pi i j}{\ln(p)} : j \in \mathbb{Z}\right\}$, or meromorphically extend to $\mathbb{C}$. Thus, the only possible poles of $\tilde{\zeta}_A(s)$ will be simple poles at $s = \log_p(m) + \frac{2\pi i j}{\ln(p)}$ for every $j \in \mathbb{Z}$. Therefore,
\begin{equation}
\mathscr{P}(A) \subseteq \left\{\log_p(m) + \frac{2\pi i j}{\ln(p)}: j \in \mathbb{Z}\right\}.
\end{equation}
\end{proof}
\newparagraph

Note that there may be $s \in \mathbb{C}$ such that $\displaystyle \int_{\alpha}^{p\alpha} t^{s-3} \tilde{A}(t) \,\mathrm{d}t = 0$, which could potentially cancel some of the poles listed in $\mathscr{P}(A)$. We can see that the pole at $s = \log_p(m)$ will never be canceled out since the integrand $t^{\log_p(m)-3} \tilde{A}(t)$ is strictly positive (and thus $\displaystyle \int_{\alpha}^{p\alpha} t^{\log_p(m)-3} \tilde{A}(t) \,\mathrm{d}t > 0$). However, we currently do not have methods to determine which of the additional non-real poles listed in $\mathscr{P}(A)$ are canceled in a general case, so we will leave this as an open problem.
\newparagraph

\subsection{Worked Example: Cantor Dust} \label{chap:cantordustexample}
Consider the Ternary Cantor Set $\mathfrak{C} \subseteq [0,1]$, and let $A$ be a Sierpi\'nski Carpet modification which produces $A = \mathfrak{C} \times \mathfrak{C} \subseteq \mathbb{R}^2$, shown in Figure \ref{fig:cantordust}.
\begin{figure}[H]
\includegraphics[scale=0.3]{Images/NEquals0}
\includegraphics[scale=0.3]{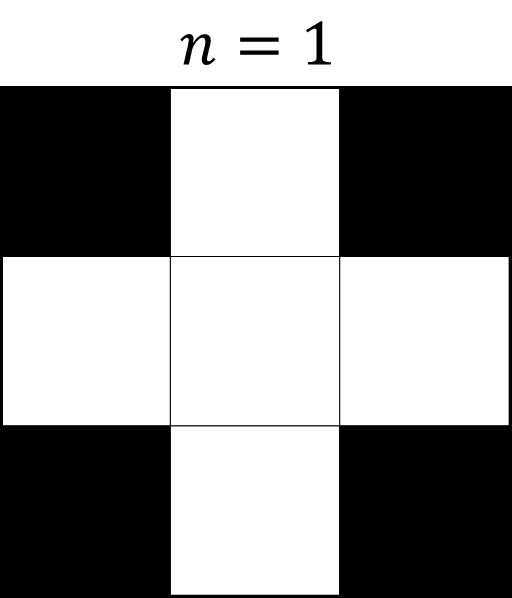}
\includegraphics[scale=0.3]{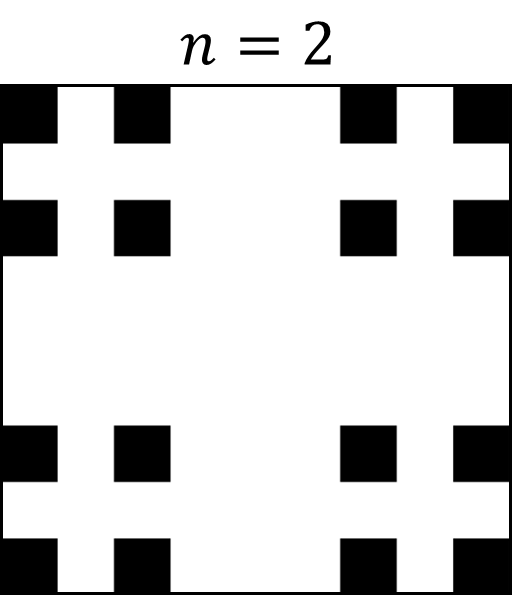}\\
\includegraphics[scale=0.3]{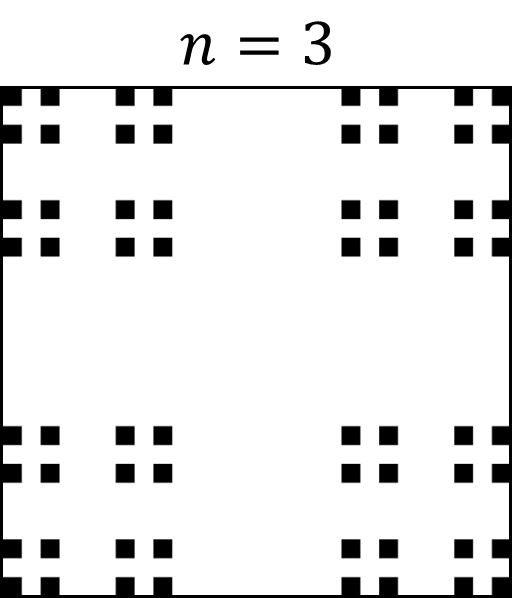}
\includegraphics[scale=0.3]{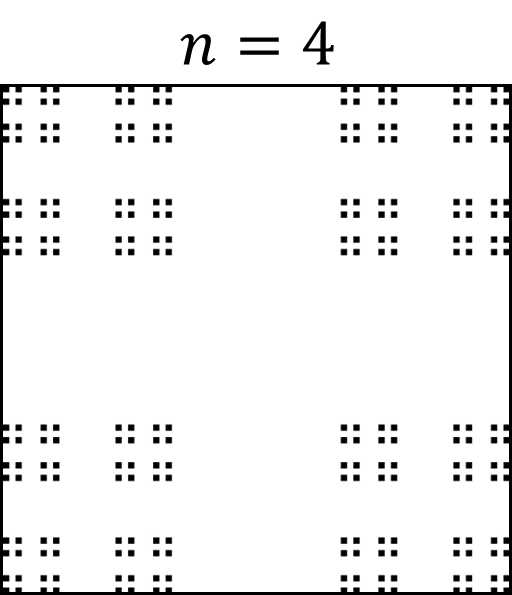}
\includegraphics[scale=0.3]{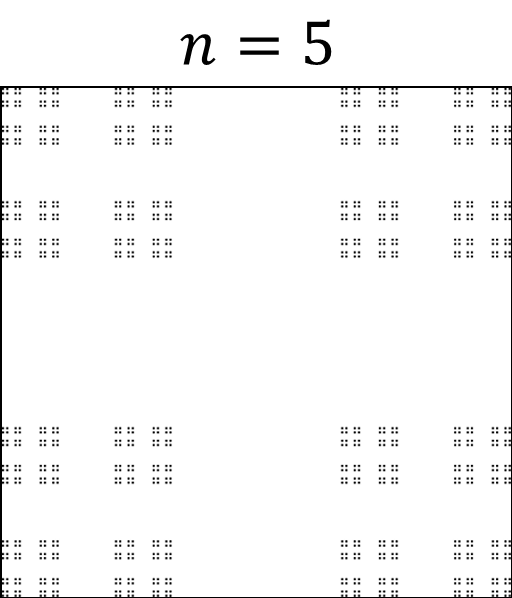}\\
\caption{The first 5 steps of the construction of a Sierpi\'nski Carpet modification which produces the classic ``Cantor Dust".}
\label{fig:cantordust}
\end{figure}

For this carpet, $m = 4$ and $p = 3$. This example is discussed in \cite[Example 4.7.15]{lapidusbook} and \cite[Conjecture 3.16]{lapidusslopaper}, where the complex dimensions are computed to be
\begin{equation}
\mathscr{P}(A) \subseteq \left\{\log_3(2) + \frac{2\pi i j}{\ln(3)}, \log_3(4) + \frac{2\pi i j}{\ln(3)}: j \in \mathbb{Z}\right\}.
\end{equation}
However, note how $\alpha = \frac{1}{6}$ is sufficient to make $A_{\alpha}$ equal 4 disjoint copies of $A_{3\alpha}$, and thus Theorem \ref{thm:alphaconnected} computes its complex dimensions to be
\begin{equation}
\mathscr{P}(A) \subseteq \left\{\log_3(4) + \frac{2\pi i j}{\ln(3)}: j \in \mathbb{Z}\right\}.
\end{equation}
Thus, we obtain a smaller set of possible poles than previously known bounds, but neither contradicts the other. \\

For completeness, we will work through the computation of the tubular neighborhood area. We already know that $\alpha = \frac{1}{6}$ and $p = 3$, and so we can assume that $\delta \in \left(\frac{1}{6}, \frac{1}{2}\right]$ without loss of generality (Corollary \ref{thm:anydelta} justifies doing this). In order to compute the area of the tubular neighborhood, we will use the ``cusp'' areas introduced in Equation (\ref{eq:cusparea}), stated again below for convenience:
\begin{equation}
C(r,d) = dr - \frac{1}{4} d \sqrt{4r^2 - d^2} - r^2 \arctan\left(\frac{d}{\sqrt{4r^2 - d^2}}\right).
\end{equation}
We also need ``petal'' areas, shown in Figure \ref{fig:petalarea}.
\begin{figure}[H]
\includegraphics[scale=0.5]{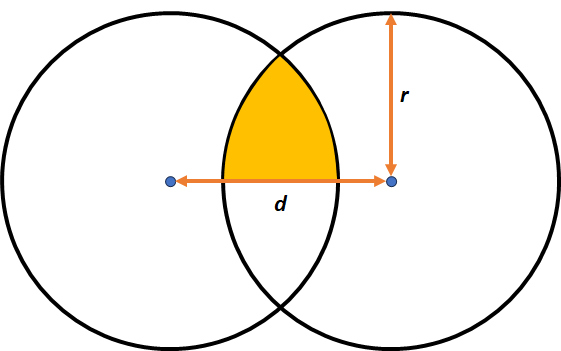}
\caption{Area of ``petal'' between two circles, shaded in yellow.}
\label{fig:petalarea}
\end{figure}
Using standard calculus techniques, we see that when $d \in (r,2r)$, we have the following area:
\begin{align}
P(r,d) &= 2\int_{-\sqrt{r^2-\frac{d^2}{4}}}^0 \left(2\sqrt{r^2-y^2} - d\right) \,\mathrm{d}y \nonumber\\
&= 3d\sqrt{r^2-\frac{d^2}{4}} + r^2 \arctan\left(\frac{\sqrt{4r^2-d^2}}{d}\right).
\label{eq:petalarea}
\end{align}
Using the ``petals'', we can immediately compute the areas of some features of the tubular neighborhood, shown in Figure \ref{fig:flowerareas}.
\begin{figure}[H]
\includegraphics[scale=0.8]{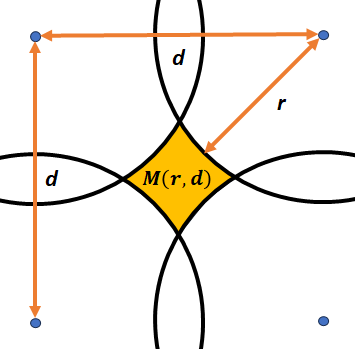}
\includegraphics[scale=0.8]{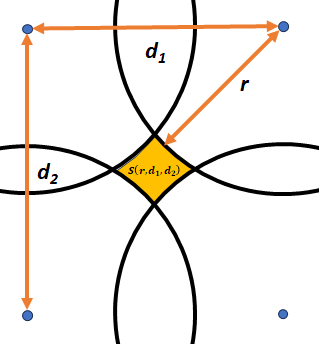}
\caption{Two features which will be subtracted from the tubular neighborhood depending on the value of $t$.}
\label{fig:flowerareas}
\end{figure}
Using the Inclusion-Exclusion Principle, we can compute $M(r,d)$ for $r \in \left(\frac{d}{2}, \frac{d\sqrt{2}}{2}\right)$:
\begin{equation}
M(r,d) = d^2 - \pi r^2 + 4 \cdot P(r,d).
\end{equation}
We can also compute $S(r,d_1,d_2)$ for $d_2 > d_1$ and $r \in \left(\frac{d_2}{2}, \frac{\sqrt{d_1^2+d_2^2}}{2}\right)$:
\begin{equation}
S(r,d_1,d_2) = d_1 \cdot d_2 - \pi r^2 + 2 \cdot P(r,d_1) + 2 \cdot P(r,d_2).
\end{equation}

Using the areas shown in Figures \ref{fig:cusparea}, \ref{fig:petalarea}, and \ref{fig:flowerareas}, we can now compute the area of the tubular neighborhood for $t \in \left(\frac{1}{6}, \frac{1}{2}\right]$:
\begin{align}
|A_t| &= 1 + 4t + \pi t^2 - 4 \left(\sum_{k=1}^{\infty} 2^{k-1} C\left(t,\frac{1}{3^k}\right)\right)\nonumber\\&- M\left(t,\frac{1}{3}\right) \chi_{\left(\frac{1}{6}, \frac{\sqrt{2}}{6}\right]}(t) - 4 \left(\sum_{k=1}^{\infty} 2^{k-1} S\left(t,\frac{1}{3^{k+1}},\frac{1}{3}\right) \chi_{\left(\frac{1}{6}, \frac{1}{2}\sqrt{\frac{1}{9}+\frac{1}{3^{2(k+1)}}}\right]}(t)\right),
\end{align}
where we denote $\chi_E(t)$ to be the indicator function for any $E \subseteq \mathbb{R}$ and $t \in \mathbb{R}$:
\begin{equation}
\chi_E(t) = \begin{cases}
1, & t \in E \\
0, & t \notin E
\end{cases}.
\end{equation}
The resulting integrals $\displaystyle\int_{\frac{1}{6}}^{\frac{1}{2}} t^{s-3} |A_t| \,\mathrm{d}t$ and $\displaystyle\int_{\frac{1}{6}}^{\delta} t^{s-3} |A_t| \,\mathrm{d}t$ are non-elementary. However, Theorem \ref{thm:entireextension} confirms that the aforementioned integrals produce no poles. Thus, we have that the only possible poles of $\tilde{\zeta}_A$ are
\begin{equation}
\mathscr{P}(A) \subseteq \left\{\log_3(4) + \frac{2\pi i j}{\ln(3)}: j \in \mathbb{Z}\right\},
\end{equation}
as predicted by Theorem \ref{thm:alphaconnected}.
%Point out that this does not contradict the conjecture from Example 4.7.15 or the SLO Paper conjecture, just adds more precise information (as opposed to the Cantor Grill)

\section{Computing Dimensions of All the Dust Types} \label{chap:computingdust}
We will now discuss the more geometrically complicated dust type examples which are not completely dusty. From here onwards, $A$ is a dust type Sierpi\'nski Carpet modification constructed by keeping $m$ squares in a $p \times p$ grid for $p \in \mathbb{N} \cap [3,\infty)$ and $m \in \mathbb{N} \cap [2, p^2-1]$, and $\alpha > 0$ is the maximal number such that $A_{\alpha}$ is not path-connected, as given by Theorem \ref{thm:thresholdexists}. As described in the proof of Theorem \ref{thm:infinitethresholdsexist}, we have that $A_{\alpha}$ consists of $m$ copies of $A_{p\alpha}$ scaled down by a factor of $\frac{1}{p}$, though these copies may not be disjoint from each other.

\subsection{Reduction to Combinatorics} \label{chap:intersectiontypes}
We can classify 9 different ``intersection types'' based on how squares can geometrically share edges or corners in Euclidean $\mathbb{R}^2$ space. These intersection types are visually shown in Figure \ref{fig:intersectiontypes}.
\begin{figure}[H]
\includegraphics[scale=0.5]{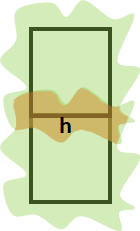}
\includegraphics[scale=0.5]{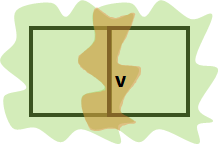}
\includegraphics[scale=0.5]{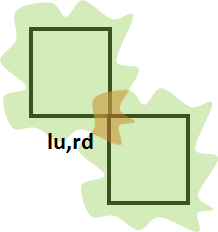}\\
\includegraphics[scale=0.5]{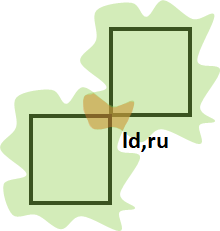}
\includegraphics[scale=0.5]{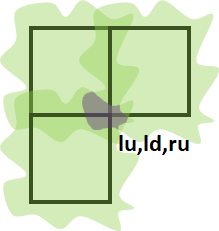}
\includegraphics[scale=0.5]{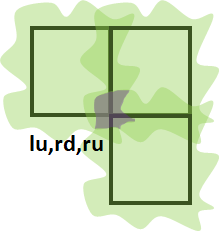}\\
\includegraphics[scale=0.5]{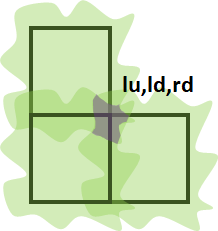}
\includegraphics[scale=0.5]{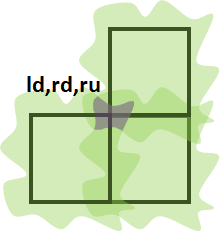}
\includegraphics[scale=0.5]{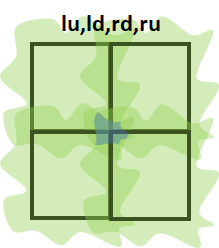}
\caption{The 9 different intersection types. The green blobs are generic representations of $A_{\alpha}$.}
\label{fig:intersectiontypes}
\end{figure}
Reading Figure \ref{fig:intersectiontypes} left-to-right and top-to-bottom, we will denote the areas of the shaded regions by $\mathfrak{I}_t^{h}, \mathfrak{I}_t^{v}, \mathfrak{I}_t^{lu,rd},\dots,\mathfrak{I}_t^{lu,ld,rd,ru}$. We can also define the intersection types algebraically as follows:
\begin{align}
\mathfrak{I}_t^{h} &= \left(\frac{1}{p} A_{pt}\right) \cap \left(\frac{1}{p} A_{pt} + \left(0,\frac{1}{p}\right)\right) \label{eq:intersectionh}, \\
\mathfrak{I}_t^{v} &= \left(\frac{1}{p} A_{pt}\right) \cap \left(\frac{1}{p} A_{pt} + \left(\frac{1}{p},0\right)\right), \\
\mathfrak{I}_t^{lu,rd} &= \left(\frac{1}{p} A_{pt} + \left(0,\frac{1}{p}\right)\right) \cap \left(\frac{1}{p} A_{pt} + \left(\frac{1}{p},0\right)\right), \\
%\mathfrak{I}_t^{ld,ru} &= \left(\frac{1}{p} A_{pt}\right) \cap \left(\frac{1}{p} A_{pt} + \left(\frac{1}{p},\frac{1}{p}\right)\right), \\
%\mathfrak{I}_t^{lu,ld,ru} &= \left(\frac{1}{p} A_{pt}\right) \cap \left(\frac{1}{p} A_{pt} + \left(0,\frac{1}{p}\right)\right) \cap \left(\frac{1}{p} A_{pt} + \left(\frac{1}{p},\frac{1}{p}\right)\right), \\
%\mathfrak{I}_t^{lu,rd,ru} &= \left(\frac{1}{p} A_{pt} + \left(0,\frac{1}{p}\right)\right) \cap \left(\frac{1}{p} A_{pt} + \left(\frac{1}{p},0\right)\right) \cap \left(\frac{1}{p} A_{pt} + \left(\frac{1}{p},\frac{1}{p}\right)\right), \\
%\mathfrak{I}_t^{lu,ld,rd} &= \left(\frac{1}{p} A_{pt}\right) \cap \left(\frac{1}{p} A_{pt} + \left(\frac{1}{p},0\right)\right) \cap \left(\frac{1}{p} A_{pt} + \left(0,\frac{1}{p}\right)\right), \\
%\mathfrak{I}_t^{ld,rd,ru} &= \left(\frac{1}{p} A_{pt}\right) \cap \left(\frac{1}{p} A_{pt} + \left(\frac{1}{p},0\right)\right) \cap \left(\frac{1}{p} A_{pt} + \left(\frac{1}{p},\frac{1}{p}\right)\right), \\
\vdots\nonumber\\
\mathfrak{I}_t^{lu,ld,rd,ru} &= \left(\frac{1}{p} A_{pt}\right) \cap \left(\frac{1}{p} A_{pt} + \left(0,\frac{1}{p}\right)\right) \cap \left(\frac{1}{p} A_{pt} + \left(\frac{1}{p},0\right)\right) \cap \left(\frac{1}{p} A_{pt} + \left(\frac{1}{p},\frac{1}{p}\right)\right). \label{eq:intersection4way}
\end{align}
Since Lebesgue measure and Euclidean metric are translation-invariant, each $\mathfrak{I}_t^{type}$ is measurable. Furthermore, since $|A_t|$ is continuous, bounded, and monotonically increasing as a function of $t$, the algebraic definitions demonstrate that the same properties apply to each $\left|\mathfrak{I}_t^{type}\right|$ since the sets being measured are finite intersections and translations of the original $A_t$. Using this fact, we can use the Inclusion-Exclusion Principle to compute the area of the tubular neighborhood of $A$.
\newparagraph

\begin{theorem}[General Dust Theorem in $\mathbb{R}^2$ (Preliminary Version)]\label{thm:gammaconnected1}
Let $A$ be a dust type Sierpi\'nski Carpet modification constructed by keeping $m \geq 2$ squares from the $p \times p$ grid. Let $\alpha > 0$ be the maximal number such that $A_{\alpha}$ is not path-connected, as given by Theorem \ref{thm:thresholdexists}. Then $\tilde{\zeta}_A$ has a meromorphic continuation to all of $\mathbb{C}$ and
\begin{align}
\mathscr{P}(A) &\subseteq \left\{\log_p(m) + \frac{2\pi i j}{\ln(p)} : j \in \mathbb{Z}\right\}\nonumber\\&\cup \bigcup_{(type) \text{ s.t. Condition } C \text{ holds}} \left\{s \in \mathbb{C} : \substack{{\text{$s$ is a pole of the meromorphic continuation of}}\\{\displaystyle\sum_{k=1}^{\infty} \frac{I_{type}(k)}{p^{(k-1)s}}}}\right\},
\end{align}
where $C$ is the condition that $\left|\left\{t \in \left[\frac{\alpha}{p},\alpha\right] : \left|\mathfrak{I}_t^{type}\right| \neq 0\right\}\right| > 0$, and $I_{type}(k)$ counts the geometric intersections of squares corresponding to that type present at fractal construction level $k$.
\end{theorem}

Examples of Condition $C$ holding or not holding are shown in Figures \ref{fig:dustyl} and \ref{fig:spaceinvaders}.

\begin{figure}[H]
\includegraphics[scale=0.4]{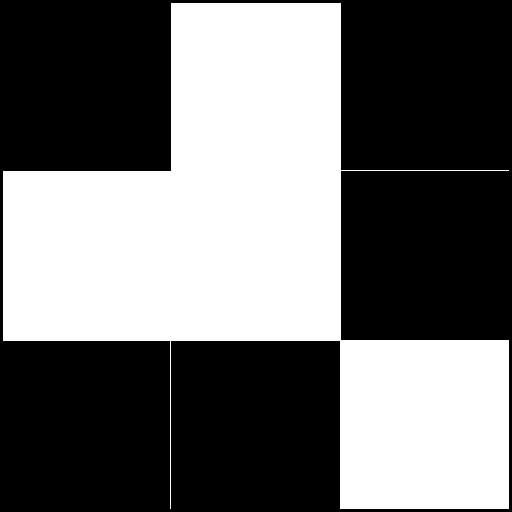}
\includegraphics[scale=0.3]{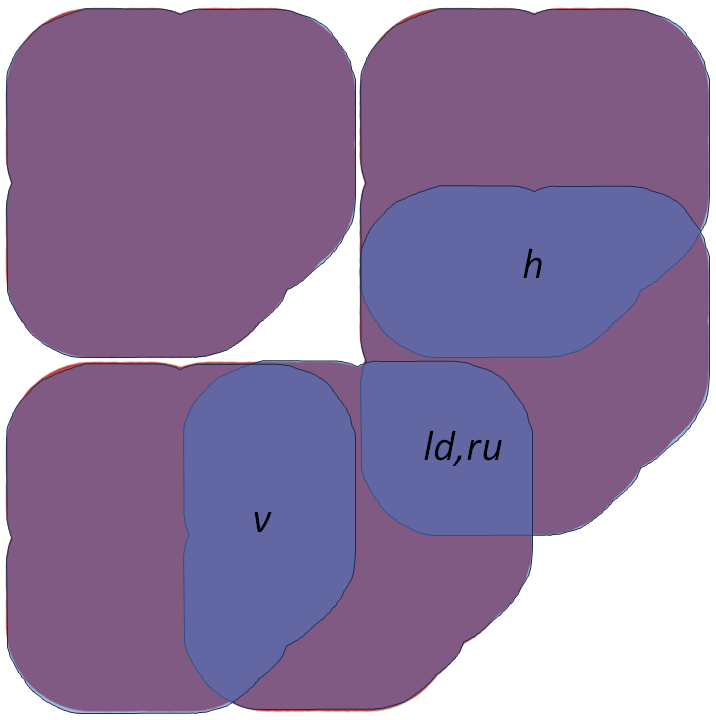}
\caption{On this dust type Sierpi\'nski Carpet in the $3 \times 3$ grid, there are $\{ld,ru\}$ intersection types in the finite construction, and $|\mathfrak{I}_t^{ld,ru}| > 0$ for all $t \in \left[\frac{\alpha}{p},\alpha\right]$ (a positive-measure subset of $\left[\frac{\alpha}{p},\alpha\right]$). In this case, $\{ld,ru\}$ intersection types will pass Condition $C$ in Theorem \ref{thm:gammaconnected1}.}
\label{fig:dustyl}
\end{figure}

\begin{figure}[H]
\includegraphics[scale=0.4]{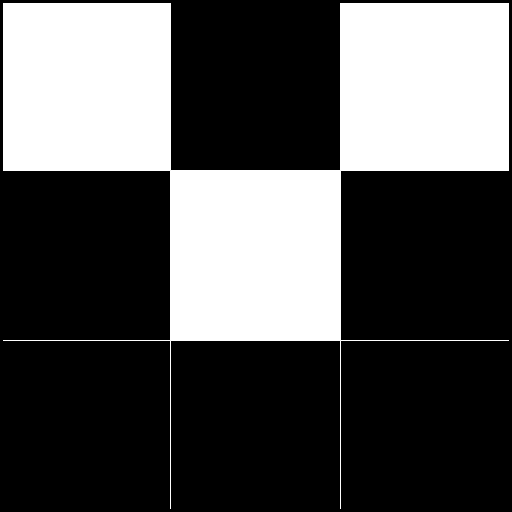}
\includegraphics[scale=0.35]{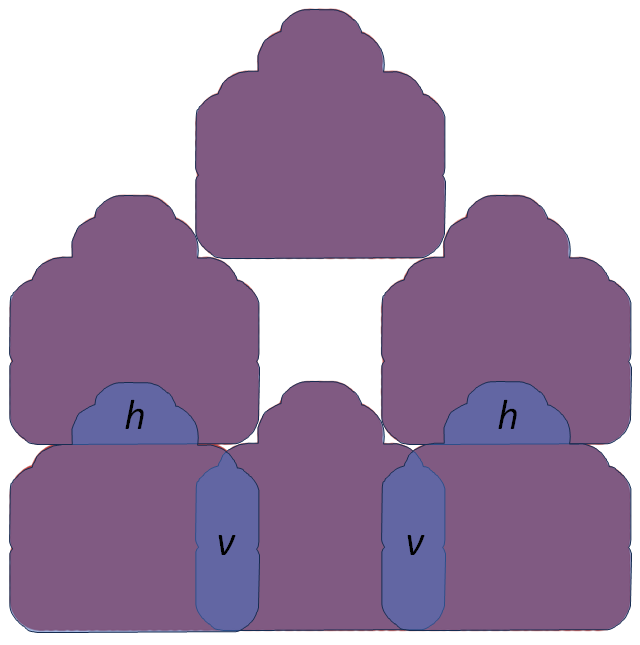}
\caption{On this dust type Sierpi\'nski Carpet in the $3 \times 3$ grid, there are $\{ld,ru\}$ intersection types in the finite construction, but $|\mathfrak{I}_t^{ld,ru}| > 0$ only for $t = \alpha$ (a measure-0 subset of $\left[\frac{\alpha}{p},\alpha\right]$). In this case, $\{ld,ru\}$ intersection types do not pass Condition $C$ in Theorem \ref{thm:gammaconnected1}.}
\label{fig:spaceinvaders}
\end{figure}

\begin{proof}
By construction of a Sierpi\'nksi Carpet modification, we have that $A_{k+1} \subseteq A_k$ for each $k \in \mathbb{N}$, and thus $(A_{k+1})_t \subseteq (A_k)_t$ for all $t \in [0,\delta]$. If we then consider $(A_k)_t$ to be composed of scaled-down copies of itself that intersect each other according to the previously-defined intersection types, then for each intersection type we also have that $(\mathfrak{I}_{k+1})_t^{type} \subseteq (\mathfrak{I}_k)_t^{type}$ for each $t \in \left(\frac{\alpha}{p^n},\frac{\alpha}{p^{n-1}}\right]$ and $k \geq n$, where $A_n$ is the first level of the fractal construction such that $(A_n)_t$ has that intersection type. Thus, since each intersection type is translation-invariant, we have a one-to-one correspondence between the geometric intersections present in $A_n$ and the number of intersection areas $\mathfrak{I}_t^{type}$ present in $A_t$ for $t \in \left(\frac{\alpha}{p^n},\frac{\alpha}{p^{n-1}}\right]$. Hence, the combinatorics relies exclusively on counting how the squares share edges and corners at each finite construction level. Let $I_{type}(n)$ count the number of these geometric intersections of squares at fractal construction level $n$. \\

Let $\alpha > 0$ be the maximal number such that $A_{\alpha}$ is not path-connected, as given by Theorem \ref{thm:thresholdexists}. Since $\mathfrak{I}_t^{type}$ is translation-invariant, using the Inclusion-Exclusion Principle and Theorem \ref{thm:infinitethresholdsexist} yields the following for each $t \in \left[\frac{\alpha}{p^k}, \frac{\alpha}{p^{k-1}}\right]$ where $k \in \mathbb{N}$:
\begin{align}
|A_t| &=
m^k \left|\left(\frac{1}{p^k} A\right)_{p^k t}\right| \nonumber\\
&- I_{h}(k) \left|\frac{1}{p^{k-1}} \mathfrak{I}_{p^{k-1} t}^{h}\right| - I_{v}(k) \left|\frac{1}{p^{k-1}} \mathfrak{I}_{p^{k-1} t}^{v}\right| \nonumber\\& - I_{ld,ru}(k) \left|\frac{1}{p^{k-1}} \mathfrak{I}_{p^{k-1} t}^{ld,ru}\right| - I_{lu,rd}(k) \left|\frac{1}{p^{k-1}} \mathfrak{I}_{p^{k-1} t}^{lu,rd}\right| \nonumber\\
&+ I_{ld,rd,ru}(k) \left|\frac{1}{p^{k-1}} \mathfrak{I}_{p^{k-1} t}^{ld,rd,ru}\right| + I_{lu,rd,ru}(k) \left|\frac{1}{p^{k-1}} \mathfrak{I}_{p^{k-1} t}^{lu,rd,ru}\right|\nonumber\\& + I_{lu,ld,ru}(k) \left|\frac{1}{p^{k-1}} \mathfrak{I}_{p^{k-1} t}^{lu,ld,ru}\right| + I_{lu,ld,rd}(k) \left|\frac{1}{p^{k-1}} \mathfrak{I}_{p^{k-1} t}^{lu,ld,rd}\right| \nonumber\\
&- I_{lu,ld,rd,ru}(k) \left|\frac{1}{p^{k-1}} \mathfrak{I}_{p^{k-1} t}^{lu,ld,rd,ru}\right|.
\end{align}
Each $\left|\mathfrak{I}_t^{type}\right|$ will be monotonically increasing, positive, and continuous just like $|A_t|$, as shown in Equations (\ref{eq:intersectionh})-(\ref{eq:intersection4way}). Now define the functions $\tilde{A}(t) : [\alpha,p\alpha] \rightarrow \mathbb{R}$ and $\tilde{\mathfrak{I}}^{type}(t) : \left[\frac{\alpha}{p}, \alpha\right] \rightarrow \mathbb{R}$, where $\tilde{A}(t) = |A_t|$ whenever $t \in [\alpha,p\alpha]$, and $\tilde{\mathfrak{I}}^{type}(t) = |\mathfrak{I}_t^{type}|$ whenever $t \in \left[\frac{\alpha}{p}, \alpha\right]$ for each of the 9 different intersection types. By the domain-restriction, we have that Condition $C$ in the statement of Theorem \ref{thm:gammaconnected1}, which is that $\left|\left\{t \in \left[\frac{\alpha}{p},\alpha\right] : \left|\mathfrak{I}_t^{type}\right| \neq 0\right\}\right| > 0$, holds if and only if $\left|\left\{t \in \left[\frac{\alpha}{p},\alpha\right] : \tilde{\mathfrak{I}}^{type}(t) \neq 0\right\}\right| > 0$. \\

We can assume without loss of generality that $\alpha < \delta \leq p\alpha$. Theorem \ref{thm:scaling} then gives us the following tubular zeta function formula:
\begin{align}
\tilde{\zeta}_A(s)
&= \int_{\alpha}^{\delta} t^{s-3} |A_t| \,\mathrm{d}t + \sum_{k=1}^{\infty} \int_{\frac{\alpha}{p^k}}^{\frac{\alpha}{p^{k-1}}} t^{s-3} |A_t| \,\mathrm{d}t \nonumber\\
&= \int_{\alpha}^{\delta} t^{s-3} \tilde{A}(t) \,\mathrm{d}t + \sum_{k=1}^{\infty} \int_{\frac{\alpha}{p^k}}^{\frac{\alpha}{p^{k-1}}} t^{s-3} \left(m^k \left|\left(\frac{1}{p^k} A\right)_{p^k t}\right| - I_{h}(k) \left|\frac{1}{p^{k-1}} \mathfrak{I}_{p^{k-1} t}^{h}\right|\right.\nonumber\\&\left. - I_{v}(k) \left|\frac{1}{p^{k-1}} \mathfrak{I}_{p^{k-1} t}^{v}\right| - I_{ld,ru}(k) \left|\frac{1}{p^{k-1}} \mathfrak{I}_{p^{k-1} t}^{ld,ru}\right| - I_{lu,rd}(k) \left|\frac{1}{p^{k-1}} \mathfrak{I}_{p^{k-1} t}^{lu,rd}\right|\right.\nonumber\\&\left. + I_{ld,rd,ru}(k) \left|\frac{1}{p^{k-1}} \mathfrak{I}_{p^{k-1} t}^{ld,rd,ru}\right| + I_{lu,rd,ru}(k) \left|\frac{1}{p^{k-1}} \mathfrak{I}_{p^{k-1} t}^{lu,rd,ru}\right|\right.\nonumber\\&\left. + I_{lu,ld,ru}(k) \left|\frac{1}{p^{k-1}} \mathfrak{I}_{p^{k-1} t}^{lu,ld,ru}\right| + I_{lu,ld,rd}(k) \left|\frac{1}{p^{k-1}} \mathfrak{I}_{p^{k-1} t}^{lu,ld,rd}\right|\right.\nonumber\\&\left. - I_{lu,ld,rd,ru}(k) \left|\frac{1}{p^{k-1}} \mathfrak{I}_{p^{k-1} t}^{lu,ld,rd,ru}\right|\right) \,\mathrm{d}t\\
&= \int_{\alpha}^{\delta} t^{s-3} \tilde{A}(t) \,\mathrm{d}t + \sum_{k=1}^{\infty} \frac{1}{p^{ks}} \int_{\alpha}^{p\alpha} t^{s-3} \left(m^k |A_t|\right) \,\mathrm{d}t\nonumber\\&
+ \frac{1}{p^{(k-1)s}} \int_{\frac{\alpha}{p}}^{\alpha} t^{s-3} \left( - I_{h}(k) \left|\mathfrak{I}_t^{h}\right| - I_{v}(k) \left|\mathfrak{I}_t^{v}\right| - I_{ld,ru}(k) \left|\mathfrak{I}_t^{ld,ru}\right|\right.\nonumber\\&\left. - I_{lu,rd}(k) \left|\mathfrak{I}_t^{lu,rd}\right| + I_{ld,rd,ru}(k) \left|\mathfrak{I}_t^{ld,rd,ru}\right| + I_{lu,rd,ru}(k) \left|\mathfrak{I}_t^{lu,rd,ru}\right|\right.\nonumber\\&\left. + I_{lu,ld,ru}(k) \left|\mathfrak{I}_t^{lu,ld,ru}\right| + I_{lu,ld,rd}(k) \left|\mathfrak{I}_t^{lu,ld,rd}\right|\right.\nonumber\\&\left. - I_{lu,ld,rd,ru}(k) \left|\mathfrak{I}_t^{lu,ld,rd,ru}\right|\right) \,\mathrm{d}t\\
&= \int_{\alpha}^{\delta} t^{s-3} \tilde{A}(t) \,\mathrm{d}t + \sum_{k=1}^{\infty} \frac{m^k}{p^{ks}} \int_{\alpha}^{p\alpha} t^{s-3} \tilde{A}(t) \,\mathrm{d}t\nonumber\\&
- \frac{I_{h}(k)}{p^{(k-1)s}} \int_{\frac{\alpha}{p}}^{\alpha} t^{s-3} \tilde{\mathfrak{I}}^{h}(t) \,\mathrm{d}t - \frac{I_{v}(k)}{p^{(k-1)s}} \int_{\frac{\alpha}{p}}^{\alpha} t^{s-3} \tilde{\mathfrak{I}}^{v}(t) \,\mathrm{d}t\nonumber\\&
- \frac{I_{ld,ru}(k)}{p^{(k-1)s}} \int_{\frac{\alpha}{p}}^{\alpha} t^{s-3} \tilde{\mathfrak{I}}^{ld,ru}(t) \,\mathrm{d}t - \frac{I_{lu,rd}(k)}{p^{(k-1)s}} \int_{\frac{\alpha}{p}}^{\alpha} t^{s-3} \tilde{\mathfrak{I}}^{lu,rd}(t) \,\mathrm{d}t\nonumber\\&
+ \frac{I_{ld,rd,ru}(k)}{p^{(k-1)s}} \int_{\frac{\alpha}{p}}^{\alpha} t^{s-3} \tilde{\mathfrak{I}}^{ld,rd,ru}(t) \,\mathrm{d}t + \frac{I_{lu,rd,ru}(k)}{p^{(k-1)s}} \int_{\frac{\alpha}{p}}^{\alpha} t^{s-3} \tilde{\mathfrak{I}}^{lu,rd,ru}(t) \,\mathrm{d}t\nonumber\\&
+ \frac{I_{lu,ld,ru}(k)}{p^{(k-1)s}} \int_{\frac{\alpha}{p}}^{\alpha} t^{s-3} \tilde{\mathfrak{I}}^{lu,ld,ru}(t) \,\mathrm{d}t + \frac{I_{lu,ld,rd}(k)}{p^{(k-1)s}} \int_{\frac{\alpha}{p}}^{\alpha} t^{s-3} \tilde{\mathfrak{I}}^{lu,ld,rd}(t) \,\mathrm{d}t\nonumber\\&
- \frac{I_{lu,ld,rd,ru}(k)}{p^{(k-1)s}} \int_{\frac{\alpha}{p}}^{\alpha} t^{s-3} \tilde{\mathfrak{I}}^{lu,ld,rd,ru}(t) \,\mathrm{d}t. \label{eq:dusttheoremending}
\end{align}
Each of the $\tilde{A}(t), \tilde{\mathfrak{I}}^{type}(t)$ functions will be monotone, continuous, and nonnegative (though the intersection areas may be 0 for all or part of the interval $\left[\frac{\alpha}{p}, \alpha\right]$) since they are domain-restrictions of their respective $|A_t|$ and $\left|\mathfrak{I}_t^{type}\right|$ functions. Furthermore, $\alpha > 0$ by Theorem \ref{thm:thresholdexists}, so we have by Theorem \ref{thm:entireextension} that each of the integrals in Equation (\ref{eq:dusttheoremending}) will extend to entire functions. Note that if $\tilde{\mathfrak{I}}^{type}(t) = 0$ for almost every $t \in \left[\frac{\alpha}{p}, \alpha\right]$, then the corresponding $\displaystyle\int_{\frac{\alpha}{p}}^{\alpha} t^{s-3} \tilde{\mathfrak{I}}^{type}(t) \,\mathrm{d}t = 0$, completely canceling any poles which could come from its corresponding $\displaystyle \sum_{k=1}^{\infty} \frac{I_{type}(k)}{p^{(k-1)s}}$. Therefore, the only possible poles of $\tilde{\zeta}_A(s)$ are solutions to $p^s - m = 0$ (which originates from the meromorphic continuation of the geometric series $\displaystyle \sum_{k=1}^{\infty} \frac{m^k}{p^{ks}}$), as well as whichever poles appear when we go to compute $\displaystyle \sum_{k=1}^{\infty} \frac{I_{type}(k)}{p^{(k-1)s}}$ for each intersection type that has the property that $\left|\mathfrak{I}_t^{type}\right| \neq 0$ for some positive-measure set of $t \in \left[\frac{\alpha}{p}, \alpha\right]$.
\end{proof}
\newparagraph

\subsection{The Combinatorics} \label{chap:countingintersections}
With Theorem \ref{thm:gammaconnected1}, we have effectively reduced the task of finding the multiset of possible complex dimensions from one of analysis to one of combinatorics. We will now discuss the geometry and combinatorics involved in counting the intersection types present in $A_n$ for each $n \in \mathbb{N}$. Note that the proceeding combinatorial algorithm does not rely on the assumption that $|\mathfrak{I}_t^{type}| \neq 0$ for a positive-measure collection of $t \in \left[\frac{\alpha}{p},\alpha\right]$, although this assumption is necessary in order for the conclusions to be accurate. We will need to run the proceeding algorithm for each intersection type present, so we will drop the ``$(type)$'' unless we have multiple types in the same equation. \\

Starting with $A_1$, we can immediately count the number of intersections of that particular type by visually hand-counting, producing $I(1)$, as shown in Figure \ref{fig:keyexamplelevel1}.
\begin{figure}[H]
\includegraphics[scale=1]{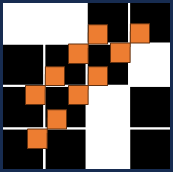}
\caption{The intersections of type $\{ld,ru\}$ present in $A_1$ are marked in orange for this particular Sierpi\'nski Carpet modification.}
\label{fig:keyexamplelevel1}
\end{figure}

As we proceed to $A_2$, we observe that $A_2$ consists of $m$ scaled-down copies of $A_1$, and thus $I(2) \geq m\cdot I(1)$. To complete the count, we will track how these intersections relate to the copies of $A_1$, labeling them $D(1)$ when the $A_1$ copies share a diagonal corner only, $H(1)$ when the $A_1$ copies share a horizontal edge, and $V(1)$ when the $A_1$ copies share a vertical edge. These $D, H$, and $V$ will themselves form sequences dependent on the finite construction level, as shown in Figure \ref{fig:keyexamplelevel2}.
\begin{figure}[H]
\includegraphics[scale=0.6]{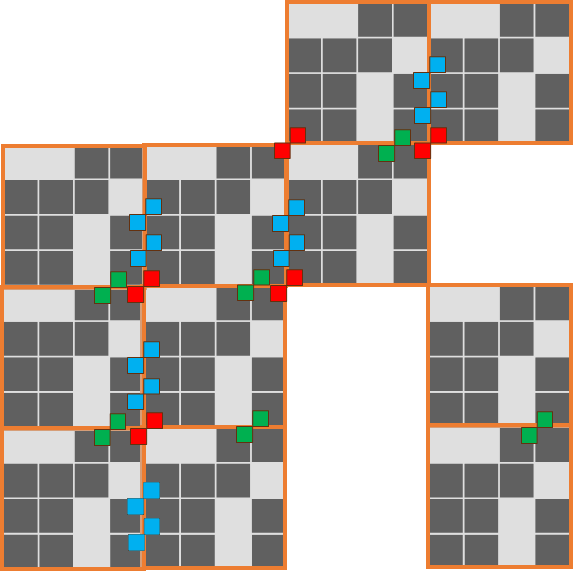}
\caption{The $D$ intersections are marked in red, $H$ in green, and $V$ in blue.}
\label{fig:keyexamplelevel2}
\end{figure}
Geometrically, these are the only locations where intersections can appear, so we conclude:
\begin{equation}
I(2) = m\cdot I(1) + D(1) + H(1) + V(1).
\end{equation}

Continuing to $A_3$, we observe that $A_3$ consists of $m^2$ scaled-down copies of $A_1$, and so $I(3) \geq m^2 \cdot I(1)$. Furthermore, $A_3$ consists of $m$ scaled-down copies of $A_2$, so $D(2) \geq m\cdot D(1)$, $H(2) \geq m\cdot H(1)$, and $V(2) \geq m\cdot V(1)$. Any unaccounted-for intersections will have properties $D, V$, or $H$, and will occur where $A_2$ copies and $A_1$ copies share edges or corners. We can then search where we had $D(1), V(1)$, and $H(1)$ intersections to determine particular replacement rules. For instance, we see from $A_3$ in Figure \ref{fig:keyexamplelevel3} that each $H$ got replaced with an $H$ and a $D$, each $V$ got replaced with 2 $V$'s and a $D$, and the $D$'s stayed as $D$'s.
\begin{figure}[H]
\includegraphics[scale=0.9]{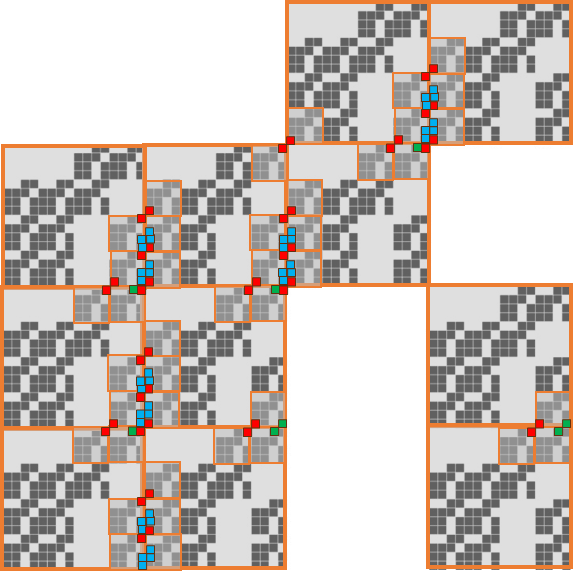}
\caption{The $D$ intersections are marked in red, $H$ in green, and $V$ in blue.}
\label{fig:keyexamplelevel3}
\end{figure}
From this, we can assemble the following facts, determining geometric multipliers for the $H$ and $V$ sequences from the replacement rules:
\begin{align}
I(3) &= m^2\cdot I(1) + D(2) + H(2) + V(2) \\
D(2) &= m\cdot D(1) + D(1) + H(1) + V(1) \\
H(2) &= m\cdot H(1) + 1 \cdot H(1) \\
V(2) &= m\cdot V(1) + 2 \cdot V(1).
\end{align}
As we proceed to $A_4$, we obtain the following via the same geometric replacement rules:
\begin{align}
I(4) &= m^3\cdot I(1) + D(3) + H(3) + V(3) \\
D(3) &= m^2\cdot D(1) + D(2) + H(2) + V(2) \\
H(3) &= m^2\cdot H(1) + 1 \cdot H(2) \\
V(3) &= m^2\cdot V(1) + 2 \cdot V(2).
\end{align}
In general for any $k \geq 3$, we have:
\begin{align}
I(k) &= m^{k-1}\cdot I(1) + D(k-1) + H(k-1) + V(k-1) \\
D(k) &= m^{k-1}\cdot D(1) + D(k-1) + d_H \cdot H(k-1) + d_V \cdot V(k-1) \\
H(k) &= m^{k-1}\cdot H(1) + h \cdot H(k-1) \\
V(k) &= m^{k-1}\cdot V(1) + v \cdot V(k-1),
\end{align}
where $d_H, d_V, h$, and $v$ are constants whose values depend on the fractal. In the example illustrated in Figures \ref{fig:keyexamplelevel1}, \ref{fig:keyexamplelevel2}, and \ref{fig:keyexamplelevel3}, we have $d_H = 1, d_V = 1, h = 1, v = 2$. The constants $h, v, d_H$, and $d_V$ can be computed via the geometric replacement rules, or by manually counting $D(1),D(2),H(1),H(2),V(1)$, and $V(2)$ and then computing the following:
\begin{align}
h &= \begin{cases}
\frac{H(2) - m\cdot H(1)}{H(1)}, & H(1) \neq 0 \\
0, & H(1) = 0
\end{cases} \\
v &= \begin{cases}
\frac{V(2) - m\cdot V(1)}{V(1)}, & V(1) \neq 0 \\
0, & V(1) = 0
\end{cases} \\
d_H \cdot H(1) + d_V \cdot V(1) &= D(2) - (m+1) \cdot D(1).
\end{align}
We will define $d_H = 0$ when $H(1) = 0$, and likewise $d_V = 0$ when $V(1) = 0$. From this, we can use algebraic methods to determine an explicit formula for $I(k)$:
\begin{align}
I(k) &= d_H H(1) \frac{h - h^{k-1} - m + m h^{k-1} + m^{k-1} - h m^{k-1}}{(m-1)(h-m)(h-1)}\nonumber\\& %Use (m^k + m(k-2) - m^2(k-1))/(m(m-1)^2) if h=1
+ d_V V(1) \frac{v - v^{k-1} - m + m v^{k-1} + m^{k-1} - v m^{k-1}}{(m-1)(v-m)(h-1)}\nonumber\\& %Use (m^k + m(k-2) - m^2(k-1))/(m(m-1)^2) if v=1
+ H(1) \frac{m^{k-1} - h^{k-1}}{m-h}
+ V(1) \frac{m^{k-1} - v^{k-1}}{m-v}
+ D(1) \frac{m^{k-1} - 1}{m-1}
+ I(1) m^{k-1}. \label{eq:zerotothezero}
\end{align}
In some instances, Equation (\ref{eq:zerotothezero}) will need to be interpreted differently to ensure the counts are accurate. For instance, in the case that $h = 0$, $H(1) \neq 0$, $V(1) = 0$, $D(1) = 0$, and $d_H = 0$, Equation (\ref{eq:zerotothezero}) will appear as follows:
\begin{align}
I(k) &= H(1) \frac{m^{k-1} - 0^{k-1}}{m}
+ I(1) m^{k-1}. \label{eq:hgetszerotothezero}
\end{align}
When we compute $I(1)$ in Equation (\ref{eq:hgetszerotothezero}), we will treat the resulting ``$0^0$" as\\ $\displaystyle \lim_{x\rightarrow 0^+} x^x = 1$ to ensure we count everything correctly. \\

Furthermore, the expression of Equation (\ref{eq:zerotothezero}) appears to have a singularity when $h = 1$ and $H(1) \neq 0$. To resolve this, in a case where $h=1, H(1) \neq 0$, we will use $\displaystyle \lim_{h\rightarrow 1} I(k)$, which produces the following:
\begin{align}
\lim_{h\rightarrow 1} I(k) &= \frac{d_H H(1)}{m-1} \left(\left(\sum_{j=1}^{k-1} m^j\right) - (m-2)\right)\nonumber\\&
+ d_V V(1) \frac{v - v^{k-1} - m + v^{k-1} m + m^{k-1} - v m^{k-1}}{(m-1)(v-m)(v-1)}\nonumber\\& %Use (m^k + m(k-2) - m^2(k-1))/(m(m-1)^2) if v=1
+ H(1) \left(\sum_{j=0}^{k-1} m^j\right)
+ V(1) \frac{m^{k-1} - v^{k-1}}{m-v}
+ D(1) \frac{m^{k-1} - 1}{m-1}
+ I(1) m^{k-1}.
\end{align}
Likewise, in a case where $v=1, V(1) \neq 0$, we will use $\displaystyle \lim_{v\rightarrow 1} I(k)$. \\

Based on the nature of the geometric replacement rules, we know that $h$ and $v$ must be integers of value $\leq p$ since the relevant ``duplicating'' only occurs along a single edge of a square, which will be divided into a discrete $p \times p$ grid at each step of the finite construction. Furthermore, the multipliers $h$ and $v$ cannot be negative because there is no geometric mechanism that would allow for that to occur. Thus, we have that
\begin{equation}
h,v \in \{0,1,2,\dots,p\}.
\end{equation}
We can also see from the geometric replacement rules that
\begin{equation}
d_H, d_V \in \{0,1\},
\end{equation}
since an $H$ or $V$ intersection can each produce at most one $D$ intersection at the corners of the squares where they generate. Moreover, it is evident that if $m \leq p$ and the carpet is of dust type, then it will be completely dusty (and thus there would be no intersection sequences), and thus we can never have $v = m$ or $h = m$. We also ignore cases where $m = 1$ as they do not fit the definition of a dust type Sierpi\'nski Carpet modification. Thus, there are no other singularities in the statement of Equation (\ref{eq:zerotothezero}), and the resulting sequence $I(k)$ will contain only nonnegative integers, as would be expected from a count of geometric properties. \\

Note that having $h = p$ or $v = p$ is possible in fractals such as the one shown in Figure \ref{fig:cantorstripes4x4} - notice how when we move from $A_2$ to $A_3$, the combination of 3 $H$'s gets replaced with 12 $H$'s and 3 $D$'s, which indicates that the multiplier for the $H$ sequence is 4, the same as $p$. It is also possible for $h = 0$ or $v = 0$, as demonstrated in Figures \ref{fig:bluevanishes4x4} and \ref{fig:bluevanishesredstays4x4}.

\begin{figure}
\boxed{\includegraphics[scale=1.47]{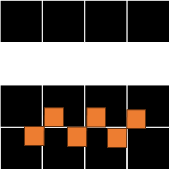}}
\boxed{\includegraphics[scale=0.35]{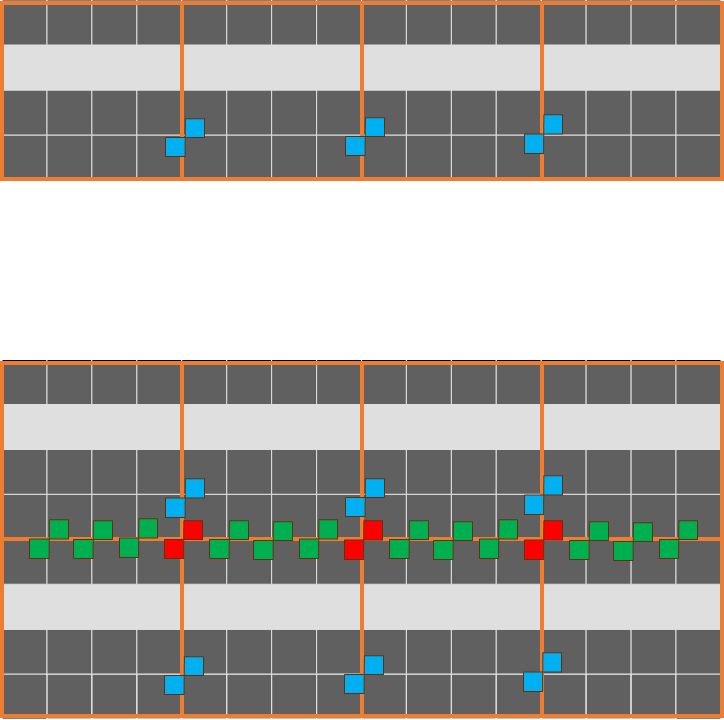}}\\
\boxed{\includegraphics[scale=0.7]{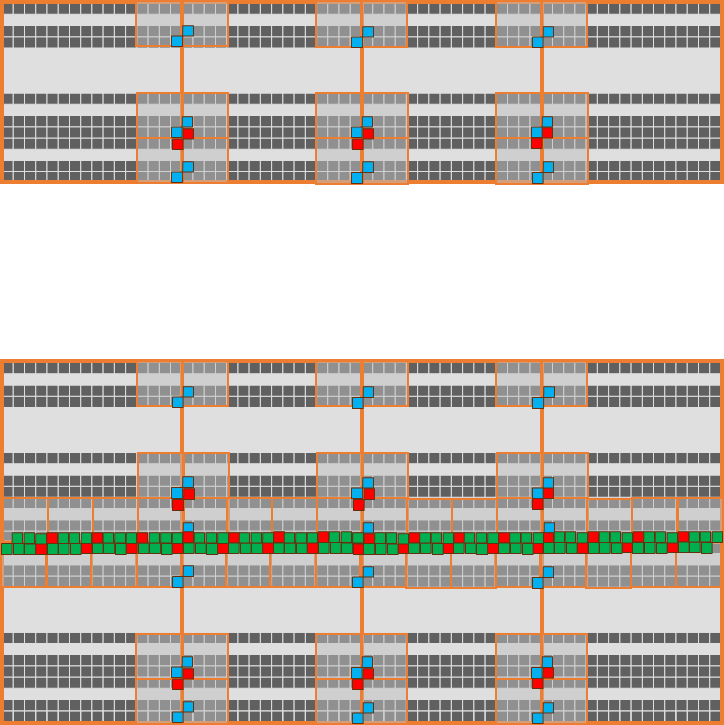}}
\caption{For the $\{ld,ru\}$ intersection type, the $H$ sequence has duplication factor 4 and the $V$ sequence has duplication factor 3, with $d_H = 1$ and $d_V = 1$.}
\label{fig:cantorstripes4x4}
\end{figure}

\begin{figure}
\boxed{\includegraphics[scale=1.47]{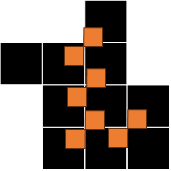}}
\boxed{\includegraphics[scale=0.35]{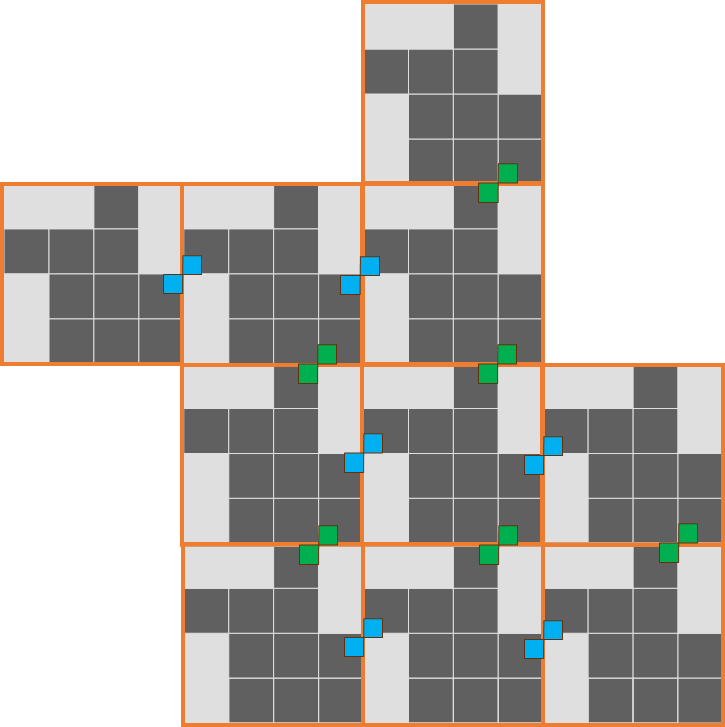}}\\
\boxed{\includegraphics[scale=0.7]{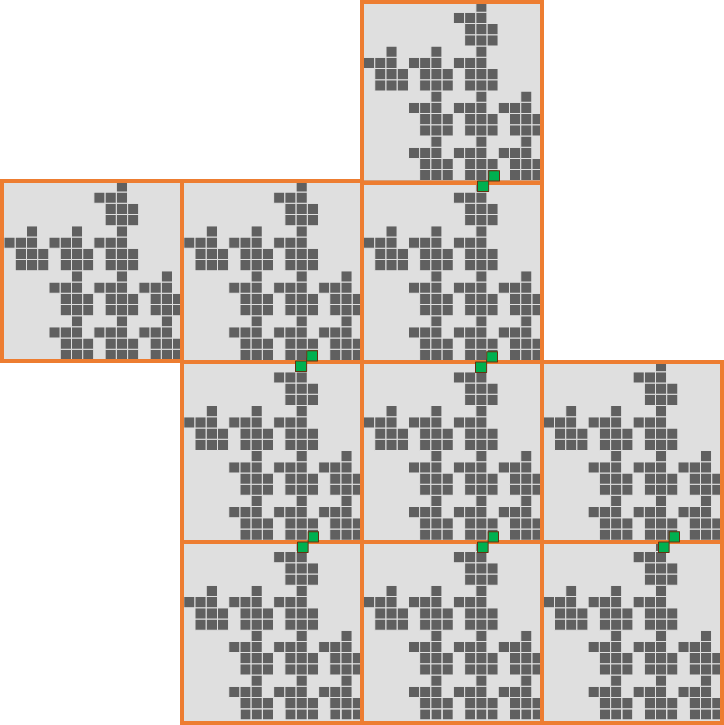}}
\caption{For the $\{ld,ru\}$ intersection type, the $H$ sequence has duplication factor 1 and the $V$ sequence has duplication factor 0, with $d_H = 0$ and $d_V = 0$.}
\label{fig:bluevanishes4x4}
\end{figure}
\begin{figure}
\boxed{\includegraphics[scale=1.47]{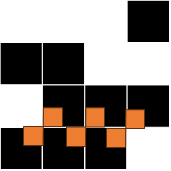}}
\boxed{\includegraphics[scale=0.35]{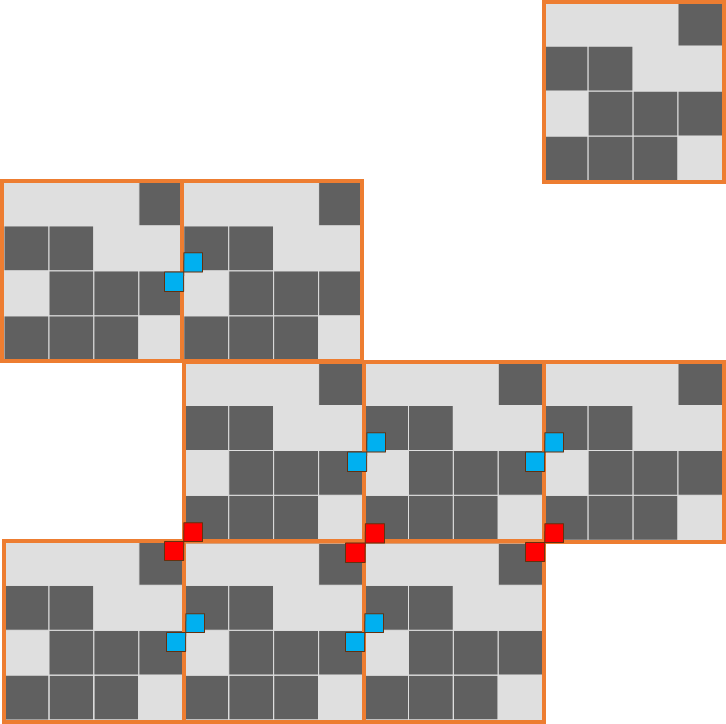}}\\
\boxed{\includegraphics[scale=0.7]{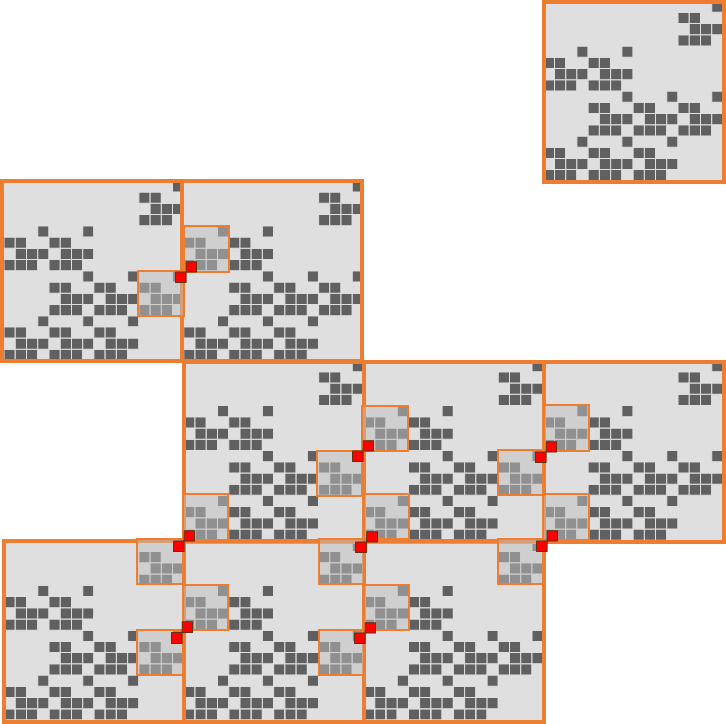}}
\caption{For the $\{ld,ru\}$ intersection type, the $V$ sequence has duplication factor 0 and $d_V = 1$.}
\label{fig:bluevanishesredstays4x4}
\end{figure}

\subsection{General Dust Theorem} \label{chap:theorem42}
Now that we have an algorithm to compute $I_{type}(k)$ for each intersection type, we can complete the statement and proof of Theorem \ref{thm:gammaconnected1}, producing an algorithm to compute the possible complex dimensions for every dust type Sierpi\'nski Carpet modification.

\begin{theorem}[General Dust Theorem in $\mathbb{R}^2$]\label{thm:gammaconnected2}
Let $A$ be a dust type Sierpi\'nski Carpet modification constructed by keeping $m \geq 2$ squares from the $p \times p$ grid. Let $\alpha > 0$ be the maximal number such that $A_{\alpha}$ is not path-connected, as given by Theorem \ref{thm:thresholdexists}. Then $\tilde{\zeta}_A$ has a meromorphic continuation to all of $\mathbb{C}$ and
\begin{align}
\mathscr{P}(A) &\subseteq \left\{\log_p(m) + \frac{2\pi i j}{\ln(p)} : j \in \mathbb{Z}\right\}\nonumber\\&\cup \bigcup_{(type) \text{ s.t. Condition } C \text{ holds}} \left(\bigcup_{r \in R_{type}} \left\{r + \frac{2\pi i j}{\ln(p)} : j \in \mathbb{Z}\right\}\right),
\end{align}
where $C$ is the condition that $\left|\left\{t \in \left[\frac{\alpha}{p},\alpha\right] : \left|\mathfrak{I}_t^{type}\right| \neq 0\right\}\right| > 0$, and where $R_{type}$ is a multiset of real numbers, such that the multiplicity of an element of $R_{type}$ corresponds to the order of the pole. For each intersection type, $R_{type}$ is determined by Table \ref{table:allthepoles}, computed by observing the numbers of intersections in $A_1, A_2$, and $A_3$.

\begin{table}[!h]
\begin{center}
\begin{tabular}{|c||c|c|c|}
\hline
\textbf{Multipliers $h$ and $v$} & \textbf{Case $D(1) \neq 0$} & \textbf{Case $D(1) = 0$} \\
\hline
\hline
$h \neq v$, both $\notin \{0,1\}$ & $\{\log_p(h), \log_p(v), \log_p(1)\}$ & $\{\log_p(h), \log_p(v)\}$ \\
\hline
$h = v \notin \{0,1\}$ & $\{\log_p(h), \log_p(1)\}$ & $\{\log_p(h)\}$ \\
\hline
$h \notin \{0,1\}; v = 1$ & $\{\log_p(h), \log_p(1), \log_p(1)\}$ & $\{\log_p(h), \log_p(1)\}$ \\
\hline
$h = 1; v \notin \{0,1\}$ & $\{\log_p(v), \log_p(1), \log_p(1)\}$ & $\{\log_p(v), \log_p(1)\}$ \\
\hline
$h = v = 1$ & $\{\log_p(1), \log_p(1)\}$ & $\{\log_p(1)\}$ \\
\hline
$h \notin \{0,1\}; v = 0$ & $\{\log_p(h), \log_p(1)\}$ & $\{\log_p(h)\}$ \\
\hline
$h = 0; v \notin \{0,1\}$ & $\{\log_p(v), \log_p(1)\}$ & $\{\log_p(v)\}$ \\
\hline
$h = v = 0$ & $\{\log_p(1)\}$ & $\emptyset$ \\
\hline
\end{tabular}
\caption{All possibilities for $R_{type}$ given a particular intersection type.}
\label{table:allthepoles}
\end{center}
\end{table}
\end{theorem}

\begin{proof}
The proof of Theorem \ref{thm:gammaconnected1} shows how for each intersection type that meets Condition $C$, the meromorphic continuation of $\displaystyle \sum_{k=1}^{\infty} \frac{I_{type}(k)}{p^{(k-1)s}}$ gives us additional poles. The multiset of the real components of these poles (with multiplicity of the element corresponding to the order of the pole) is precisely $R_{type}$ by definition. Furthermore, we computed the following in Equation (\ref{eq:zerotothezero}):
\begin{align}
I(k) &= d_H H(1) \frac{h - h^{k-1} - m + h^{k-1} m + m^{k-1} - h m^{k-1}}{(m-1)(h-m)(h-1)}\nonumber\\& %Use (m^k + m(k-2) - m^2(k-1))/(m(m-1)^2) if h=1
+ d_V V(1) \frac{v - v^{k-1} - m + v^{k-1} m + m^{k-1} - v m^{k-1}}{(m-1)(v-m)(h-1)}\nonumber\\& %Use (m^k + m(k-2) - m^2(k-1))/(m(m-1)^2) if v=1
+ H(1) \frac{m^{k-1} - h^{k-1}}{m-h}
+ V(1) \frac{m^{k-1} - v^{k-1}}{m-v}
+ D(1) \frac{m^{k-1} - 1}{m-1}
+ I(1) m^{k-1}.
\end{align}
Computing the relevant infinite sum using the formula for geometric series, we have the following:
\begin{align}
\sum_{k=1}^{\infty} \frac{I(k)}{p^{ks}} &= \frac{d_H H(1)}{(p^s-1)(p^s-h)(p^s-m)}
+ \frac{d_V V(1)}{(p^s-1)(p^s-v)(p^s-m)}\nonumber\\&
+ \frac{H(1)}{(p^s-h)(p^s-m)}
+ \frac{V(1)}{(p^s-h)(p^s-m)}
+ \frac{D(1)}{(p^s-1)(p^s-m)}
+ \frac{I(1)}{p^s-m}.
\end{align}
Thus, we see that the poles of the meromorphic continuation of $\displaystyle \sum_{k=1}^{\infty} \frac{I(k)}{p^{(k-1)s}} = p^s \sum_{k=1}^{\infty} \frac{I(k)}{p^{ks}}$ correspond to what is listed in Table \ref{table:allthepoles}.
\end{proof}
\newparagraph

\subsection{Worked Example: Modified Cantor Grill} \label{chap:cantorstripesathomeexample}
%Worked Example for Cantor Grill (Cantor Stripes): Equation 2.2.99 (page 140) gets dimensions at 0, 1, log_3(2), log_3(6) [a = 1/3, d = 1], our result will be way less precise but should still get the point across
Consider the Sierpi\'nski Carpet modification constructed on the $3\times3$ grid given by Figure \ref{fig:cantorstripesathomepicture}.
\begin{figure}[H]
\includegraphics[scale=0.3]{Images/NEquals0}
\includegraphics[scale=0.3]{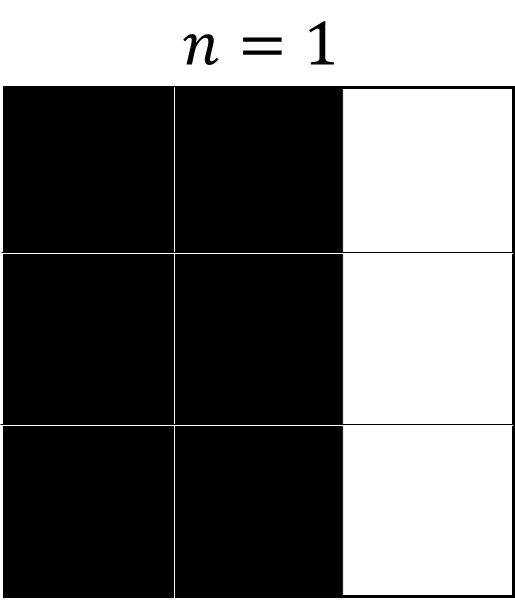}
\includegraphics[scale=0.3]{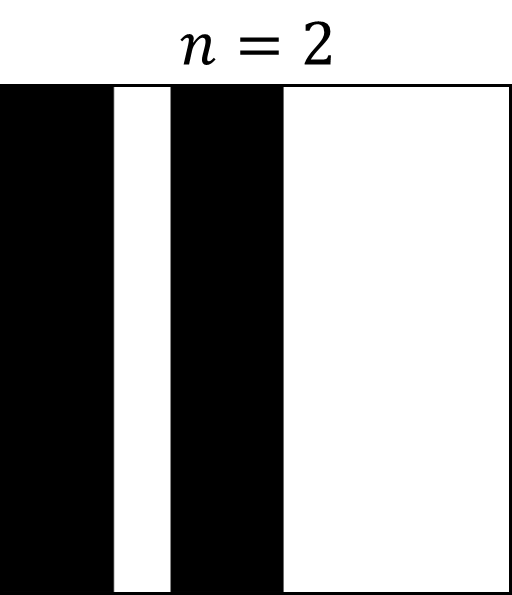}\\
\includegraphics[scale=0.3]{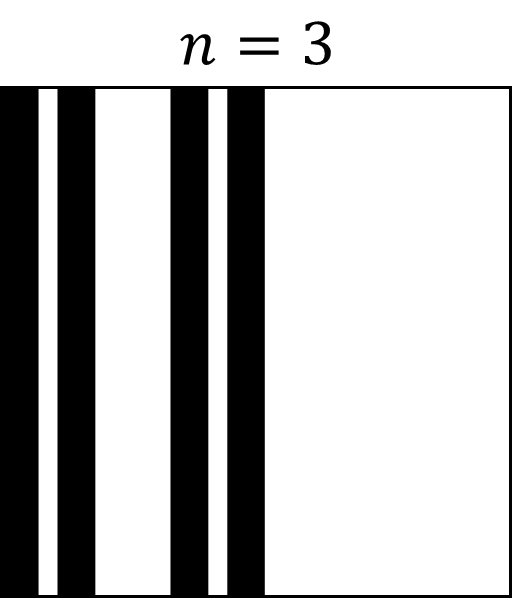}
\includegraphics[scale=0.3]{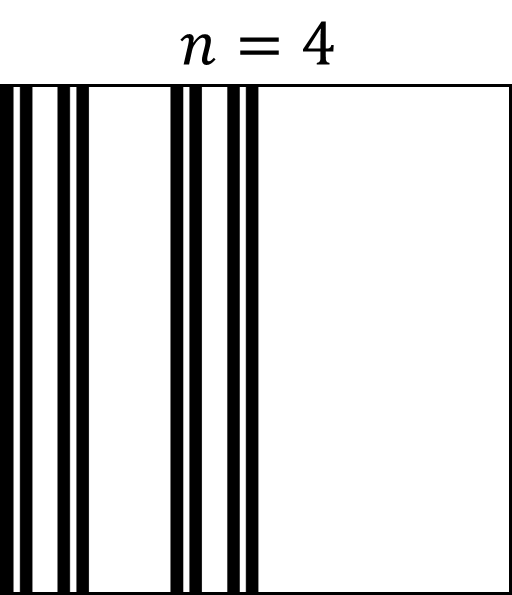}
\includegraphics[scale=0.3]{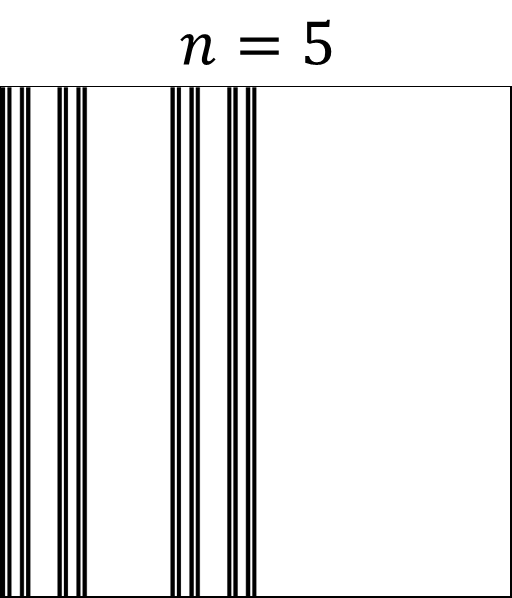}\\
\caption{The first 5 steps of the construction of a Sierpi\'nski Carpet modification, which we call a ``modified Cantor Grill''.}
\label{fig:cantorstripesathomepicture}
\end{figure}
In this carpet, $m = 6$ and $p = 3$. We will first compute the complex dimensions of this modified Cantor Grill using Theorem \ref{thm:gammaconnected2}. We can clearly see that the largest value $\alpha \in [0,\sqrt{2}]$ that makes $A_{\alpha}$ not path-connected (as guaranteed by Theorem \ref{thm:thresholdexists}) is $\displaystyle \alpha = \frac{1}{2} \cdot \sum_{k=2}^{\infty} \frac{1}{3^k} = \frac{1}{12}$. We can then observe how the scaled copies of $A_t$ behave for $t \in [\alpha,3\alpha]$ and $t \in \left[\frac{\alpha}{3}, \alpha\right]$, depicted in Figure \ref{fig:cantorstripesthresholds}.
\begin{figure}[H]
\boxed{\includegraphics[scale=0.3]{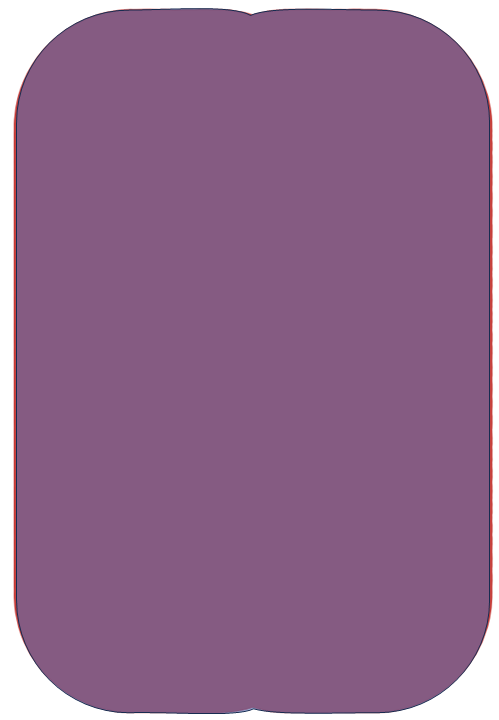}}
\boxed{\includegraphics[scale=0.3]{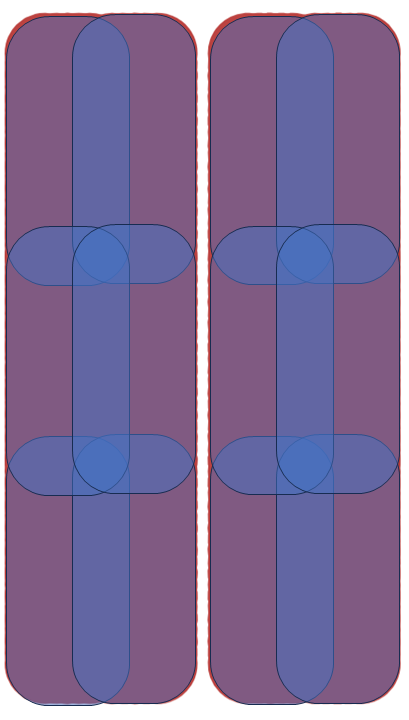}}
\boxed{\includegraphics[scale=0.3]{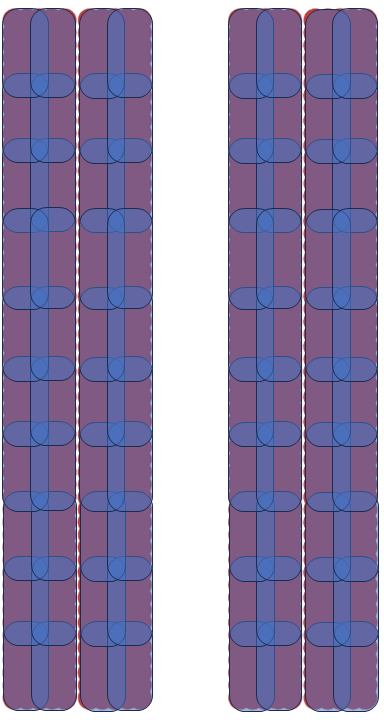}}
\caption{Left to right: the tubular neighborhood $A_t$ for $t = 3\alpha$, $t = \alpha$, and $t = \frac{\alpha}{3}$, with scaled copies of $A_{3\alpha}$ marked in blue.}
\label{fig:cantorstripesthresholds}
\end{figure}
From Figure \ref{fig:cantorstripesthresholds}, we can see that $\left|\mathfrak{I}_t^{type}\right| > 0$ for all $t \in \left[\frac{\alpha}{3}, \alpha\right]$ for each of the 9 different intersection types. We can now compute the counts visually on $A_1, A_2$, and $A_3$, as shown in Figure \ref{fig:cantorstripesathomeintersections} and Table \ref{table:cantorstripesathomeintersections}.

\begin{figure}[H]
\boxed{\includegraphics[scale=0.15]{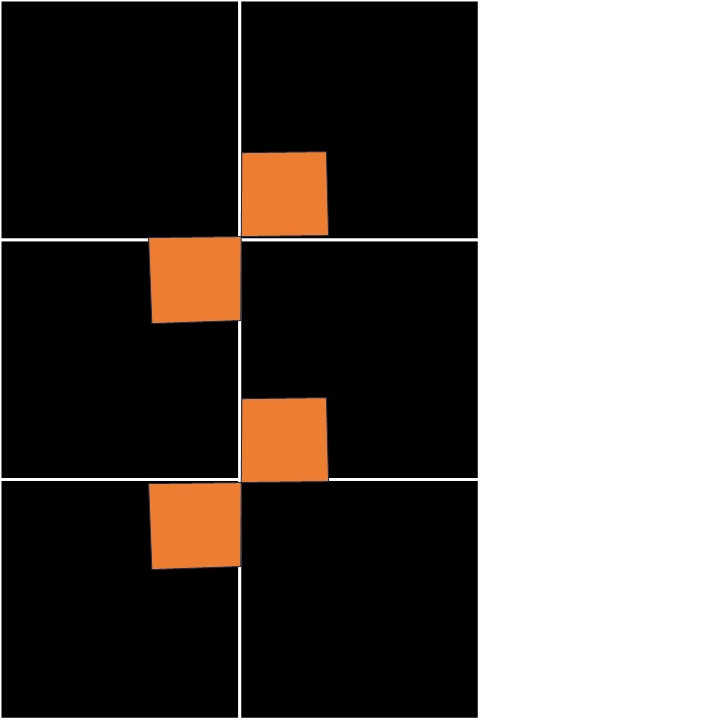}}
\boxed{\includegraphics[scale=0.15]{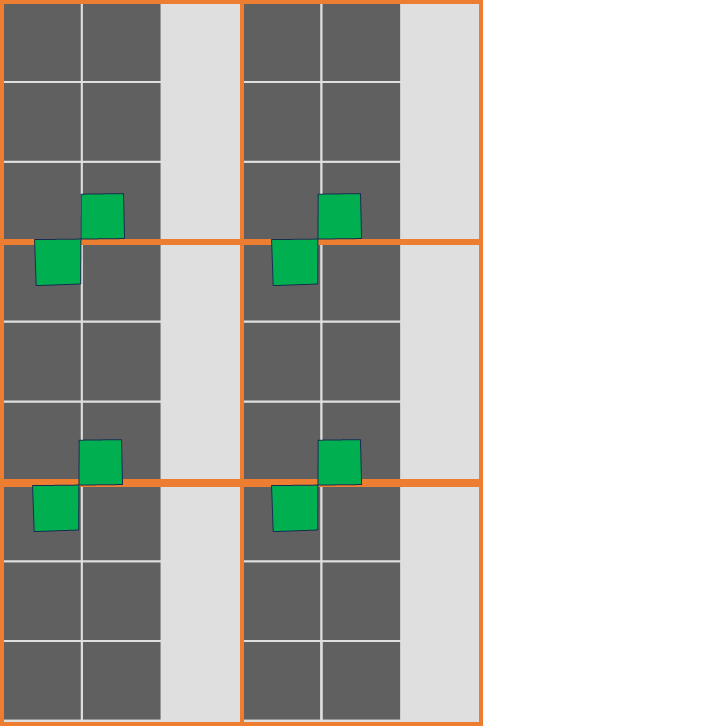}}
\boxed{\includegraphics[scale=0.4]{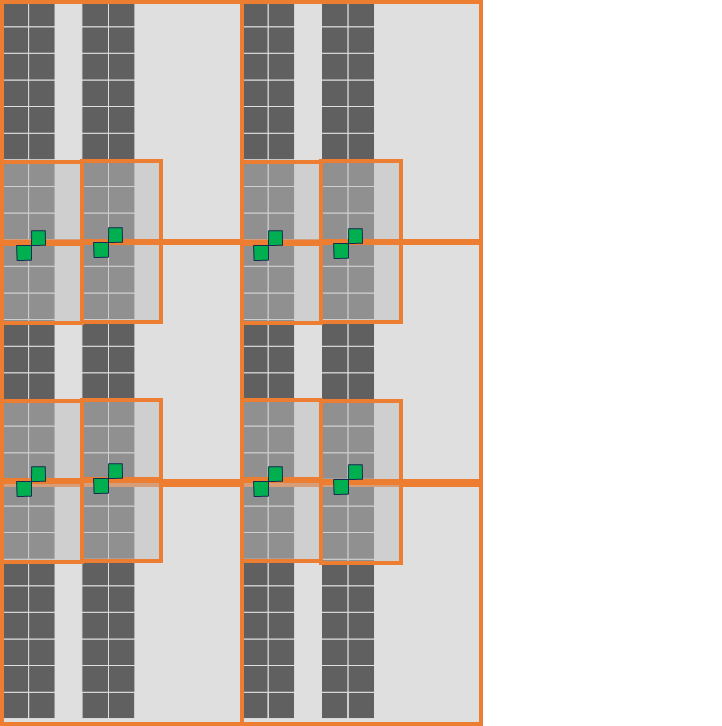}}
\caption{Visual display of $I_{ld,ru}(1)$ (orange in $A_1$), and the relevant parts of $H_{ld,ru}(k)$ that need to be counted (green in $A_2$ and $A_3$). Note that $V_{ld,ru}(k) = 0$ and $D_{ld,ru}(k) = 0$ for all $k \in \mathbb{N}$.}
\label{fig:cantorstripesathomeintersections}
\end{figure}

\begin{table}[!h]
\begin{center}
\begin{tabular}{|c||c|c|c|c|}
\hline
\textbf{Intersection Type} & \textbf{$I_{type}(k)$ for $k=1,2,3,4$} & \textbf{$H_{type}(1)$} & \textbf{$H_{type}(2)$} & \textbf{$h_{type}$} \\
\hline
\hline
$h$ & 4,32,208,1280 & 8 & 64 & 2 \\
\hline
$v$ & 3,18,108,648 & 0 & 0 & 0 \\
\hline
$lu,rd$ & 2,16,104,640 & 4 & 32 & 2 \\
\hline
$ld,ru$ & 2,16,104,640 & 4 & 32 & 2 \\
\hline
$lu,ld,ru$ & 2,16,104,640 & 4 & 32 & 2 \\
\hline
$lu,rd,ru$ & 2,16,104,640 & 4 & 32 & 2 \\
\hline
$lu,ld,rd$ & 2,16,104,640 & 4 & 32 & 2 \\
\hline
$ld,rd,ru$ & 2,16,104,640 & 4 & 32 & 2 \\
\hline
$lu,ld,rd,ru$ & 2,16,104,640 & 4 & 32 & 2 \\
\hline
\end{tabular}
\caption{The relevant values of $I_{type}$ and $H_{type}$ for each intersection type. Note that $V_{type}(1) = 0$ and $D_{type}(1) = 0$ for each intersection type.}
\label{table:cantorstripesathomeintersections}
\end{center}
\end{table}
From Table \ref{table:cantorstripesathomeintersections}, we obtain the following formulas for $I_{type}(k)$:
\begin{align}
I_h(k) &= H_h(1) \frac{m^{k-1} - h_h^{k-1}}{m-h_h} + I_h(1) m^{k-1} = 6^k - 2^k, \\
I_v(k) &= I_v(1) m^{k-1} = 3 \cdot 6^{k-1}, \\
I_{lu,rd}(k) &= H_{lu,rd}(1) \frac{m^{k-1} - h_{lu,rd}^{k-1}}{m-h_{lu,rd}} + I_{lu,rd}(1) m^{k-1} = 3 \cdot 6^{k-1} - 2^{k-1}, \\
I_{ld,ru}(k) &= H_{ld,ru}(1) \frac{m^{k-1} - h_{ld,ru}^{k-1}}{m-h_{ld,ru}} + I_{ld,ru}(1) m^{k-1} = 3 \cdot 6^{k-1} - 2^{k-1}, \\
\vdots\nonumber\\
I_{lu,ld,rd,ru}(k) &= H_{lu,ld,rd,ru}(1) \frac{m^{k-1} - h_{lu,ld,rd,ru}^{k-1}}{m-h_{lu,ld,rd,ru}} + I_{lu,ld,rd,ru}(1) m^{k-1}\nonumber\\&= 3 \cdot 6^{k-1} - 2^{k-1}.
\end{align}
Proceeding as in the proof of Theorem \ref{thm:gammaconnected2} yields:
\begin{equation}
\mathscr{P}(A) \subseteq \left\{\log_3(2) + \frac{2\pi i j}{\ln(3)}, \log_3(6) + \frac{2\pi i j}{\ln(3)} : j \in \mathbb{Z}\right\}.
\end{equation}

We can compute the complex dimensions in another way. First, consider a modified Cantor Set constructed by starting with the interval $[0,1] \subseteq \mathbb{R}$, dividing it into thirds, and then removing the highest-valued third. We will define it as follows:
\begin{align}
\mathfrak{C}^* &= [0,1] \backslash \left(\left(\frac{2}{3},1\right] \cup
\left(\left(\frac{2}{9},\frac{1}{3}\right] \cup \left(\frac{5}{9},\frac{2}{3}\right]\right)\right.\nonumber\\&\left. \cup
\left(\left(\frac{2}{27},\frac{1}{9}\right] \cup \left(\frac{5}{27},\frac{2}{9}\right] \cup \left(\frac{11}{27},\frac{4}{9}\right] \cup \left(\frac{14}{27},\frac{5}{9}\right]\right) \cup\dots\right).
\end{align}
Assume without loss of generality that $\delta > \frac{1}{12}$ (Corollary \ref{thm:anydelta} justifies doing this). We can compute the tubular zeta function of $\mathfrak{C}^*$ embedded in $\mathbb{R}^1$ by directly making use of the definition. For values $t \in (\frac{1}{12},\delta]$, we obtain:
\begin{equation}
|\mathfrak{C}_t^*| = \left(1 - \sum_{j=1}^{\infty} \frac{1}{3^j}\right) + 2t = \frac{1}{2} + 2t.
\end{equation}
When considering $t \in \left(\frac{1}{12\cdot3^k}, \frac{1}{12\cdot3^{k-1}}\right]$ for each $k \in \mathbb{N}$, we obtain the following:
\begin{equation}
|\mathfrak{C}_t^*| = 2^k \left(\left(\frac{1}{3^k} - \sum_{j=k+1}^{\infty} \frac{1}{3^j}\right) + 2t\right) = 2^k \left(\frac{1}{2\cdot3^k} + 2t\right).
\end{equation}
By properties of Lebesgue integration, we can partition the interval $[0,\delta]$ into a countable number of disjoint subintervals, which allows us to compute the tubular zeta function as follows:
\begin{align}
\tilde{\zeta}_{\mathfrak{C}^*}(s) &= \int_0^{\delta} t^{s-2} |\mathfrak{C}_t^*| \,\mathrm{d}t \\
&= \int_{\frac{1}{12}}^{\delta} t^{s-2} \left(\frac{1}{2} + 2t\right) + \sum_{k=1}^{\infty} \int_{\frac{1}{12\cdot3^k}}^{\frac{1}{12\cdot3^{k-1}}} t^{s-2} \left(2^k \left(\frac{1}{2\cdot3^k} + 2t\right)\right) \,\mathrm{d}t \\
&= \frac{2\delta^s - \frac{2}{12^s}}{s} + \frac{\frac{1}{2}\delta^{s-1} - \frac{6}{12^s}}{s-1}
+ \sum_{k=1}^{\infty} \frac{2^k}{3^{(k+1)s}} \left(\frac{1 - 4s + 3^s (2s-1)}{s(s-1)}\right) \\
&= \frac{2\delta^s - \frac{2}{12^s}}{s} + \frac{\frac{1}{2}\delta^{s-1} - \frac{6}{12^s}}{s-1}
+ \frac{2 - 8s + 3^s (4s-2)}{3^s (3^s-2) s(s-1)}. \label{eq:onedzeta}
\end{align}
After meromorphic extension to all of $\mathbb{C}$, we see that $\tilde{\zeta}_{\mathfrak{C}^*}(s)$ has removable singularities at $s = 0$ and $s = 1$, and it has simple poles at all solutions to the equation $3^s - 2 = 0$. Thus, we have that
\begin{equation}
\mathscr{P}(\mathfrak{C}^*) = \left\{\log_3(2) + \frac{2\pi i j}{\ln(3)} : j \in \mathbb{Z}\right\}.
\end{equation}

Note that the highest real part of $\mathscr{P}(\mathfrak{C}^*)$ is $\log_3(2) < 1$, and thus we can apply \cite[Theorem 4.7.3]{lapidusbook} to see that embedding $\mathfrak{C}^*$ into $\mathbb{R}^2$ yields the same
\begin{equation}
\mathscr{P}\left(\mathfrak{C}_{\mathbb{R}^2}^*\right) = \left\{\log_3(2) + \frac{2\pi i j}{\ln(3)} : j \in \mathbb{Z}\right\}.
\end{equation}
Lastly, we have by definition that $A = \mathfrak{C}^* \times [0,1]$. Thus, by \cite[Lemma 2.2.31]{lapidusbook}, we see that $\tilde{\zeta}_A(s) = \tilde{\zeta}_{\mathfrak{C}^*}(s-1) + \tilde{\zeta}_{\mathfrak{C}_{\mathbb{R}^2}^*}(s)$, from which we conclude that
\begin{equation}
\mathscr{P}(A) = \left\{\log_3(2) + \frac{2\pi i j}{\ln(3)}, \log_3(6) + \frac{2\pi i j}{\ln(3)} : j \in \mathbb{Z}\right\}.
\end{equation}
This demonstrates that Theorem \ref{thm:gammaconnected2} will not contradict known results in the literature, though it may give less information than what was previously known (in this case giving that $\mathscr{P}(A) \subseteq \left\{\log_3(2) + \frac{2\pi i j}{\ln(3)}, \log_3(6) + \frac{2\pi i j}{\ln(3)} : j \in \mathbb{Z}\right\}$ and unable to determine that the two sets are in fact equal). In other cases, such as discussed in Section \ref{chap:cantordustexample}, Theorems \ref{thm:alphaconnected} and \ref{thm:gammaconnected2} may give a smaller set of possible poles than what was previously known in the literature.

\section{Future Study} \label{chap:futurestudy}
We discuss how the work in this paper can be adapted to work with ``polyhedral zeta functions'' (Section \ref{chap:polyhedralzetafunctions}). We describe limitations in applying the methods to non-dust type Sierpi\'nski Carpet modifications (Section \ref{chap:nondustcarpets}) and Bedford--McMullen Carpets (Section \ref{chap:bedfordmcmullencarpets}). Lastly, we discuss an extension of the ``completely dusty'' case into $\mathbb{R}^N$ for any $N \in \mathbb{N}$, culminating in Theorem \ref{thm:alphaconnectedRN} (Section \ref{chap:mengersponges}).

\subsection{Polyhedral Zeta Functions} \label{chap:polyhedralzetafunctions}
Recent developments by Claire David and Michel L. Lapidus (\cite{polyhedralzetafunctions}, \cite{kochcurvepolyhedral}, \cite{weierstrasscurve1}, \cite{weierstrasscurve2}, and others) give rise to ``polyhedral zeta functions'', which compute the exact set of complex dimensions by utilizing a ``polyhedral neighborhood'' (as opposed to the tubular neighborhood). Currently, these polyhedral zeta functions have been applied to continuous curves. In order to apply them to Sierpi\'nski Carpets and modifications, we would need to adapt those methods to work in cases where there are no continuous non-constant curves in the limit (such as the Ternary Cantor Set). It would be interesting to use these new methods to compute the exact complex dimensions where possible.
\newparagraph

%%%Just briefly talk about it
%Question for Lapidus: what would the polyhedral neighborhoods look like for something that's not a curve full of vertices? For example, what would it look like for Cantor Dust?
    %Potentially the methods could more directly apply to the Branch Types since their boundaries will be one continuous curve...?

\subsection{Non-Dust Sierpi\'nski Carpet Modifications} \label{chap:nondustcarpets}
%This does NOT apply whatsoever, it's doomed from the start, need totally different methods
The combinatorial methods described in Section \ref{chap:computingdust} can still be applied to non-dust type Sierpi\'nski Carpet modifications, but the results may not be the possible set of complex dimensions. To see the distinction, consider the Sierpi\'nski Carpet modification on the $3\times3$ grid given by Figure \ref{fig:suscarpet}.
\begin{figure}[H]
\includegraphics[scale=0.3]{Images/NEquals0}
\includegraphics[scale=0.3]{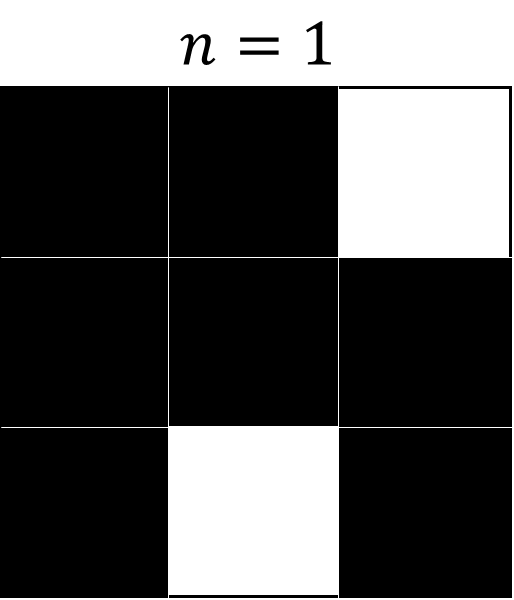}
\includegraphics[scale=0.3]{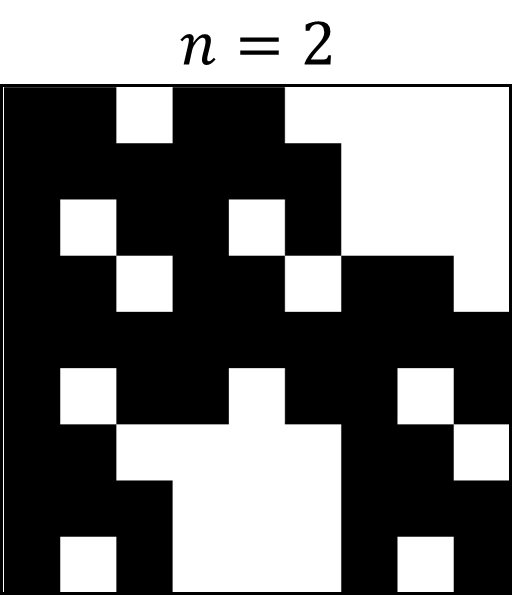}\\
\includegraphics[scale=0.3]{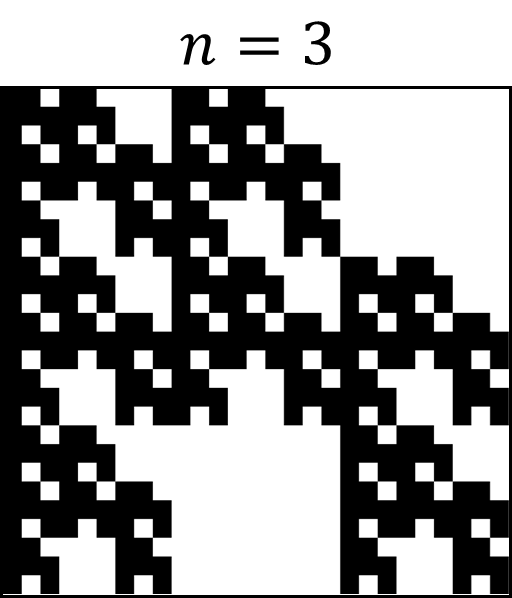}
\includegraphics[scale=0.3]{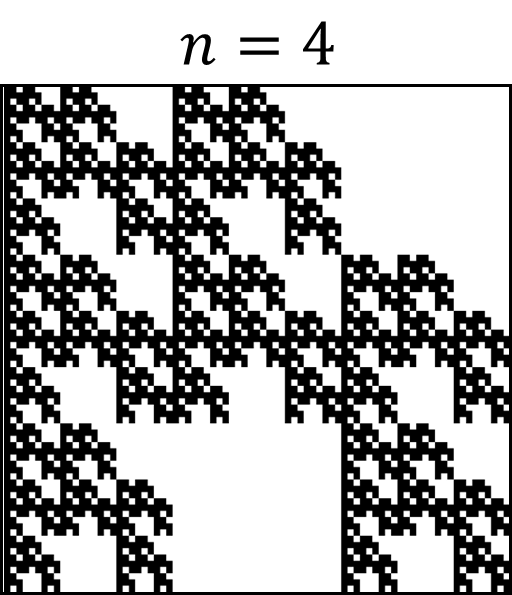}
\includegraphics[scale=0.3]{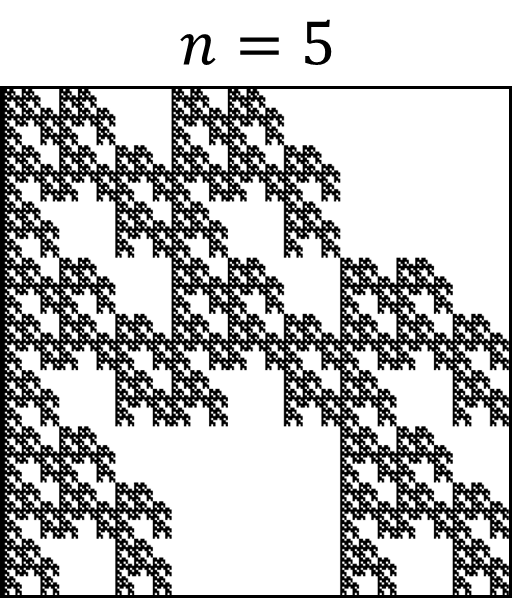}\\
\caption{The first 5 steps of the construction of a Sierpi\'nski Carpet modification.}
\label{fig:suscarpet}
\end{figure}

Applying the combinatorial methods, we would produce the following set of possible poles:
\begin{equation} \label{eq:susdimensionscombinatorial}
\mathscr{P}(A) \subseteq \left\{0 + \frac{2\pi i j}{\ln(3)}, 1 + \frac{2\pi i j}{\ln(3)}, \log_3(7) + \frac{2\pi i j}{\ln(3)} : j \in \mathbb{Z}\right\}.
\end{equation}
However, this does not match the results shown in \cite[Equation (10.190)]{mastersthesis}, which were the following:
\begin{align} \label{eq:susdimensionsmasters}
\mathscr{P}(A) &\subseteq \left\{0 + \frac{2\pi i j}{\ln(3)}^{[2]}, 0^{[3]}, \log_3(3-\sqrt{2}) + \frac{2\pi i j}{\ln(3)}, \log_3(2) + \frac{2\pi i j}{\ln(3)}^{[2]},\right.\nonumber\\&\left. 1, \log_3(3+\sqrt{2}) + \frac{2\pi i j}{\ln(3)}, \log_3(7) + \frac{2\pi i j}{\ln(3)} : j \in \mathbb{Z}\right\}.
\end{align}
where the ${}^{[n]}$ indicates that the pole is of order $n$ (all other poles listed are assumed to be simple). The fact that the combinatorial methods discussed cannot possibly produce poles of $\log_3(3 \pm \sqrt{2})$ is a strong indication that they will not be helpful in determining the set of possible complex dimensions for a Sierpi\'nski Carpet modification that is not of dust type. The author is actively researching this topic.
\newparagraph

\subsection{Bedford--McMullen Carpets} \label{chap:bedfordmcmullencarpets}
%Potentially could be done, need to adjust \ref{thm:scaling} to allow multiple scaling factors at the same time
Bedford--McMullen Carpets can be constructed by using the definition of a Sierpi\'nski Carpet modification and allowing to divide into a $p\times q$ grid for some $p,q \in \mathbb{N}$ where $p \neq q$. An example is given in Figure \ref{fig:bedfordmcmullencarpet}.

\begin{figure}[H]
\includegraphics[scale=0.3]{Images/NEquals0}
\includegraphics[scale=0.3]{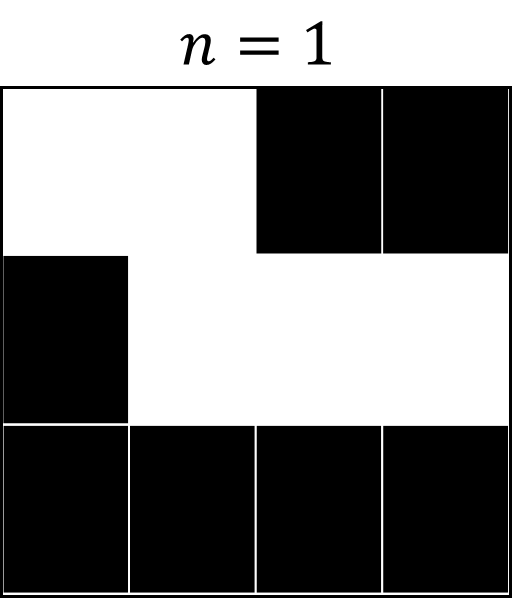}
\includegraphics[scale=0.3]{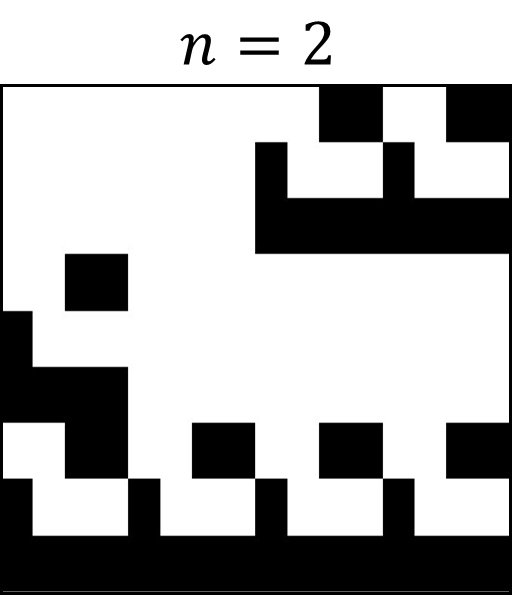}\\
\includegraphics[scale=0.3]{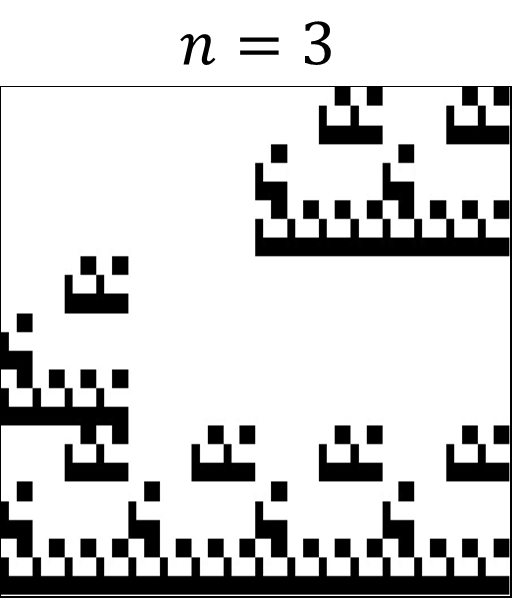}
\includegraphics[scale=0.3]{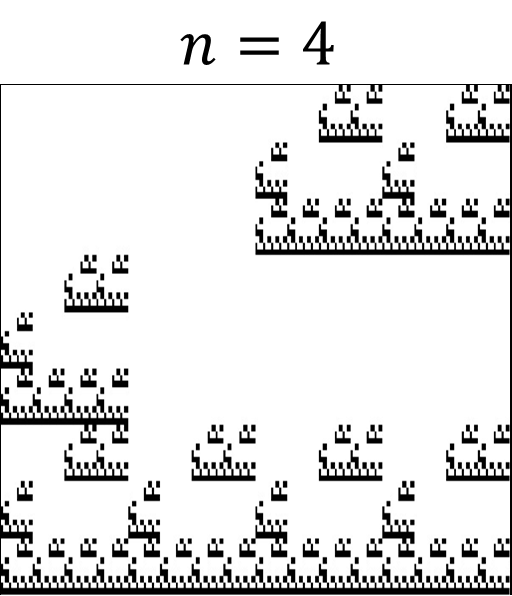}
\includegraphics[scale=0.3]{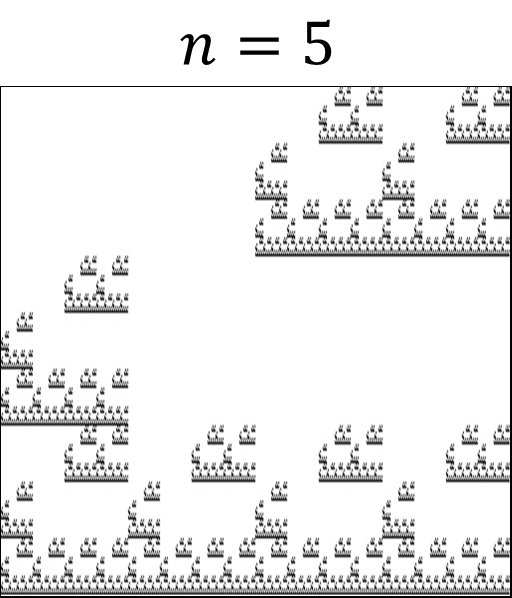}\\
\caption{The first 5 steps of the construction of a Bedford--McMullen Carpet in a $4\times3$ grid, which would be classified as ``dust type''.}
\label{fig:bedfordmcmullencarpet}
\end{figure}

These fractals are self-affine rather than self-similar, and they will have two simultaneous scaling factors: $\frac{1}{p}$ and $\frac{1}{q}$. At each finite step of the construction, we have rectangles instead of squares. The definition of ``dust type'' still makes sense in this new context, and the combinatorial methods discussed in Section \ref{chap:countingintersections} will still apply since the ``intersection types'' originate from the ways that squares (and rectangles) can share edges and corners in Euclidean $\mathbb{R}^2$ space. However, Theorems \ref{thm:scaling} and \ref{thm:infinitethresholdsexist} require the uniform scaling given by the construction of a Sierpi\'nski Carpet modification, and thus Theorems \ref{thm:alphaconnected}, \ref{thm:gammaconnected1}, and \ref{thm:gammaconnected2} will not immediately apply. In order to apply the results of this paper to dust type Bedford--McMullen Carpets, we will have to modify Theorems \ref{thm:scaling} and \ref{thm:infinitethresholdsexist} to allow for 2-directional non-uniform scaling. We plan to address this problem in future work.
\newparagraph

\subsection{Menger Sponge Modifications} \label{chap:mengersponges}
We can define analogues of Sierpi\'nski Carpet modifications in $\mathbb{R}^N$ for any $N \in \mathbb{N}$. We can then extend Theorem \ref{thm:alphaconnected} to this new $\mathbb{R}^N$ context (Theorem \ref{thm:alphaconnectedRN}). However, there are limitations to our current definition of ``dust type'' which hinder our ability to extend the rest of the paper's results to $\mathbb{R}^N$.
%Definitely can be done, just account for all the ways that cubes can share corners/edges/faces. Might need to adjust definition of "Dust Type"?
\begin{definition}\label{def:mengerspongemodification}
We construct a \textit{Menger Sponge modification} by the following algorithm:
\begin{enumerate}
    \item Begin with the unit $N$-cube $A_0 = [0,1] \times [0,1] \times\dots\times [0,1] \subseteq \mathbb{R}^N$ for $N \geq 3$.
    \item Dissect the $N$-cube into a uniform $p \times\dots\times p$ grid, producing nonoverlapping $N$-cubes of side length $\frac{1}{p}$.
    \item (Produce Level 1 of the fractal construction) Choose a collection of $m \in \mathbb{N}$ grid $N$-cubes to keep, removing all the others. We will call the result $A_1$. We require $m \in [1,p^N-1] \cap \mathbb{N}$ in order to ensure $A_1 \neq \emptyset$ and $A_1 \neq A_0$.
    \item (Produce Level $n$ of the fractal construction from Level $n-1$) For each $N$-cube that remains in $A_{n-1}$ (Level $n-1$ of the fractal construction), divide each of them into their own $p \times\dots\times p$ grids (making nonoverlapping $N$-cubes of side length $\frac{1}{p^n}$) and remove the corresponding collection of $N$-cubes from each grid. We will call the result $A_n$. \label{mengerrecursionstep}
    \item Repeat Step \ref{mengerrecursionstep} infinitely-many times to produce $\displaystyle A = \bigcap_{n=1}^{\infty} A_n$.
\end{enumerate}
\end{definition}

The theorems we proved in Sections \ref{chap:scalingtheorem} and \ref{chap:entireextensiontheorem} will apply to this situation as well. Furthermore, the definition of ``completely dusty'' (Definition \ref{def:alphaconnected}) and resulting Theorem \ref{thm:alphaconnected} extend nicely to Menger Sponge modifications, as demonstrated below.
\newparagraph

\begin{theorem}[Completely Dusty Theorem for $\mathbb{R}^N$] \label{thm:alphaconnectedRN}
Let $N \in \mathbb{N}$ and let $A \subseteq \mathbb{R}^N$ be a Menger Sponge modification constructed by keeping $m \geq 2$ $N$-cubes from the $p \times\dots\times p$ grid. Suppose $A$ has the property that there exists $\alpha > 0$ such that $A_{\alpha}$ is a disjoint union of $m$ copies of $\frac{1}{p} \left(A_{p\alpha}\right)$. Then $\tilde{\zeta}_A$ has a meromorphic continuation to all of $\mathbb{C}$ and
\begin{equation}
\mathscr{P}(A) \subseteq \left\{\log_p(m) + \frac{2\pi i j}{\ln(p)}: j \in \mathbb{Z}\right\}.
\end{equation}
\end{theorem}
\begin{proof}
First, note that since $A \subseteq [0,1]\times\dots\times[0,1] \subseteq \mathbb{R}^N$, we have that choosing $\delta > \sqrt{N}$ yields $A \subseteq B(x,\delta)$ for any $x \in A$, which makes $A_{\delta}$ simply-connected. Thus, the set $\mathscr{T} = \{t > 0 : A_t \text{ is a disjoint union of $m$ copies of $\frac{1}{p} \left(A_{pt}\right)$}\}$ is bounded above by $\delta$, and is non-empty in light of the assumption that $\alpha \in \mathscr{T}$. By completeness of $\mathbb{R}$, the set $\mathscr{T}$ has a supremum. Without loss of generality, we will assume $\alpha = \sup(\mathscr{T})$. \\

Since $A$ is a Menger Sponge modification constructed by keeping $m$ $N$-cubes from a $p\times\dots\times p$ grid, we have a self-similar fractal which can be written as follows:
\begin{equation}
A = \bigcup_{i=1}^m \left(\frac{1}{p} A + (x_{1,i},x_{2,i},\dots,x_{N,i})\right),
\end{equation}
where $(x_{1,i},\dots,x_{N,i})$ are the corners of the sub-cubes and we translate $\frac{1}{p} A$ in order to line up with the grid. Since this is a recursive definition, we have the following for each $n \in \mathbb{N}$:
\begin{equation}
A = \bigcup_{j=1}^{m^n} \left(\frac{1}{p^n} A + (x_{1,j},x_{2,j},\dots,x_{N,j})\right),
\end{equation}
where $(x_{1,j},\dots,x_{N,j})$ are the corners of the sub-cubes and we translate $\frac{1}{p^n} A$ in order to line up with the grid. Since the Euclidean metric is translation-invariant in $\mathbb{R}^N$, we have the following for each $t \in [0,\delta]$:
\begin{equation} \label{eq:selfsimilardefnd}
A_t = \bigcup_{j=1}^{m^n} \left(\left(\frac{1}{p^n} A\right)_t + (x_{1,j},x_{2,j},\dots,x_{N,j})\right).
\end{equation}
By Lemma \ref{thm:inflationlemma}, we obtain the following:
\begin{equation}
A_{\frac{\alpha}{p^n}} = \bigcup_{j=1}^{m^n} \left(\left(\frac{1}{p^n} A\right)_{\frac{\alpha}{p^n}} + (x_{1,s},\dots,x_{N,s})\right)
= \bigcup_{j=1}^{m^n} \left(\frac{1}{p^n} A_{\alpha} + (x_{1,s},\dots,x_{N,s})\right).
\end{equation}
Since $t = \alpha$ is the maximal number such that $A_t$ is a disjoint union of $m$ copies of $\frac{1}{p} \left(A_{pt}\right)$, we now have that $A_{\frac{\alpha}{p^n}}$ is a disjoint union of $m^n$ copies of $\frac{1}{p^n} A_{\alpha}$. \\

Now consider the function $\tilde{A}(t) : [\alpha, p\alpha] \rightarrow \mathbb{R}$ such that $\tilde{A}(t) = |A_t|$ whenever $t \in [\alpha,p\alpha]$. Then we have that whenever $t \in \left[\frac{\alpha}{p^k}, \frac{\alpha}{p^{k-1}}\right]$ for some $k \in \mathbb{N}$, we obtain the following:
\begin{equation}
|A_t| = m \left|\left(\frac{1}{p} A\right)_{p t}\right| =\dots= m^k \left|\left(\frac{1}{p^k} A\right)_{p^k t}\right|.
\end{equation}
By Corollary \ref{thm:anydelta}, we can assume without loss of generality that $\alpha < \delta \leq p\alpha$. Plugging this into the tubular zeta function formula and using Theorem \ref{thm:scaling} yields:
\begin{align}
\tilde{\zeta}_A(s)
&= \int_{\alpha}^{\delta} t^{s-N-1} |A_t| \,\mathrm{d}t + \sum_{k=1}^{\infty} \int_{\frac{\alpha}{p^k}}^{\frac{\alpha}{p^{k-1}}} t^{s-N-1} |A_t| \,\mathrm{d}t \\
&= \int_{\alpha}^{\delta} t^{s-N-1} \tilde{A}(t) \,\mathrm{d}t + \sum_{k=1}^{\infty} \int_{\frac{\alpha}{p^k}}^{\frac{\alpha}{p^{k-1}}} t^{s-N-1} m^k \left|\left(\frac{1}{p^k} A\right)_{p^k t}\right| \,\mathrm{d}t\\
&= \int_{\alpha}^{\delta} t^{s-N-1} \tilde{A}(t) \,\mathrm{d}t + \sum_{k=1}^{\infty} \int_0^{\frac{\alpha}{p^{k-1}}} t^{s-N-1} m^k \left|\left(\frac{1}{p^k} A\right)_{p^k t}\right| \,\mathrm{d}t \nonumber\\&- \int_0^{\frac{\alpha}{p^k}} t^{s-N-1} m^k \left|\left(\frac{1}{p^k} A\right)_{p^k t}\right| \,\mathrm{d}t \\
&= \int_{\alpha}^{\delta} t^{s-N-1} \tilde{A}(t) \,\mathrm{d}t + \sum_{k=1}^{\infty} \frac{1}{p^{k s}} \int_0^{p\alpha} t^{s-N-1} m^k \left|A_t\right| \,\mathrm{d}t \nonumber\\&- \frac{1}{p^{k s}} \int_0^{\alpha} t^{s-N-1} m^k \left|A_t\right| \,\mathrm{d}t\\
&= \int_{\alpha}^{\delta} t^{s-N-1} \tilde{A}(t) \,\mathrm{d}t + \sum_{k=1}^{\infty} \frac{m^k}{p^{k s}} \int_{\alpha}^{p\alpha} t^{s-N-1} \tilde{A}(t) \,\mathrm{d}t \\
&= \int_{\alpha}^{\delta} t^{s-N-1} \tilde{A}(t) \,\mathrm{d}t + \frac{m}{p^s - m} \int_{\alpha}^{p\alpha} t^{s-N-1} \tilde{A}(t) \,\mathrm{d}t.
\end{align}
Since $\tilde{A}(t)$ is monotone and continuous due to it being a domain restriction of $|A_t|$, and since $\alpha > 0$, we have by Theorem \ref{thm:entireextension} that $\displaystyle \int_{\alpha}^{\delta} t^{s-N-1} \tilde{A}(t) \,\mathrm{d}t$ and $\displaystyle \int_{\alpha}^{p\alpha} t^{s-N-1} \tilde{A}(t) \,\mathrm{d}t$ extend to entire functions. Thus, the only possible poles of $\tilde{\zeta}_A(s)$ will be solutions to $p^s - m = 0$, which yields simple poles for every $s = \log_p(m) + \frac{2\pi i j}{\ln(p)}$ for all $j \in \mathbb{Z}$. Therefore,
\begin{equation}
\mathscr{P}(A) \subseteq \left\{\log_p(m) + \frac{2\pi i j}{\ln(p)}: j \in \mathbb{Z}\right\}.
\end{equation}
\end{proof}

Note that for the Ternary Cantor Set $\mathfrak{C} \subseteq [0,1]$, the set $A = \mathfrak{C}\times\dots\times\mathfrak{C} \subseteq \mathbb{R}^N$ will meet the criteria for Theorem \ref{thm:alphaconnectedRN} for any $N \in \mathbb{N}$, since $\alpha = \frac{1}{6}$ is sufficient to make $A_{\alpha}$ equal $2^N$ disjoint copies of $A_{3\alpha}$. In such a case, the complex dimensions will be computed to be
\begin{equation}
\mathscr{P}(A) \subseteq \left\{\log_3\left(2^N\right) + \frac{2\pi i j}{\ln(3)}: j \in \mathbb{Z}\right\}.
\end{equation}
The traditional Menger Sponge will not meet the criteria for Theorem \ref{thm:alphaconnectedRN}, so its complex dimensions will need to be computed in a different way. To generalize beyond the $\mathbb{R}^N$ analogue of the ``completely dusty'' case, the combinatorics described in Section \ref{chap:computingdust} will need to be modified in order to accommodate additional intersection types created by $N$-cubes geometrically sharing corners, edges, faces, and $k$-dimensional face analogues for $3 < k < N$. However, we face a more fundamental problem with the current definition of ``dust type'', as it means that some Sierpi\'nski Carpet modifications would change their classifications depending on which $\mathbb{R}^N$ they are embedded into. An example of this phenomenon is shown in Figure \ref{fig:tetrislcarpet}.
\begin{figure}[H]
\includegraphics[scale=0.3]{Images/NEquals0}
\includegraphics[scale=0.3]{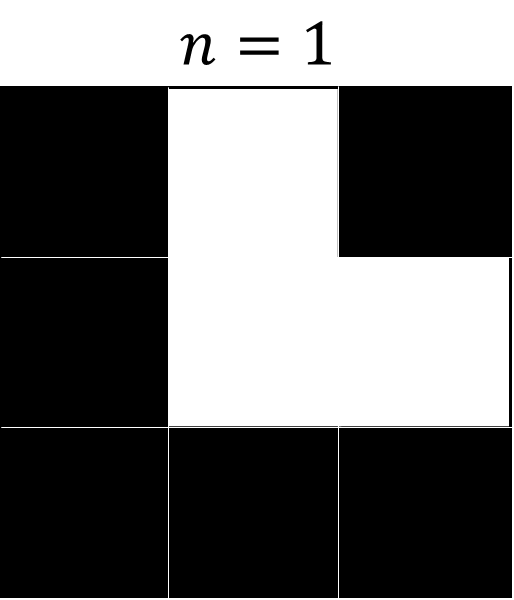}
\includegraphics[scale=0.3]{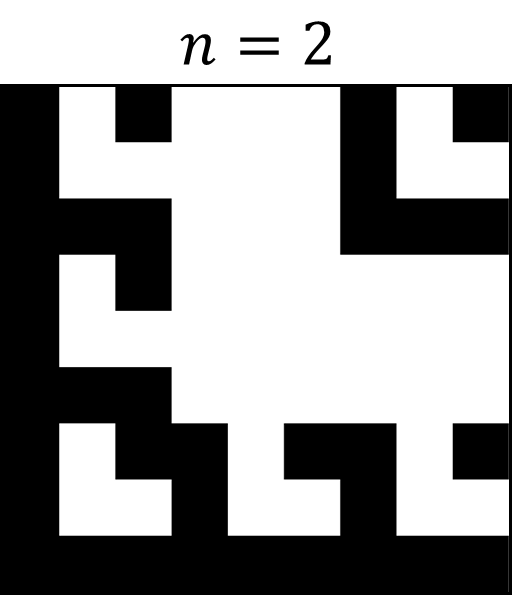}\\
\includegraphics[scale=0.3]{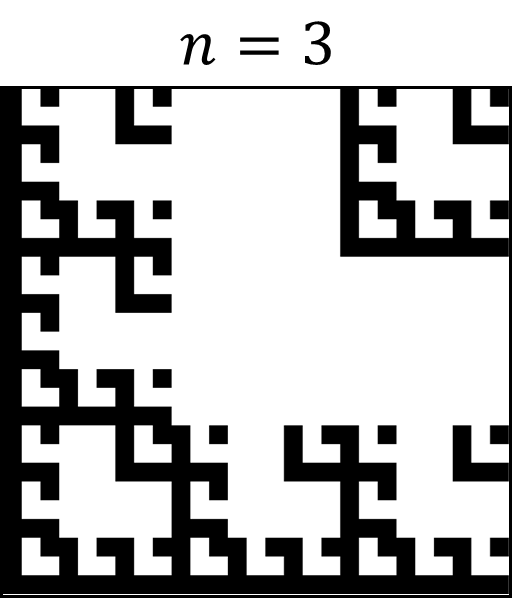}
\includegraphics[scale=0.3]{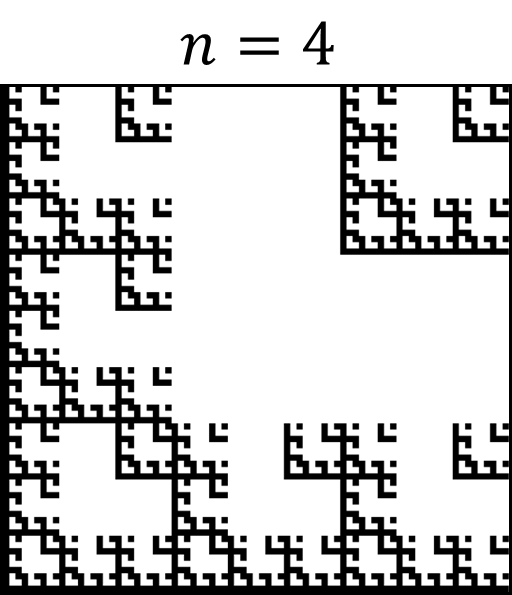}
\includegraphics[scale=0.3]{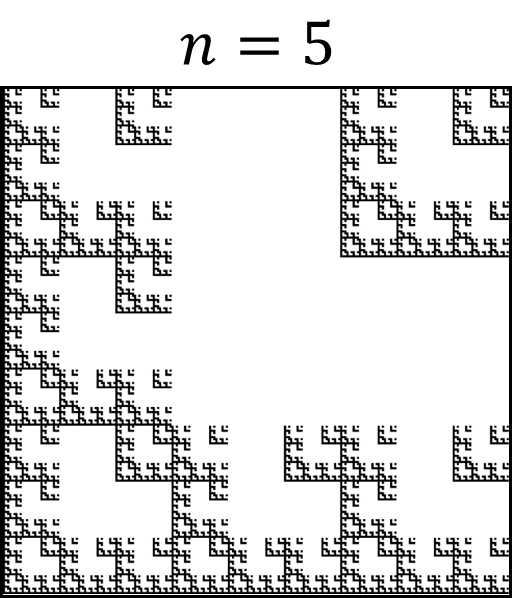}\\
\caption{The first 5 steps of the construction of a Sierpi\'nski Carpet modification. Embedded in $\mathbb{R}^2$, it is not dust type since the complement is not path-connected. Embedded in $\mathbb{R}^3$, it is dust type since the complement would be path-connected.}
\label{fig:tetrislcarpet}
\end{figure}

Since \cite[Theorem 4.7.3]{lapidusbook} states that the poles of $\tilde{\zeta}_A$ are invariant under choice of $\mathbb{R}^N$ to embed the set $A$ into (as long as $N \geq \dim(A)$), we believe that the classification of $A$ as a ``dust type'' should also be invariant. We plan to address this issue in later work.

%The real issue: find Theorem in FZF book that states that the dimensions don't change if you change the R^N you embed into
%Dimensions of Tetris L were found to be not the ones you get from Dust Theorems (maybe include that example instead of Sus Carpet?)

%We noticed in the $N=2$ case that the corner-sharing intersections (sourced from a dimension-0 ``face'') are in a sense ``independent'' of the edge-sharing intersections (sourced from a dimension-1 ``face''). We see this in the table of results - the poles generated by the $D$ intersections add on to the poles generated by the $H$ and $V$ intersections (producing poles of order 2 if either the $H$ or $V$ multipliers equal 1), whereas the poles generated by the $H$ and $V$ intersections produce simple poles according to their multipliers but we do not get poles of order 2 if the multipliers happen to be the same. \\

%We conjecture that in dust type Menger Sponge modifications, intersections which source from different-dimension face analogues will have their orders add together if the multipliers coincide. On the other hand, the intersections which source from the same-dimension face analogues will behave like the $H$ and $V$ intersection sequences: producing additional poles if their multipliers are different, but the orders do not add together if any of their multipliers are the same.

%------
% Insert acknowledgments and information
% regarding funding at the end of the last
% section, i.e., right before the bibliography.
%------

\begin{ack}
This paper is a continuation of work started in \cite{summerresearch} and \cite{mastersthesis}, advised by Dr. Sean Watson and Dr. Erin Pearse respectively. We would like to thank the Bill and Linda Frost Fund for their generous support of this work. We thank the anonymous referees for taking the time to read the paper and give helpful feedback.
\end{ack}

%\begin{funding}
%We would like to thank the Bill and Linda Frost Fund for their generous support of this work in the Summer of 2022 and 2020.
%\end{funding}

%------
% Insert the bibliography.
%------

\end{document}